\documentclass[11pt,oneside]{report}
\usepackage[utf8]{inputenc}
\usepackage[T1]{fontenc}
\usepackage[english]{babel}

\usepackage[all]{xy}
\usepackage{amsmath,amssymb,amsfonts,amsthm}
\usepackage{mathtools}

\usepackage{graphicx}
\usepackage{tikz-cd}

\usepackage{hyperref}
\usepackage{cleveref}

\usepackage{geometry}
\usepackage{microtype}
\usepackage{hyphenat}

\usepackage{enumitem}
\usepackage{csquotes}

\usepackage{makeidx}
\makeindex

\usepackage{bussproofs}  
\usepackage{tikz}       

\theoremstyle{plain} 
\newtheorem{thm}{Theorem}[section] 

\theoremstyle{definition} 
\newtheorem{defn}[thm]{Definition} 

\theoremstyle{remark} 
\newtheorem*{remark}{Remark} 

\title{A Categorical Integration of Logical Connectives via Higher Category Theory}
\author{Joaquim Reizi Barreto}
\date{\today}

\begin{document}

\maketitle
\tableofcontents

\chapter{Introduction}

\section{Research Background and Motivation}
\subsection{The Role of Logical Connectives in Categorical Semantics}
\label{subsec:logical_connectives}

In categorical semantics, \index{logical connective} logical connectives such as negation, conjunction, disjunction, and implication play a crucial role in bridging syntactic expressions with their semantic interpretations. This section provides an overview of the semantic significance of these connectives within a categorical framework.

\paragraph{Negation:}
The notion of negation in a categorical setting often corresponds to the existence of an adjunction or a dualizing object. Specifically, the \index{negation} negation can be modeled using the duality functor in a \index{closed category} closed category, where the complement of an object is understood in terms of its relationship with a fixed dualizing object.

\paragraph{Conjunction and Disjunction:}
Conjunction and disjunction are generally modeled by \index{product} products and \index{coproduct} coproducts respectively. The universal properties of products and coproducts ensure that they capture the essence of logical conjunction and disjunction. In particular, the product in a category abstracts the notion of conjunction by providing a unique arrow that factors through projections, while the coproduct abstracts disjunction via injections.

\paragraph{Implication:}
Implication in categorical semantics is typically interpreted through the notion of an exponential object in a \index{cartesian closed category} cartesian closed category. The exponential object $B^A$ represents the space of morphisms from $A$ to $B$, thereby capturing the idea of a functional relationship or logical implication.

The interplay between these logical connectives and categorical structures not only provides a robust semantic framework but also reveals deep connections between logic and category theory. In subsequent sections, we will formalize these ideas by constructing local categories corresponding to each logical connective and exploring their integration into a unified categorical semantics.

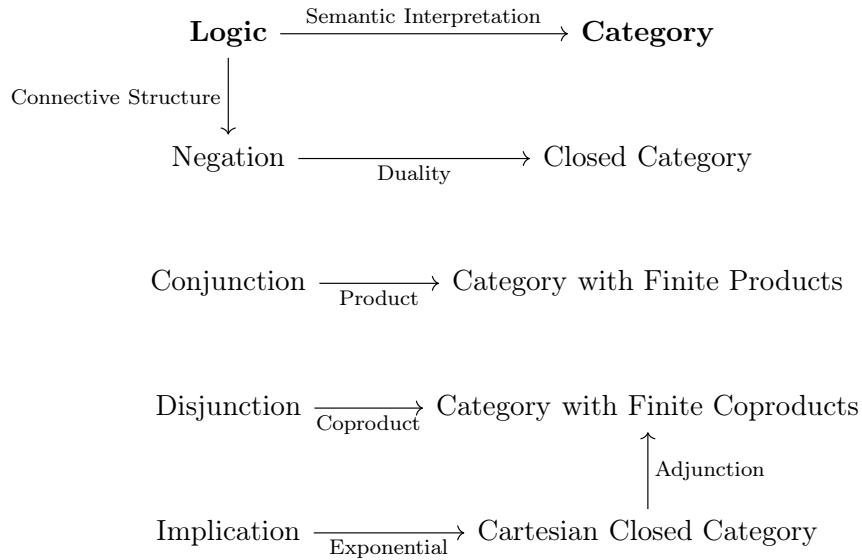
\begin{center}
\begin{tikzcd}[column sep=large, row sep=large]
\textbf{Logic} \arrow[r, "\text{Semantic Interpretation}"] \arrow[d, "\text{Connective Structure}"'] & \textbf{Category} \\
\text{Negation} \arrow[r, "\text{Duality}"'] & \text{Closed Category} \\
\text{Conjunction} \arrow[r, "\text{Product}"'] & \text{Category with Finite Products} \\
\text{Disjunction} \arrow[r, "\text{Coproduct}"'] & \text{Category with Finite Coproducts} \\
\text{Implication} \arrow[r, "\text{Exponential}"'] & \text{Cartesian Closed Category} \arrow[u, "\text{Adjunction}"']
\end{tikzcd}
\end{center}

\subsection{Limitations of Traditional 1-Category Approaches}

Traditional 1-category approaches, while robust in many aspects, exhibit significant limitations in addressing universality and coherence issues. In a 1-category, universal properties are defined by strict commutativity conditions, which restrict the flexibility needed for modeling complex structures where only up-to-isomorphism conditions naturally hold.

\paragraph{Issues with Universality:}
In a 1-category setting, the universal property is enforced by exact commutative diagrams. This rigidity can be problematic when the construction inherently requires a more relaxed, "weak" form of universality, where the uniqueness of mediating morphisms is defined only up to isomorphism.

\paragraph{Coherence Problems:}
Moreover, coherence problems arise due to the absence of higher morphisms. Without 2-morphisms, there is no natural mechanism to express associativity and unit laws up to coherent isomorphism, leading to potential inconsistencies when composing morphisms in complex diagrams.

The limitations discussed here motivate the extension to higher categorical frameworks, such as 2-categories, where these issues are alleviated by replacing strict equalities with natural isomorphisms.

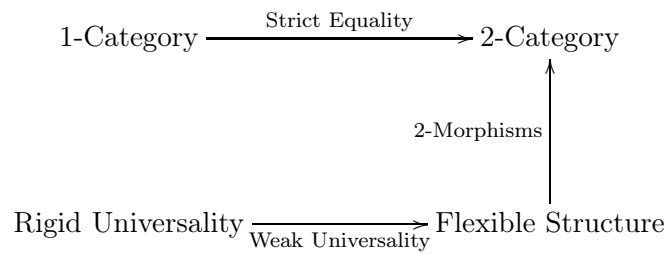
\begin{center}
\[
\xymatrix@C=6em@R=5em{
{\text{1-Category}} 
  \ar[r]^-{\text{Strict Equality}} & 
  {\text{2-Category}} \\
{\text{Rigid Universality}} 
  \ar[r]_-{\text{Weak Universality}} & 
  {\text{Flexible Structure}} 
  \ar[u]^-{\text{2-Morphisms}}
}
\]
\end{center}

\subsection{Emergence of Higher Category Theory}
\label{subsec:emergence_higher_cat}

The emergence of \index{higher category theory} higher category theory, particularly the concepts of 2-categories and \index{bicategory} bicategories, is a direct response to the limitations observed in traditional 1-category frameworks. Traditional 1-categories enforce strict commutativity, which often fails to capture the flexible nature of structures encountered in advanced mathematical contexts such as homotopy theory and algebraic topology.

\paragraph{Historical Context:}  
The development of 2-categories and bicategories was largely motivated by the need to model phenomena where morphisms between morphisms (2-morphisms) are essential. This additional layer allows for the relaxation of strict associativity and identity laws, replacing them with coherent isomorphisms. Consequently, higher category theory offers a framework where universal properties can be satisfied up to isomorphism rather than equality.

\paragraph{Enhanced Flexibility:}  
The introduction of 2-morphisms provides the necessary flexibility to handle complex constructions by allowing multiple compositional pathways to be coherently related. This not only preserves critical structural information but also facilitates more intricate integrations of logical and algebraic systems. The shift from rigid to flexible structures underpins many modern developments in categorical semantics and related fields.

\begin{center}
\[
\xymatrix@C=6em@R=5em{
{\text{1-Category}} 
  \ar[r]^-{\text{Strict Equality}} & 
  {\text{2-Category / Bicategory}} \\
{\text{Rigid Structures}} 
  \ar[r]_-{\text{Coherent Isomorphisms}} & 
  {\text{Flexible Structures}} 
  \ar[u]^-{\text{2-Morphisms}}
}
\]
\end{center}

\newpage
\section{Problem Setting and Objectives}
\subsection{Integration of Local Categories for Logical Connectives}
\label{subsec:integration_local_cat}

In categorical semantics, each \index{local category} local category is constructed to capture the universal properties associated with a specific logical connective, such as negation, conjunction, disjunction, or implication. However, integrating these disparate local categories into a unified categorical framework presents several challenges that must be addressed to preserve both the semantic integrity and the flexibility of the overall structure.

The primary challenges in this integration process include:

\paragraph{Heterogeneity of Structures:}
Each local category is tailored to model a particular logical connective with its own universal property and compositional structure. When attempting to integrate these into a global category, one must reconcile differences in the ambient categorical frameworks, such as differing notions of limits, colimits, or exponentials. This often necessitates moving to a higher categorical setting, such as a 2-category, where the presence of \index{2-morphism} 2-morphisms allows for the expression of equivalences up to natural isomorphism.

\paragraph{Coherence and Compatibility:}
The integration must ensure that the universal constructions in the individual local categories remain coherent when combined. In other words, the transition from local to global should respect the natural isomorphisms that arise in each case, leading to the construction of \index{pseudo-limit} pseudo-limits or pseudo-colimits. These constructions are designed to preserve the essential structural properties while allowing for the necessary flexibility in the compositional rules.

\paragraph{Adjunctions and Dualities:}
Local categories often involve adjunctions or dualizing objects to model negation and implication. Integrating these requires ensuring that the adjunctions in different local settings interact in a compatible manner within the global category. This compatibility is crucial for maintaining the overall semantic consistency, particularly when transferring properties from one logical connective to another.

The following diagram illustrates the conceptual flow from local categories, each corresponding to a specific logical connective, to the integrated global category. The use of 2-morphisms (or natural isomorphisms) is essential in managing the coherence between different compositional paths.

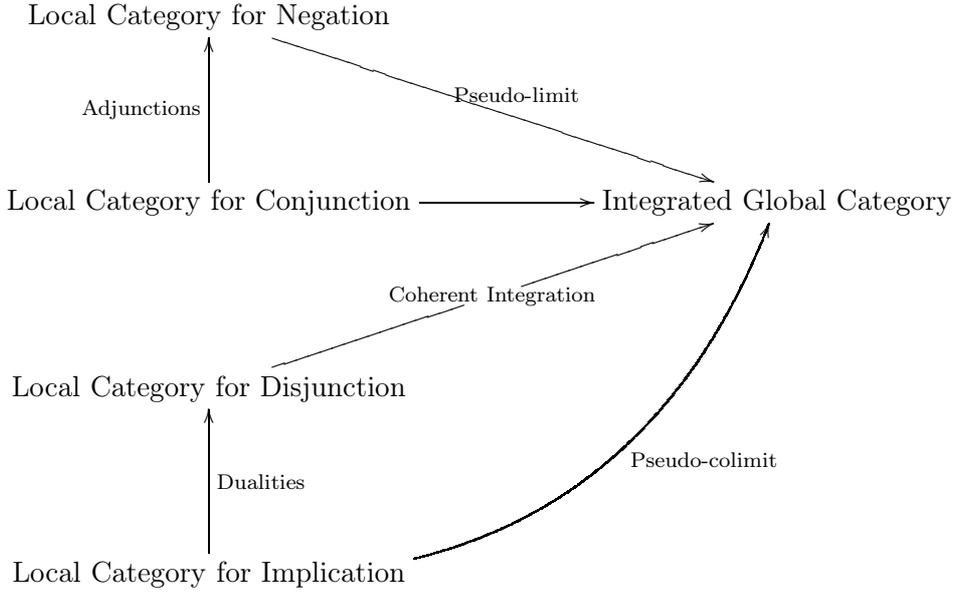
\begin{center}
\[
\xymatrix@C=6em@R=5em{
{\text{Local Category for Negation}} 
  \ar[dr]^-{\text{Pseudo-limit}} & {} \\
{\text{Local Category for Conjunction}} 
  \ar[r]_-{\quad \quad \quad \quad} 
  \ar[u]^-{\text{Adjunctions}} & 
  {\text{Integrated Global Category}} \\
{\text{Local Category for Disjunction}} 
  \ar[ur]|-{\text{Coherent Integration}} & {} \\
{\text{Local Category for Implication}} 
  \ar[u]_-{\text{Dualities}} 
  \ar@/_4pc/[uur]_-{\text{Pseudo-colimit}} & {}
}
\]

\end{center}

\subsection{Ensuring Universality and Coherence}
\label{subsec:ensuring_universality_coherence}

Ensuring universality and coherence in the integrated category is a central challenge in our construction. In this setting, the goal is to guarantee that the universal properties inherited from each local category are preserved after integration and that all possible compositional pathways yield results that are naturally isomorphic.

\paragraph{Universality:}
To formalize the universality condition, we require that for every object in the integrated category, there exists a unique (up to a canonical isomorphism) mediating morphism satisfying the relevant factorization properties. This is captured by the following definition.

\begin{defn}
\label{def:universal_property}
Let \( F : \mathcal{L} \to \mathcal{G} \) be the integration functor from a local category \(\mathcal{L}\) (associated with a specific logical connective) into the integrated global category \(\mathcal{G}\). Then, for every object \(X\) in \(\mathcal{G}\), there exists a unique (up to isomorphism) morphism 
\[
u: F(L) \to X
\]
such that for any other morphism \(v: F(L) \to Y\) satisfying the corresponding factorization conditions, there is a unique isomorphism \(\phi: X \to Y\) making the following diagram commute.
\end{defn}

The above definition encapsulates the idea that the universal constructions defined in each local category lift coherently to the global level, albeit up to natural isomorphism.

\paragraph{Coherence:}
Coherence ensures that all different compositional pathways, which arise when integrating the various local categories, are equivalent up to a specified natural isomorphism. In other words, if multiple compositional orders exist for a given set of morphisms, then the corresponding diagrams must commute in the sense that there is a canonical \index{coherence isomorphism} coherence isomorphism linking them.

Formally, let \( A \), \( B \), and \( C \) be objects in \(\mathcal{G}\) with morphisms \( f: A \to B \) and \( g: B \to C \). Coherence requires that all composites, such as \(g\circ f\) and any alternative composite \(f\star g\) obtained via the integration process, are related by a natural isomorphism:
\[
\alpha_{f,g} : (g\circ f) \xrightarrow{\sim} f\star g.
\]
This condition guarantees that the integrated structure behaves in a manner consistent with the original universal properties.

An illustrative diagram of the coherence condition is given by
\[
\begin{tikzcd}[column sep=large, row sep=large]
& A \arrow[dr, "f"] \arrow[dl, swap, "h"] & \\
B \arrow[rr, swap, "g"] & & C
\end{tikzcd}
\]
where the two distinct composites (through \(h\) and directly via \(f\)) are connected by a natural isomorphism, ensuring overall consistency.

Finally, the integration process employs a strictification procedure to convert these weak (up-to-isomorphism) structures into an equivalent strict 2-category, thereby preserving both universality and coherence in a rigorously defined manner.

\subsection{The Need for Strictification Techniques}
\label{subsec:need_strictification_techniques}

In the study of \index{bicategory}bicategories and \index{2-category}2-categories, certain coherence diagrams are only required to commute up to specified \index{isomorphism}isomorphisms rather than strictly. This leads to a degree of flexibility but can complicate proofs and constructions. \index{strictification}Strictification techniques offer a systematic way to transform these “weak” structures into \index{strict 2-category}strict 2-categories, in which associativity and unit laws hold on the nose.

\begin{thm}[Strictification Theorem]
\label{thm:strictification}
Every \index{bicategory}bicategory is equivalent (as a bicategory) to a \index{strict 2-category}strict 2-category. More precisely, for any bicategory \(\mathcal{B}\), there exists a strict 2-category \(\mathcal{B}^{\mathrm{str}}\) and a pair of \index{pseudofunctor}pseudofunctors
\[
F: \mathcal{B} \longrightarrow \mathcal{B}^{\mathrm{str}}, 
\quad
G: \mathcal{B}^{\mathrm{str}} \longrightarrow \mathcal{B}
\]
such that \(G \circ F\) and \(F \circ G\) are each connected to the respective identity functors by \index{pseudonatural transformation}pseudonatural transformations that satisfy the coherence conditions for an equivalence of bicategories.
\end{thm}

This theorem states that while bicategories allow weakened versions of associativity and identity laws, these “weak” laws can be replaced by strict ones without losing any essential structure. In essence, \(\mathcal{B}^{\mathrm{str}}\) retains the same “shape” of \(\mathcal{B}\) up to coherent isomorphisms, but organizes it in a strictly associative and unital manner.

Intuitively, the proof leverages the idea that all coherence isomorphisms in a bicategory can be chosen in a canonical, systematic way, effectively “rigidifying” the weak aspects. These choices ensure that every associative or unit diagram commutes strictly in the new strict 2-category.

\begin{proof}
\renewcommand{\qedsymbol}{\(\blacksquare\)}
\noindent
\textbf{Constructive Proof Outline:}
\begin{enumerate}
  \item \textbf{Free Construction:}
  Begin by constructing a free 2-category on the underlying 1-category of \(\mathcal{B}\). This involves freely adding 2-morphisms corresponding to the bicategory’s 2-morphisms but without imposing any relations other than those strictly needed.
  \item \textbf{Adjoin Coherence Isomorphisms:}
  Next, incorporate the coherence isomorphisms from \(\mathcal{B}\) as 2-morphisms in the free construction, ensuring they satisfy strict associativity and identity laws in the resulting structure.
  \item \textbf{Identify Redundant 2-Morphisms:}
  Factor out any redundant relations by a quotient that identifies distinct 2-morphisms whenever they correspond to the same composition in \(\mathcal{B}\). This step ensures that all coherent compositions are identified appropriately.
  \item \textbf{Establish Equivalence:}
  Define pseudofunctors \(F\) and \(G\) between \(\mathcal{B}\) and the new strict 2-category \(\mathcal{B}^{\mathrm{str}}\). Construct pseudonatural transformations witnessing the required equivalence, verifying they satisfy all coherence conditions.
\end{enumerate}
Thus, we obtain \(\mathcal{B}^{\mathrm{str}}\) and the equivalence as claimed, completing the strictification procedure.
\end{proof}

\[
\begin{tikzcd}[column sep=large, row sep=large]
\mathcal{B} \arrow[r, bend left=25, "F"] \arrow[r, bend right=25, swap, ""]
& \mathcal{B}^{\mathrm{str}} \arrow[l, bend left=25, "G"] \arrow[l, bend right=25, swap, ""]
\end{tikzcd}
\]
Here, the unlabeled arrows represent the pseudonatural transformations that exhibit \(G \circ F\) and \(F \circ G\) as equivalences up to coherent isomorphism.

\newpage
\section{Overview of the Proposed Approach}

\subsection{Local Category Construction}
\label{subsec:local_category_construction}

In this section, we briefly outline the construction of \index{local category}local categories corresponding to each logical connective—namely, \index{negation}negation, \index{conjunction}conjunction, \index{disjunction}disjunction, and \index{implication}implication. Each local category is designed to encapsulate the unique universal property of its associated connective within a categorical framework.

\paragraph{Negation:}
The local category for negation is constructed by introducing a dualizing object. For any object \( A \) in the base category, its negation \(\lnot A\) is defined as the unique (up to isomorphism) object that satisfies the following universal property: there exists a canonical morphism from \( A \) to a fixed dualizing object \( D \) such that for any other object \( X \) with a morphism from \( A \) to \( D \), there is a unique mediating morphism \( \phi: \lnot A \to X \) making the corresponding diagram commute. This setup models logical negation through a contravariant functor preserving the adjunction structure.

\paragraph{Conjunction and Disjunction:}
For conjunction, the local category is structured around the universal property of \index{product}products. Given objects \( A \) and \( B \), their product \( A \times B \) is characterized by the existence of projection morphisms \( \pi_A: A \times B \to A \) and \( \pi_B: A \times B \to B \) such that for any object \( X \) with morphisms \( f: X \to A \) and \( g: X \to B \), there exists a unique mediating morphism \( \langle f, g \rangle: X \to A \times B \) satisfying \( \pi_A \circ \langle f, g \rangle = f \) and \( \pi_B \circ \langle f, g \rangle = g \). Dually, disjunction is modeled via \index{coproduct}coproducts: for objects \( A \) and \( B \), the coproduct \( A + B \) is defined by injection morphisms and a corresponding universal property that uniquely factors any pair of morphisms from \( A \) and \( B \) into a common codomain.

\paragraph{Implication:}
The local category for implication is developed in the context of a \index{cartesian closed category}cartesian closed category. Here, the exponential object \( B^A \) represents the implication from \( A \) to \( B \). The universal property of the exponential asserts that for every morphism \( f: C \times A \to B \), there is a unique morphism \( \tilde{f}: C \to B^A \) such that \( f \) factors through the evaluation map \( \mathrm{ev}: B^A \times A \to B \), i.e.,
\[
f = \mathrm{ev} \circ (\tilde{f} \times \mathrm{id}_A).
\]

\noindent
Collectively, these constructions form the building blocks for the global categorical framework. In later sections, we will describe how these local categories are integrated, ensuring that the universal properties associated with each logical connective are preserved in a coherent and structured manner.

\subsection{Extension to a 2-Category Framework}
\label{subsec:extension_2cat_framework}

To capture the full flexibility of the logical constructions, the local categories are extended to a 2-category framework. In this extension, not only are objects and morphisms considered, but also \index{2-morphism}2-morphisms, which represent natural isomorphisms between functors and mediate the coherence between different compositions.

In the 2-category extension, each local category is enriched as follows:
\begin{itemize}
  \item The objects remain as in the original local category.
  \item The morphisms are as previously defined, preserving the universal properties of the logical connectives.
  \item For any two parallel morphisms \( f, g: A \to B \), a 2-morphism \( \alpha: f \Rightarrow g \) is introduced to express the natural isomorphism between different ways of factorizing or composing the morphisms.
\end{itemize}

This additional layer of 2-morphisms allows us to relax the strict equality conditions to natural isomorphisms, thereby handling coherence issues more flexibly. For instance, given two compositional paths in the integrated category, the existence of a 2-morphism ensures that the diagrams commute up to a specified isomorphism. This is essential when integrating multiple local categories that may have non-strict interactions.

The following diagram schematically illustrates the idea:
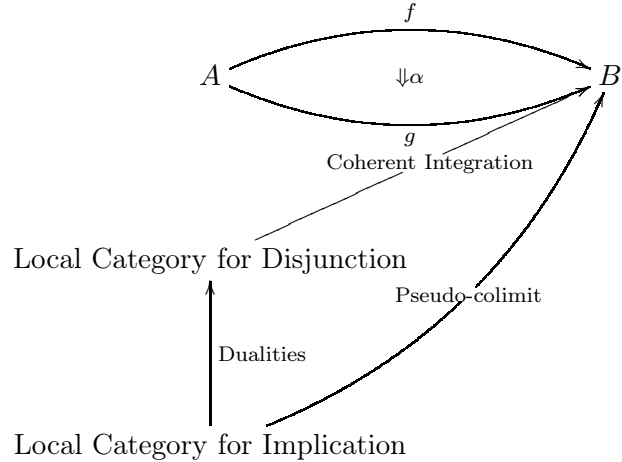
\[
\xymatrix@C=6em@R=5em{
A \ar@/^1.5pc/[r]^{f} \ar@/_1.5pc/[r]_{g} \ar@{}[r]|{\Downarrow \alpha} & B \\
{\text{Local Category for Disjunction}} 
  \ar[ur]|-{\text{\ \ Coherent Integration}} & {} \\
{\text{Local Category for Implication}} 
  \ar[u]_-{\text{Dualities}} 
  \ar@/_2pc/[uur]|-{\text{Pseudo-colimit}} & {}
}
\]
Here, the 2-morphism \(\alpha: f \Rightarrow g\) indicates that the two morphisms \(f\) and \(g\) are naturally isomorphic, providing the necessary coherence for further compositions.

Overall, the extension to a 2-category framework plays a crucial role in ensuring that the integrated categorical structure accurately reflects the flexible semantics of logical connectives while maintaining rigorous coherence.

\subsection{2-Category Composition and Coherence Verification}
\label{subsec:2cat_composition_coherence}

In the 2-category framework, the composition of morphisms is governed by both horizontal and vertical compositions, and coherence conditions ensure that all the associativity and unit laws hold up to specified natural isomorphisms. In other words, while the composition is not strictly associative or unital, the existence of 2-morphisms (or natural isomorphisms) guarantees that different compositional pathways are coherently equivalent.

\paragraph{2-Category Composition:}
Given objects \( A \), \( B \), and \( C \) in the 2-category, consider two composable morphisms:
\[
f: A \to B \quad \text{and} \quad g: B \to C.
\]
Their horizontal composite \( g \circ f \) is defined up to a natural isomorphism. Furthermore, if there are two 2-morphisms
\[
\alpha: f \Rightarrow f' \quad \text{and} \quad \beta: g \Rightarrow g',
\]
the horizontal composition of these 2-morphisms, denoted \( \beta \ast \alpha \), provides a coherent isomorphism between \( g \circ f \) and \( g' \circ f' \).

\paragraph{Coherence Verification:}
To verify coherence, one must show that all possible compositions of 2-morphisms (arising from different association orders) are equivalent. For instance, consider three composable morphisms:
\[
f: A \to B, \quad g: B \to C, \quad h: C \to D.
\]
The associativity coherence condition demands the existence of a canonical 2-morphism 
\[
a_{f,g,h}: (h \circ g) \circ f \xRightarrow{\sim} h \circ (g \circ f)
\]
such that for any quadruple of composable morphisms, the corresponding coherence diagrams (such as the pentagon diagram) commute.

A schematic diagram that illustrates the coherence of two compositional paths is given by:
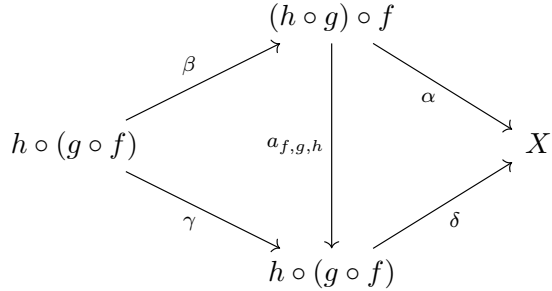
\[
\begin{tikzcd}[column sep=large, row sep=large]
& (h\circ g)\circ f \arrow[dr, "\alpha"'] \arrow[dd, "\scriptstyle a_{f,g,h}"'] & \\
h\circ (g\circ f) \arrow[ur, "\beta"] \arrow[dr, "\gamma"'] & & X \\
& h\circ (g\circ f) \arrow[ur, "\delta"'] &
\end{tikzcd}
\]
In this diagram, the 2-morphisms \(\alpha\), \(\beta\), \(\gamma\), and \(\delta\) serve to relate the two different composites, and coherence is verified by showing that the composite 2-morphism along the top equals that along the bottom.

\noindent
Overall, the 2-category composition and its coherence verification provide a robust framework for integrating local categories while ensuring that all structural isomorphisms maintain the desired universal properties.

\subsection{Strictification and Evaluation}
\label{subsec:strictification_evaluation}

This section addresses two interrelated aspects. First, we apply the strictification theorem to transform the integrated bicategory into a strict 2-category. Second, we evaluate the effectiveness of the constructed framework by examining concrete logical examples, specifically \index{Curryfication}Curryfication and \index{Dedekind-style reasoning}Dedekind-style reasoning.

\paragraph{Strictification Application:}
By invoking the Strictification Theorem (see Theorem \ref{thm:strictification}), we establish that every bicategory, including our integrated framework, is equivalent to a strict 2-category. This transformation ensures that associativity and unit conditions hold exactly, thereby simplifying further constructions and proofs within the system.

\paragraph{Evaluation via Logical Examples:}
To assess the practical relevance of our approach, we evaluate it using two fundamental logical transformations:
\begin{itemize}
  \item \textbf{Curryfication:} The process of transforming a multi-argument function into a sequence of single-argument functions is elegantly captured via the exponential objects in a cartesian closed category. This provides a clear categorical interpretation of Curryfication.
  \item \textbf{Dedekind-style Reasoning:} Often associated with the analysis of order and completeness properties, this reasoning is modeled within our framework through the universal properties of coproducts and limits. Such an approach demonstrates the framework's capability to represent and manipulate complex logical constructs.
\end{itemize}

\noindent
Collectively, the strictification and evaluation steps not only validate the theoretical foundations of the integrated categorical structure but also illustrate its potential for unifying diverse logical constructs under a single, coherent paradigm.

\newpage
\section{Main Contributions}
\begin{itemize}
  \item A systematic construction of local categories corresponding to logical connectives.
  \item Extension of these local structures to a 2-category framework to flexibly capture universal properties via natural isomorphisms.
  \item A novel method for integrating local categories through 2-categorical composition (including pseudo-limits and pseudo-colimits) ensuring coherence.
  \item Application of strictification techniques to obtain a strict 2-category equivalent, thereby rigorously preserving universality and coherence.
  \item Evaluation of the integrated framework via concrete logical examples such as currying and the Deduction Theorem.
\end{itemize}

\section{Organization of the Thesis}
This thesis is organized as follows:
\begin{itemize}
  \item \textbf{Chapter 2} reviews the necessary background in category theory, logical connectives, and higher category theory.
  \item \textbf{Chapter 3} presents the detailed construction of local categories corresponding to various logical connectives.
  \item \textbf{Chapter 4} extends these local categories to a 2-category framework, introducing natural isomorphisms.
  \item \textbf{Chapter 5} discusses the 2-categorical composition techniques used to integrate the local categories.
  \item \textbf{Chapter 6} rigorously verifies the coherence conditions within the integrated category.
  \item \textbf{Chapter 7} applies strictification techniques to reconstitute the integrated structure into a strict 2-category.
  \item \textbf{Chapter 8} evaluates the proposed framework with concrete logical examples.
  \item \textbf{Chapter 9} concludes the thesis and outlines future research directions.
\end{itemize}

\chapter{Preliminaries and Theoretical Background}
\section{Basic Concepts in Category Theory}
\subsection{Objects, Morphisms, and Composition}
\label{subsec:objects_morphisms_composition}

Categories are foundational structures in mathematics that consist of a collection of objects and morphisms (or arrows) between these objects. The composition of these morphisms, along with the existence of identity arrows, is governed by specific axioms that ensure the structure's consistency.

\begin{defn}
\label{defn:category}
A \index{category} \emph{category} \(\mathcal{C}\) consists of:
\begin{enumerate}
  \item A class of \index{object} \emph{objects}, denoted by \(\operatorname{Ob}(\mathcal{C})\).
  \item For every pair of objects \(A, B \in \operatorname{Ob}(\mathcal{C})\), a set of \index{morphism} \emph{morphisms} from \(A\) to \(B\), denoted by \(\operatorname{Hom}_{\mathcal{C}}(A, B)\).
  \item For every object \(A \in \operatorname{Ob}(\mathcal{C})\), an \emph{identity morphism} \(\mathrm{id}_A \in \operatorname{Hom}_{\mathcal{C}}(A, A)\).
  \item A binary operation called \emph{composition}, which assigns to each pair of morphisms 
  \[
  f \in \operatorname{Hom}_{\mathcal{C}}(A, B) \quad \text{and} \quad g \in \operatorname{Hom}_{\mathcal{C}}(B, C)
  \]
  a morphism
  \[
  g \circ f \in \operatorname{Hom}_{\mathcal{C}}(A, C).
  \]
\end{enumerate}
These data are required to satisfy the following axioms:
\begin{enumerate}
  \item \textbf{Associativity:} For any morphisms
  \[
  f \in \operatorname{Hom}_{\mathcal{C}}(A, B),\quad g \in \operatorname{Hom}_{\mathcal{C}}(B, C),\quad \text{and} \quad h \in \operatorname{Hom}_{\mathcal{C}}(C, D),
  \]
  we have
  \[
  h \circ (g \circ f) = (h \circ g) \circ f.
  \]
  \item \textbf{Identity:} For every morphism \( f \in \operatorname{Hom}_{\mathcal{C}}(A, B) \), the following equalities hold:
  \[
  \mathrm{id}_B \circ f = f \quad \text{and} \quad f \circ \mathrm{id}_A = f.
  \]
\end{enumerate}
\end{defn}

This definition encapsulates the essence of categorical structure, ensuring that the composition of morphisms is well-defined and that each object acts as an identity element with respect to composition.

\noindent
Intuitively, objects represent entities (such as sets, spaces, or algebraic structures), while morphisms represent structure-preserving mappings between these entities. The associativity and identity conditions guarantee that the process of composing morphisms is consistent, thereby enabling abstract and coherent mathematical reasoning within the category.

\subsection{Universal Properties}
\label{subsec:universal_properties}

Universal constructions are central in category theory, as they provide a systematic way to characterize objects by their mapping properties. In particular, constructions such as \index{product}products, \index{coproduct}coproducts, and \index{exponential object}exponentiation capture the essence of logical connectives within a categorical framework.

\begin{defn}
\label{defn:product}
Let \(\mathcal{C}\) be a category and \(A, B\) objects in \(\mathcal{C}\). A \emph{product} of \(A\) and \(B\) is an object \(P\) together with a pair of morphisms
\[
\pi_A: P \to A \quad \text{and} \quad \pi_B: P \to B,
\]
such that for any object \(X\) with morphisms \(f: X \to A\) and \(g: X \to B\), there exists a unique morphism \(\langle f, g \rangle: X \to P\) making the diagram commute:
\[
\xymatrix@R=5em@C=5em{
X 
  \ar@/^1.5pc/[drr]^-{g} 
  \ar[dr]^(0.6){\langle f, g \rangle} 
  \ar@/_1.5pc/[ddr]_-{f} & & \\
& P \ar[r]^-{\pi_B} \ar[d]_-{\pi_A} & B \\
& A &
}
\]
\end{defn}

\begin{defn}
\label{defn:coproduct}
Dually, a \emph{coproduct} of \(A\) and \(B\) is an object \(Q\) together with a pair of injection morphisms
\[
\iota_A: A \to Q \quad \text{and} \quad \iota_B: B \to Q,
\]
such that for any object \(X\) with morphisms \(f: A \to X\) and \(g: B \to X\), there exists a unique morphism \([f, g]: Q \to X\) making the following diagram commute:
\[
\begin{tikzcd}
A \arrow[r, "\iota_A"] \arrow[dr, "f"'] 
  & Q \arrow[d, dashed, "{[f, g]}"] 
  & B \arrow[l, "\iota_B"'] \arrow[dl, "g"] \\
& X &
\end{tikzcd}
\]
\end{defn}

\begin{defn}
\label{defn:exponential}
Let \(\mathcal{C}\) be a cartesian closed category and let \(A, B\) be objects in \(\mathcal{C}\). An \emph{exponential object} \(B^A\) is an object together with an \emph{evaluation} morphism
\[
\mathrm{ev}: B^A \times A \to B,
\]
satisfying the universal property that for any object \(X\) and morphism \(f: X \times A \to B\), there exists a unique morphism \(\tilde{f}: X \to B^A\) such that
\[
f = \mathrm{ev} \circ (\tilde{f} \times \mathrm{id}_A).
\]
\end{defn}

\paragraph{Logical Connectives and Universal Properties:}
These universal constructions correspond to logical connectives in the following manner:
\begin{itemize}
  \item \textbf{Conjunction:} The product \(A \times B\) models logical conjunction (\(\land\)) since it provides the unique means to combine information from \(A\) and \(B\).
  \item \textbf{Disjunction:} The coproduct \(A + B\) serves as a categorical analogue of logical disjunction (\(\lor\)), where the injection morphisms represent the inclusion of each alternative.
  \item \textbf{Implication:} The exponential object \(B^A\) reflects the logical implication (\(\to\)) by encapsulating the idea of a function space from \(A\) to \(B\), corresponding to a transformation that preserves the structure of arguments.
\end{itemize}

\noindent
Thus, universal properties provide a robust and abstract framework for interpreting logical operations in categorical semantics, where the focus lies on the uniqueness and existence of mediating morphisms that make diagrams commute.

\subsection{Examples and Motivating Cases}
\label{subsec:examples_motivating_cases}

In this subsection, we present several classical examples that illustrate the abstract concepts developed in this work. These examples not only ground the theoretical framework but also demonstrate its applicability across different logical and mathematical settings.

\paragraph{Example 1: The Category of Sets}
The category \(\mathbf{Set}\) is the prototypical example where:
\begin{itemize}
  \item Objects are sets.
  \item Morphisms are functions between sets.
  \item The product \(A \times B\) corresponds to the Cartesian product, modeling logical conjunction.
  \item The coproduct \(A \sqcup B\) (disjoint union) models logical disjunction.
  \item The exponential object \(B^A\), which is the set of all functions from \(A\) to \(B\), provides a concrete interpretation of logical implication.
\end{itemize}
This example serves as a foundation, highlighting how universal properties naturally arise in familiar mathematical contexts.

\paragraph{Example 2: Cartesian Closed Categories}
Beyond \(\mathbf{Set}\), any \index{cartesian closed category}cartesian closed category furnishes a setting where the exponential object captures the notion of implication. For instance, in the category of finite sets or even in the category of small categories, the exponential \(B^A\) models the space of morphisms from \(A\) to \(B\), which is central to Curryfication—a process that underpins many constructions in type theory and lambda calculus.

\paragraph{Example 3: Topoi and Internal Logic}
A \index{topos}topos is a category that resembles \(\mathbf{Set}\) but comes equipped with an internal logic. In a topos:
\begin{itemize}
  \item Logical formulas can be interpreted as objects.
  \item Logical connectives are represented through categorical constructions such as limits (for conjunction), colimits (for disjunction), and exponentiation (for implication).
  \item The internal language of a topos allows one to seamlessly translate logical reasoning into categorical terms.
\end{itemize}
This framework shows the deep connection between categorical structures and logic, reinforcing the idea that universal constructions are central to understanding logical semantics.

\paragraph{Motivating Cases}
These examples motivate the need for an integrated categorical framework by demonstrating that:
\begin{itemize}
  \item Universal constructions (products, coproducts, and exponentials) in \(\mathbf{Set}\) and related categories naturally model logical connectives.
  \item Cartesian closed categories and topoi extend these ideas to richer, more expressive logical systems.
  \item An abstract treatment using local categories, their integration, and subsequent strictification captures the essential features observed in these classical examples, thus unifying diverse mathematical theories under a common framework.
\end{itemize}

In summary, the classical examples discussed here provide both intuition and motivation for the abstract constructions developed in this paper, highlighting the role of universal properties in bridging logic and category theory.

\newpage
\section{Categorical Semantics of Logical Connectives}
\subsection{Interpretation of Negation, Conjunction, and Disjunction}
\label{subsec:interpretation_logic_connectives}

The categorical interpretation of logical connectives provides a natural semantic framework where each connective is modeled via a universal construction. In this subsection, we explain how negation, conjunction, and disjunction are interpreted within a categorical setting.

\paragraph{Negation:}  
Negation is often modeled through the use of a dualizing object in a closed category. For an object \(A\) in a category \(\mathcal{C}\), its negation, denoted \(\lnot A\), is characterized by an adjunction with a fixed dualizing object \(D\). More precisely, there exists a natural isomorphism
\[
\operatorname{Hom}_{\mathcal{C}}(A, D) \cong \operatorname{Hom}_{\mathcal{C}}(1, \lnot A),
\]
where \(1\) denotes the terminal object in \(\mathcal{C}\). This formulation captures the notion that negation, much like logical complement, reverses the direction of morphisms in a manner compatible with the adjunction structure.

\paragraph{Conjunction:}  
Conjunction corresponds to the \index{product}product in category theory. Given two objects \(A\) and \(B\) in \(\mathcal{C}\), their product \(A \times B\) is defined by the existence of projection morphisms
\[
\pi_A: A \times B \to A \quad \text{and} \quad \pi_B: A \times B \to B,
\]
which satisfy the following universal property: for any object \(X\) with morphisms \(f: X \to A\) and \(g: X \to B\), there exists a unique mediating morphism \(\langle f, g \rangle: X \to A \times B\) such that
\[
\pi_A \circ \langle f, g \rangle = f \quad \text{and} \quad \pi_B \circ \langle f, g \rangle = g.
\]
This unique factorization mirrors the logical "and" operation, as it combines the information from both \(A\) and \(B\) into a single object.

\paragraph{Disjunction:}  
Dually, disjunction is modeled by the \index{coproduct}coproduct. For objects \(A\) and \(B\) in \(\mathcal{C}\), the coproduct \(A + B\) is equipped with injection morphisms
\[
\iota_A: A \to A + B \quad \text{and} \quad \iota_B: B \to A + B.
\]
The universal property of the coproduct states that for any object \(X\) and morphisms \(f: A \to X\) and \(g: B \to X\), there exists a unique morphism \([f, g]: A + B \to X\) such that
\[
[f, g] \circ \iota_A = f \quad \text{and} \quad [f, g] \circ \iota_B = g.
\]
This construction captures the essence of logical "or" by providing a mechanism to include either \(A\) or \(B\) into a larger structure.

\noindent
Collectively, these interpretations illustrate how the abstract notion of universal properties serves to model logical connectives in a coherent and structured manner within category theory.

\subsection{Implication and Exponential Objects}
\label{subsec:implication_exponential}

In a \index{cartesian closed category}cartesian closed category, logical implication is interpreted through the concept of \emph{exponential objects}. For any objects \(A\) and \(B\) in such a category \(\mathcal{C}\), the exponential object \(B^A\) can be viewed as the categorical analogue of the function space from \(A\) to \(B\).

\begin{defn}
Let \(\mathcal{C}\) be a cartesian closed category and let \(A, B \in \operatorname{Ob}(\mathcal{C})\). An \emph{exponential object} \(B^A\) is an object together with an \emph{evaluation map}
\[
\mathrm{ev}: B^A \times A \to B,
\]
which satisfies the following universal property: For every object \(X \in \operatorname{Ob}(\mathcal{C})\) and every morphism 
\[
f: X \times A \to B,
\]
there exists a unique morphism 
\[
\tilde{f}: X \to B^A
\]
such that the following diagram commutes:
\[
\begin{tikzcd}[column sep=large, row sep=large]
X \times A \arrow[r, "f"] \arrow[d, "\tilde{f}\times\mathrm{id}_A"'] & B \\
B^A \times A \arrow[ur, "\mathrm{ev}"'] &
\end{tikzcd}
\]
\end{defn}

\paragraph{Interpretation:}
This definition implies that to give a morphism \(f: X \times A \to B\) is equivalent to giving a morphism \(\tilde{f}: X \to B^A\). In logical terms, this is analogous to the process of currying in lambda calculus, where a function taking two arguments can be transformed into a function returning another function. Hence, the exponential object \(B^A\) models the logical implication \(A \to B\), capturing the idea that implication corresponds to forming a space of functions from \(A\) to \(B\).

Thus, exponential objects not only provide a framework for understanding functional abstraction in categorical terms but also establish a direct correspondence between logical implication and categorical structure.

\subsection{Universal Properties in Logical Frameworks}
\label{subsec:universal_properties_logical}

Universal properties serve as the backbone of categorical semantics, providing an abstract and robust means to interpret logical connectives. In categorical logic, the meaning of a connective is determined by a universal construction that characterizes it uniquely up to isomorphism.

\paragraph{Conjunction via Products:}  
Consider the product of two objects \(A\) and \(B\) in a category \(\mathcal{C}\). The product \(A \times B\) comes equipped with projection maps
\[
\pi_A: A \times B \to A \quad \text{and} \quad \pi_B: A \times B \to B.
\]
The universal property of the product states that for any object \(X\) with morphisms \(f: X \to A\) and \(g: X \to B\), there exists a unique mediating morphism \(\langle f, g \rangle: X \to A \times B\) such that
\[
\pi_A \circ \langle f, g \rangle = f \quad \text{and} \quad \pi_B \circ \langle f, g \rangle = g.
\]
This uniqueness condition mirrors the logical notion that the truth of the conjunction \(A \land B\) is determined precisely by the simultaneous truth of \(A\) and \(B\).

\paragraph{Disjunction via Coproducts:}  
Dually, the coproduct \(A + B\) models logical disjunction. It is defined by the existence of injection maps
\[
\iota_A: A \to A + B \quad \text{and} \quad \iota_B: B \to A + B,
\]
which satisfy the universal property: for any object \(X\) and morphisms \(f: A \to X\) and \(g: B \to X\), there exists a unique morphism \([f, g]: A + B \to X\) such that
\[
[f, g] \circ \iota_A = f \quad \text{and} \quad [f, g] \circ \iota_B = g.
\]
This reflects the intuition that the disjunction \(A \lor B\) holds if at least one of \(A\) or \(B\) is true, with the unique mediating morphism providing a canonical method to combine evidence from either side.

\paragraph{Implication via Exponential Objects:}  
In cartesian closed categories, the exponential object \(B^A\) captures the essence of logical implication. The universal property of the exponential states that for every object \(X\) and every morphism 
\[
f: X \times A \to B,
\]
there exists a unique morphism 
\[
\tilde{f}: X \to B^A
\]
such that the following diagram commutes:
\[
\begin{tikzcd}[column sep=large, row sep=large]
X \times A \arrow[r, "f"] \arrow[d, "\tilde{f}\times \mathrm{id}_A"'] & B \\
B^A \times A \arrow[ur, "\mathrm{ev}"'] &
\end{tikzcd}
\]
This property aligns with the Curryfication process, where implication is understood as a transformation that converts a function of two arguments into a function returning another function. In logical terms, \(B^A\) represents the set (or object) of proofs of the implication \(A \to B\).

\paragraph{Summary:}  
The universal properties of products, coproducts, and exponential objects provide canonical constructions that underlie the semantics of the logical connectives conjunction, disjunction, and implication, respectively. These constructions ensure that the meaning of a logical connective is captured by the existence and uniqueness of mediating morphisms, thereby endowing the logical framework with a precise and structurally robust semantic interpretation.

\newpage
\section{Introduction to Higher Category Theory}
\subsection{2-Categories and Bicategories}
\label{subsec:2cat_bicat}

\begin{defn}
\label{defn:2-category}
A \emph{2-category} \(\mathcal{C}\) consists of:
\begin{enumerate}
  \item A class of \index{object} objects, denoted by \(\operatorname{Ob}(\mathcal{C})\).
  \item For every pair of objects \(A,B \in \operatorname{Ob}(\mathcal{C})\), a category \(\mathcal{C}(A,B)\) whose objects are called \index{1-morphism} 1-morphisms from \(A\) to \(B\) and whose morphisms are called \index{2-morphism} 2-morphisms.
  \item For every triple of objects \(A,B,C \in \operatorname{Ob}(\mathcal{C})\), a composition functor
  \[
  \circ: \mathcal{C}(B,C) \times \mathcal{C}(A,B) \longrightarrow \mathcal{C}(A,C),
  \]
  which is strictly associative and unital; that is, there exist identity 1-morphisms \(\mathrm{id}_A \in \mathcal{C}(A,A)\) such that for any 1-morphism \(f \in \mathcal{C}(A,B)\),
  \[
  f \circ \mathrm{id}_A = f = \mathrm{id}_B \circ f,
  \]
  and the horizontal composition of 2-morphisms satisfies the interchange law with vertical composition.
\end{enumerate}
\end{defn}

\begin{defn}
\label{defn:bicategory}
A \emph{bicategory} \(\mathcal{B}\) consists of:
\begin{enumerate}
  \item A class of \index{object} objects.
  \item For each pair of objects \(A,B\), a category \(\mathcal{B}(A,B)\) of 1-morphisms and 2-morphisms.
  \item For every triple of objects \(A,B,C\), a composition functor
  \[
  \circ: \mathcal{B}(B,C) \times \mathcal{B}(A,B) \longrightarrow \mathcal{B}(A,C),
  \]
  which is associative and unital only up to specified natural isomorphisms. In particular, there exist:
  \begin{itemize}
    \item An \index{associator}associator isomorphism
    \[
    a_{f,g,h}: (h \circ g) \circ f \xrightarrow{\sim} h \circ (g \circ f),
    \]
    for any three composable 1-morphisms \(f: A \to B\), \(g: B \to C\), and \(h: C \to D\).
    \item \index{left unitor}Left and \index{right unitor}right unitor isomorphisms
    \[
    l_f: \mathrm{id}_B \circ f \xrightarrow{\sim} f \quad \text{and} \quad r_f: f \circ \mathrm{id}_A \xrightarrow{\sim} f,
    \]
    for each 1-morphism \(f: A \to B\).
  \end{itemize}
These isomorphisms are required to satisfy the standard \index{pentagon identity}pentagon and \index{triangle identity}triangle coherence conditions.
\end{enumerate}
\end{defn}

\subsection{2-Morphisms and Natural Transformations}
\label{subsec:2morphisms_natural_transformations}

In a 2-category (or bicategory), the notion of morphisms is enriched by the introduction of 2-morphisms. These are arrows between 1-morphisms that provide a way to express the idea of “morphisms between morphisms”. A typical example of 2-morphisms in many familiar categories is given by natural transformations.

\paragraph{2-Morphisms:}
Let \( f, g: A \to B \) be two 1-morphisms in a 2-category \(\mathcal{C}\). A 2-morphism
\[
\alpha: f \Rightarrow g
\]
is an arrow in the hom-category \(\mathcal{C}(A,B)\) that relates \(f\) and \(g\). These 2-morphisms allow us to compare different ways of mapping from \(A\) to \(B\), and they come equipped with two kinds of compositions:
\begin{itemize}
  \item \textbf{Vertical composition:} Given 2-morphisms \(\alpha: f \Rightarrow g\) and \(\beta: g \Rightarrow h\), their vertical composite is a 2-morphism \(\beta \cdot \alpha: f \Rightarrow h\).
  \item \textbf{Horizontal composition:} Given 2-morphisms \(\alpha: f \Rightarrow f'\) (where \(f, f': A \to B\)) and \(\beta: g \Rightarrow g'\) (where \(g, g': B \to C\)), their horizontal composite is a 2-morphism \(\beta \ast \alpha: g \circ f \Rightarrow g' \circ f'\).
\end{itemize}

\paragraph{Natural Transformations as 2-Morphisms:}
In the 2-category \(\mathbf{Cat}\), whose objects are (small) categories, 1-morphisms are functors, and 2-morphisms are natural transformations, the above concepts become concrete. For two functors \(F, G: \mathcal{A} \to \mathcal{B}\), a natural transformation \(\eta: F \Rightarrow G\) assigns to each object \(A \in \mathcal{A}\) a morphism \(\eta_A: F(A) \to G(A)\) in \(\mathcal{B}\) such that for every morphism \(f: A \to A'\) in \(\mathcal{A}\), the following diagram commutes:
\[
\begin{tikzcd}[column sep=large, row sep=large]
F(A) \arrow[r, "F(f)"] \arrow[d, "\eta_A"'] & F(A') \arrow[d, "\eta_{A'}"] \\
G(A) \arrow[r, "G(f)"'] & G(A')
\end{tikzcd}
\]
Here, \(\eta\) is the 2-morphism between the functors \(F\) and \(G\).

\paragraph{Modifications:}
Beyond natural transformations, when dealing with higher structures, one may also consider \emph{modifications}. Given two natural transformations \(\eta, \theta: F \Rightarrow G\) between functors \(F, G: \mathcal{A} \to \mathcal{B}\), a modification \(\mu: \eta \Rrightarrow \theta\) is a family of 2-morphisms (typically identities in \(\mathbf{Cat}\)) that provide a higher level of coherence between \(\eta\) and \(\theta\). While modifications are less commonly discussed in elementary texts, they become essential in the study of higher category theory to ensure the full coherence of transformations at all levels.

\noindent
Overall, 2-morphisms like natural transformations (and modifications, in more complex settings) play a vital role in expressing how various compositional structures in a 2-category or bicategory are related, ensuring that different paths of composition are coherently isomorphic.

\subsection{Pseudo-Limits and Pseudo-Colimits}
\label{subsec:pseudo_limits_colimits}

In higher category theory, strict limits and colimits are often too rigid to capture the full behavior of diagrams. Instead, one uses \emph{pseudo-limits} and \emph{pseudo-colimits}, where the required commutativity holds only up to coherent isomorphism.

\begin{defn}
Let \(D: \mathcal{J} \to \mathcal{C}\) be a 2-functor from an index category \(\mathcal{J}\) to a 2-category \(\mathcal{C}\). A \emph{pseudo-limit} of \(D\) is an object \(L \in \operatorname{Ob}(\mathcal{C})\) equipped with:
\begin{itemize}
  \item A family of 1-morphisms \(\{\pi_j: L \to D(j)\}_{j \in \operatorname{Ob}(\mathcal{J})}\),
  \item For every morphism \(u: j \to k\) in \(\mathcal{J}\), an invertible 2-morphism
  \[
  \theta_u: \pi_k \Longrightarrow D(u) \circ \pi_j,
  \]
\end{itemize}
such that for any other object \(X\) with 1-morphisms \(\{f_j: X \to D(j)\}\) and invertible 2-morphisms \(\{\phi_u: f_k \Longrightarrow D(u) \circ f_j\}\) satisfying analogous coherence conditions, there exists a unique (up to unique isomorphism) 1-morphism \(u: X \to L\) together with invertible 2-morphisms making all the induced diagrams commute.
\end{defn}

\begin{defn}
Dually, given a 2-functor \(D: \mathcal{J} \to \mathcal{C}\), a \emph{pseudo-colimit} of \(D\) is an object \(C \in \operatorname{Ob}(\mathcal{C})\) together with:
\begin{itemize}
  \item A family of 1-morphisms \(\{\iota_j: D(j) \to C\}_{j \in \operatorname{Ob}(\mathcal{J})}\),
  \item For every morphism \(u: j \to k\) in \(\mathcal{J}\), an invertible 2-morphism
  \[
  \psi_u: \iota_k \Longrightarrow \iota_j \circ D(u),
  \]
\end{itemize}
such that for any other object \(X\) receiving a compatible cone from \(D\), there exists a unique 1-morphism \(v: C \to X\) (up to unique isomorphism) making the entire structure commute up to coherent isomorphism.
\end{defn}

\paragraph{Discussion:}  
The key difference between strict and pseudo-limits (or colimits) is that the universal property is required to hold only \emph{up to a coherent isomorphism} rather than on the nose. This relaxation is critical in higher categories, where many natural constructions fail to be strictly associative or unital but can be made so up to coherent equivalence.

A schematic diagram representing the universal property of a pseudo-limit is given by:
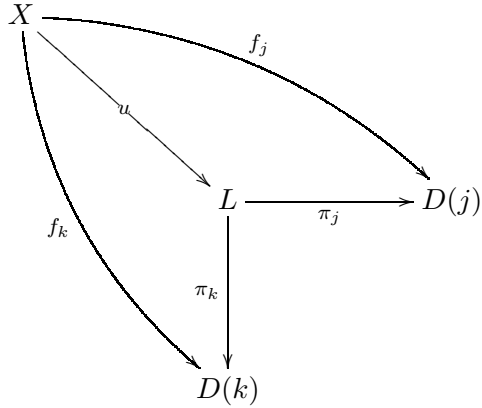
\[
\xymatrix@R=5em@C=5em{
X 
  \ar@/^1.5pc/[drr]^-{f_j} 
  \ar[dr]|-{u} 
  \ar@/_1.5pc/[ddr]_-{f_k} & & \\
& L \ar[r]_-{\pi_j} \ar[d]_-{\pi_k} & D(j) \\
& D(k) &
}
\]
Here, the existence of coherent invertible 2-morphisms between the different compositions ensures that the universal property holds in a flexible (or pseudo) sense.

\noindent
Overall, pseudo-limits and pseudo-colimits provide the appropriate generalization of universal constructions in settings where strict equalities are too constraining, thereby capturing the inherent flexibility of higher categorical structures.

\newpage
\section{Coherence Issues and Strictification}
\subsection{Coherence in Weak Categorical Structures}
\label{subsec:coherence_weak_structures}

In bicategories and other weak 2-categories, composition of 1-morphisms and 2-morphisms is defined only up to specified natural isomorphisms. Unlike strict 2-categories, where associativity and unit laws hold on the nose, weak structures rely on coherence data to ensure that various compositional pathways yield equivalent results.

\paragraph{Challenges:}
\begin{itemize}
  \item \textbf{Non-strict Associativity:} The composition of 1-morphisms is only associative up to an associator isomorphism 
  \[
  a_{f,g,h}: (h \circ g) \circ f \xrightarrow{\sim} h \circ (g \circ f),
  \]
  which must satisfy the \emph{pentagon identity} for any four composable 1-morphisms.
  \item \textbf{Unit Laws Up to Isomorphism:} The left and right unit laws are satisfied only up to natural isomorphisms (the unitors)
  \[
  l_f: \mathrm{id}_B \circ f \xrightarrow{\sim} f \quad \text{and} \quad r_f: f \circ \mathrm{id}_A \xrightarrow{\sim} f,
  \]
  which are required to satisfy the \emph{triangle identity} in conjunction with the associator.
  \item \textbf{Complexity of Coherence Data:} As the categorical structure becomes more intricate (with modifications and higher morphisms), the number and complexity of coherence conditions increase, making it challenging to verify that all diagrams commute.
  \item \textbf{Implications for Universal Constructions:} The flexibility of weak structures can complicate the formulation and verification of universal properties, since the usual uniqueness conditions are replaced by uniqueness up to coherent isomorphism.
\end{itemize}

\paragraph{Approaches to Managing Coherence:}
\begin{itemize}
  \item \emph{Strictification Theorems:} It is often possible to replace a weak 2-category or bicategory with a strictly associative one (i.e., a strict 2-category) without loss of essential structure. For example, every bicategory is biequivalent to a strict 2-category.
  \item \emph{Coherence Theorems:} Coherence theorems, such as Mac Lane's Coherence Theorem for monoidal categories, ensure that all diagrams built from the associators and unitors commute. Such results provide a powerful tool to simplify reasoning in weak categorical settings.
\end{itemize}

\noindent
In summary, while weak categorical structures offer greater flexibility, they require intricate coherence conditions to maintain consistency. Addressing these challenges is crucial for applying higher category theory to areas such as categorical logic and the semantics of logical connectives.

\subsection{The Strictification Theorem}
\label{subsec:strictification_theorem}

\begin{thm}[Strictification Theorem]
Every bicategory \(\mathcal{B}\) is biequivalent to a strict 2-category \(\mathcal{B}^{\mathrm{str}}\). In other words, there exists a strict 2-category \(\mathcal{B}^{\mathrm{str}}\) and a pair of pseudofunctors
\[
F: \mathcal{B} \longrightarrow \mathcal{B}^{\mathrm{str}}, \quad G: \mathcal{B}^{\mathrm{str}} \longrightarrow \mathcal{B},
\]
together with pseudonatural equivalences
\[
G \circ F \simeq \operatorname{Id}_{\mathcal{B}} \quad \text{and} \quad F \circ G \simeq \operatorname{Id}_{\mathcal{B}^{\mathrm{str}}},
\]
which satisfy the requisite coherence conditions.
\end{thm}

\begin{proof}
A complete proof of the strictification theorem involves the following key steps:
\begin{enumerate}
    \item \textbf{Free Construction:} Construct the free strict 2-category on the underlying data of \(\mathcal{B}\).
    \item \textbf{Imposition of Coherence:} Impose the coherence isomorphisms (associators and unitors) of \(\mathcal{B}\) as equations by taking an appropriate quotient of the free 2-category. This forces the associativity and identity constraints to hold strictly.
    \item \textbf{Establishing Biequivalence:} Define pseudofunctors \(F: \mathcal{B} \to \mathcal{B}^{\mathrm{str}}\) and \(G: \mathcal{B}^{\mathrm{str}} \to \mathcal{B}\) and construct pseudonatural transformations that show \(G \circ F\) and \(F \circ G\) are each equivalent to the respective identity functors.
\end{enumerate}
For a detailed treatment, see \cite{GordonPowerStreet1995}.
\end{proof}

\paragraph{Discussion:}
The strictification theorem is fundamental because it allows one to work within a strictly associative framework without losing the essential structure of the original bicategory. This result greatly simplifies the verification of coherence conditions and the construction of further categorical structures.
conversion of a weak structure into a strict 2-category.

\subsection{Implications for Categorical Logic}
\label{subsec:implications_categorical_logic_2}

Strictification techniques offer significant benefits in categorical logic, particularly when integrating local categories that model various logical connectives. By converting a weak structure (such as a bicategory) into a strict 2-category, strictification aids in the following ways:

\begin{itemize}
  \item \textbf{Preservation of Universality:} The universal properties (e.g., products, coproducts, and exponentials) inherent in the local categories are preserved up to coherent isomorphism. This means that even after strictification, the essential logical semantics captured by these universal constructions remain valid.
  
  \item \textbf{Enhanced Coherence:} In weak categorical structures, associativity and identity laws hold only up to natural isomorphism, necessitating elaborate coherence conditions. Strictification enforces these laws strictly, reducing the burden of verifying complex coherence diagrams. This streamlining makes it easier to prove that the integrated system maintains the desired logical relationships.
  
  \item \textbf{Simplification of Logical Reasoning:} With strict composition, reasoning about the interaction of logical connectives becomes more transparent. The elimination of extraneous coherence data minimizes potential ambiguities and errors, leading to clearer proofs and more robust semantic interpretations.
\end{itemize}

\noindent
Thus, strictification is not merely a technical tool—it plays a crucial conceptual role by ensuring that the integrated categorical framework accurately reflects the universal properties and coherence required for a sound interpretation of logical connectives.

\newpage
\section{Summary of Theoretical Background}
\label{sec:summary_background}

In this chapter, we laid the foundation for the integration of logical connectives within a categorical framework. Key concepts and results presented include:

\begin{itemize}
  \item \textbf{Local Categories for Logical Connectives:} We defined local categories corresponding to logical connectives such as negation, conjunction, disjunction, and implication. Each local category is characterized by a universal property—products for conjunction, coproducts for disjunction, dualizing objects for negation, and exponentials for implication—thereby establishing a robust semantic interpretation.
  
  \item \textbf{Universal Properties and Their Role:} Universal constructions, such as products, coproducts, and exponential objects, were shown to underpin the semantics of logical connectives by ensuring the existence and uniqueness (up to isomorphism) of mediating morphisms. These properties are essential in modeling the logical operations within categorical semantics.
  
  \item \textbf{2-Categories and Bicategories:} We introduced 2-categories and bicategories as frameworks that extend the notion of categories by incorporating 2-morphisms. These higher structures allow for a flexible treatment of associativity and identity laws via natural isomorphisms, thereby facilitating a more nuanced handling of coherence.
  
  \item \textbf{Pseudo-Limits and Pseudo-Colimits:} Recognizing that strict limits and colimits are often too rigid in higher category theory, we discussed pseudo-limits and pseudo-colimits. These constructions ensure that universal properties hold up to coherent isomorphism, which is crucial in the context of weak categorical structures.
  
  \item \textbf{Strictification Theorem:} A central result presented was the strictification theorem, which guarantees that every bicategory is biequivalent to a strict 2-category. This transformation not only simplifies the verification of coherence conditions but also preserves the universal properties critical to logical semantics.
  
  \item \textbf{Implications for Categorical Logic:} Finally, we explored how strictification and the associated universal constructions enable the integration of local categories, ensuring that the overall framework maintains both universality and coherence. This lays the groundwork for applying the developed theoretical tools to concrete logical systems.
\end{itemize}

This summary sets the stage for the detailed constructions and integration approaches discussed in subsequent chapters.

\chapter{Local Categories Corresponding to Logical Connectives}
\section{Overview of Local Categories for Logical Connectives}
In this chapter, we introduce the local categories corresponding to the logical connectives (negation, product, coproduct, and exponential). Each category is defined with its own objects, morphisms, and universal properties that capture the semantics of the corresponding logical connective.

\section{Negation Category}
\subsection{Definition and Motivation}
\label{subsec:negation_definition_motivation}

The \index{negation category}negation category is introduced as a local category whose purpose is to capture the logical operation of negation in a categorical setting. In categorical logic, negation is often modeled via duality or by means of an adjunction with a fixed dualizing object. This approach allows us to interpret the logical complement in a manner that is both structurally and semantically robust.

\begin{defn}
\label{defn:negation_category}
Let \(\mathcal{C}\) be a category equipped with a distinguished \index{dualizing object}dualizing object \(D\). For each object \(A \in \operatorname{Ob}(\mathcal{C})\), define its \emph{negation} \(\lnot A\) as an object together with a universal morphism 
\[
\eta_A: A \longrightarrow D,
\]
which satisfies the following universal property: For any object \(X\) and any morphism \(f: A \to X\) that factors through \(D\), there exists a unique morphism \(u: \lnot A \to X\) such that the diagram
\[
\begin{tikzcd}[column sep=large, row sep=large]
A \arrow[r, "\eta_A"] \arrow[dr, "f"'] & D \arrow[d, dashed, "u"] \\
& X
\end{tikzcd}
\]
commutes. The \emph{negation category} is then defined as the category whose objects are those of \(\mathcal{C}\) equipped with the negation operation, and whose morphisms are those that preserve this structure.
\end{defn}

\paragraph{Motivation:}
The construction of the negation category is motivated by the desire to provide a categorical counterpart to the classical logical operation of negation. In traditional logic, negation captures the idea of complementarity and contradiction. Similarly, in categorical terms, by using a dualizing object \(D\) and enforcing the universal property in Definition~\ref{defn:negation_category}, we ensure that the operation \(\lnot A\) behaves as the \emph{logical complement} of \(A\). This not only allows for the systematic treatment of negation within the broader framework of categorical logic but also facilitates the integration of negation with other logical connectives (such as conjunction and disjunction) in a coherent and unified manner.

\subsection{Objects and Morphisms}
\label{subsec:negation_objects_morphisms}

In the negation category, each object is an object \(A\) from the underlying category \(\mathcal{C}\) that is equipped with a distinguished morphism
\[
\eta_A: A \longrightarrow D,
\]
where \(D\) is a fixed \index{dualizing object}dualizing object. This morphism \(\eta_A\) encapsulates the notion of “negation” by relating \(A\) to the dualizing object, and the corresponding object \(\lnot A\) (often defined via a universal property) serves as the categorical representation of the logical complement of \(A\).

\paragraph{Objects:}  
An object in the negation category is given by the pair \((A,\eta_A)\), where:
\begin{itemize}
  \item \(A\) is an object of \(\mathcal{C}\).
  \item \(\eta_A: A \to D\) is a fixed morphism that represents the negation mapping.
\end{itemize}
Intuitively, \(\lnot A\) is defined through the universal property that relates \(A\) to any other object \(X\) via a unique mediating morphism when the map factors through \(D\).

\paragraph{Morphisms:}  
A morphism \(f: (A,\eta_A) \to (B,\eta_B)\) in the negation category is a morphism \(f: A \to B\) in \(\mathcal{C}\) that preserves the negation structure. More precisely, \(f\) is \emph{negation-preserving} if the following diagram commutes:
\[
\begin{tikzcd}[column sep=large, row sep=large]
A \arrow[r, "f"] \arrow[d, "\eta_A"'] & B \arrow[d, "\eta_B"] \\
D \arrow[r, equal] & D
\end{tikzcd}
\]
This commutativity ensures that applying \(f\) to \(A\) is compatible with the negation mappings, meaning that the negation structure is maintained under the morphism.

\noindent
In summary, the objects and morphisms of the negation category are defined to inherently include the negation operation via the distinguished maps \(\eta_A\), thereby providing a categorical framework that mirrors the logical concept of negation.

\subsection{Universal Property and Duality}
\label{subsec:universal_property_duality}

In the negation category, the operation of negation is characterized by a universal property that embodies a form of duality between an object \(A\) and its complement \(\lnot A\). Given a fixed dualizing object \(D\) in the underlying category \(\mathcal{C}\), each object \(A\) is equipped with a negation morphism 
\[
\eta_A: A \to D.
\]
The universal property ensures that the complement \(\lnot A\) of \(A\) is defined in such a way that any morphism from \(A\) into an object \(X\) (which, in some sense, factors through \(D\)) uniquely factors through \(\lnot A\).

\begin{defn}
\label{defn:negation_universal_property}
Let \(A\) be an object in \(\mathcal{C}\) with a negation morphism \(\eta_A: A \to D\). An object \(\lnot A\) is said to be the \emph{negation} (or \emph{complement}) of \(A\) if there exists a morphism 
\[
\nu_A: A \to \lnot A,
\]
such that for every object \(X\) and every morphism \(f: A \to X\) satisfying
\[
\begin{tikzcd}[column sep=large, row sep=large]
A \arrow[r, "f"] \arrow[d, "\eta_A"'] & X \\
D \arrow[ur, dashed, "u"'] &
\end{tikzcd}
\]
there exists a unique morphism \(u: \lnot A \to X\) making the following diagram commute:
\[
\begin{tikzcd}[column sep=large, row sep=large]
A \arrow[r, "\nu_A"] \arrow[dr, "\eta_A"'] & \lnot A \arrow[d, "u"] \\
& D
\end{tikzcd}
\]
\end{defn}

\paragraph{Duality and Complementarity:}  
This universal property reflects the duality inherent in the concept of negation. The map \(\nu_A\) canonically embeds \(A\) into its complement \(\lnot A\), and the unique factorization through \(\lnot A\) ensures that \(\lnot A\) behaves as a true categorical complement of \(A\). This construction parallels the classical logical notion where the negation of a proposition corresponds to its complement, establishing a duality between truth and falsity in the logical framework.

\noindent
In summary, the universal property in Definition~\ref{defn:negation_universal_property} not only defines the negation operation categorically but also provides a rigorous foundation for duality and complementarity within the negation category.

\subsection{Examples and Discussion}
\label{subsec:negation_examples_discussion}

To illustrate the significance of the negation category, we now present some concrete examples and related discussions that highlight its universal property and its role as a model for logical complement.

\paragraph{Example 1: Negation in the Category of Sets}  
In the category \(\mathbf{Set}\), let \(U\) be a fixed universal set and consider the dualizing object \(D = U\). For any subset \(A \subseteq U\), we define the negation of \(A\) as its set-theoretic complement:
\[
\lnot A = U \setminus A.
\]
The canonical inclusion \(\eta_A: A \hookrightarrow U\) plays the role of the negation mapping. The universal property in this context asserts that for any function \(f: A \to X\) that factors through \(U\), there exists a unique function \(u: U \setminus A \to X\) such that the following diagram commutes:
\[
\begin{tikzcd}[column sep=large, row sep=large]
A \arrow[r, "\eta_A"] \arrow[dr, "f"'] & U \arrow[d, dashed, "u"] \\
& X
\end{tikzcd}
\]
This example shows how the notion of complementarity in set theory can be captured categorically via a universal factorization property.

\paragraph{Example 2: Negation in Boolean Algebras}  
Consider a Boolean algebra \(\mathcal{B}\), viewed as a category where objects are the elements of \(\mathcal{B}\) and there is a unique morphism \(a \to b\) if and only if \(a \leq b\). The complement (negation) of an element \(a \in \mathcal{B}\) is given by \(\lnot a\), satisfying the equations
\[
a \land \lnot a = 0 \quad \text{and} \quad a \lor \lnot a = 1.
\]
In categorical terms, the complement \(\lnot a\) is characterized by a universal property: for any element \(x\) such that there exists a morphism \(a \to x\) that factors through the designated dualizing element (here, \(0\) or \(1\), depending on the context), there is a unique mediating morphism \(u: \lnot a \to x\). This mirrors the classical logical notion that an element and its complement collectively exhaust the Boolean universe.

\paragraph{Discussion:}  
These examples emphasize the conceptual power of the negation category. In both the category of sets and Boolean algebras, the universal property ensures that the negation operation is not an arbitrary assignment but is defined by a unique factorization condition. This approach guarantees consistency when combining negation with other logical operations in a broader categorical framework. Moreover, such a formulation is flexible enough to be adapted to more complex logical systems, thereby providing a unified foundation for categorical logic.

The negation category thus serves as a crucial component in the integration of local categories for logical connectives, ensuring that the semantics of negation are captured rigorously and coherently.

\newpage
\section{Product Category}
\subsection{Definition of the Product Structure}
\label{subsec:product_definition}

In the product category, the combination of objects and the corresponding projection morphisms are defined so as to satisfy the universal property of products. In this context, given two objects \(A\) and \(B\) from the underlying category \(\mathcal{C}\), their product is an object \(A \times B\) together with two projection morphisms
\[
\pi_A: A \times B \to A \quad \text{and} \quad \pi_B: A \times B \to B.
\]
These projections must satisfy the following universal property:

\begin{defn}
\label{defn:product_category}
Let \(A\) and \(B\) be objects in \(\mathcal{C}\). A \emph{product} of \(A\) and \(B\) is an object \(A \times B\) together with morphisms \(\pi_A\) and \(\pi_B\) such that for any object \(X\) with morphisms
\[
f: X \to A \quad \text{and} \quad g: X \to B,
\]
there exists a unique morphism \(\langle f, g \rangle: X \to A \times B\) making the following diagram commute:
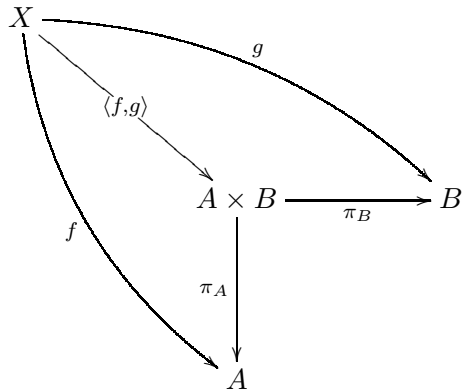
\[
\xymatrix@R=5em@C=5em{
X 
  \ar@/^1.5pc/[drr]^-{g} 
  \ar[dr]|-{\langle f, g \rangle} 
  \ar@/_1.5pc/[ddr]_-{f} & & \\
& A \times B \ar[r]_-{\pi_B} \ar[d]_-{\pi_A} & B \\
& A &
}
\]
That is, the equations
\[
\pi_A \circ \langle f, g \rangle = f \quad \text{and} \quad \pi_B \circ \langle f, g \rangle = g
\]
hold.
\end{defn}

\paragraph{Interpretation:}  
The above definition encapsulates the idea that \(A \times B\) is the most general object that projects onto \(A\) and \(B\) in a way that any other object mapping to \(A\) and \(B\) factors uniquely through \(A \times B\). This universal property is the categorical analogue of the logical conjunction, where having both \(A\) and \(B\) is encoded by the pair \((A, B)\).

\paragraph{Discussion:}  
In our construction of the product category, the focus is on ensuring that the projection morphisms preserve the structure required by the universal property. This guarantees that the product, as defined, is unique up to unique isomorphism and can serve as a building block for more complex categorical constructions in logical semantics.

\subsection{Universal Property of the Product}
\label{subsec:universal_property_product}

Let \(\mathcal{C}\) be a category and let \(A, B \in \mathrm{Ob}(\mathcal{C})\). A \emph{product} of \(A\) and \(B\) is an object \(P\) equipped with two projection morphisms
\[
\pi_A: P \to A \quad \text{and} \quad \pi_B: P \to B,
\]
such that for any object \(X\) and any pair of morphisms
\[
f: X \to A \quad \text{and} \quad g: X \to B,
\]
there exists a unique morphism \(\langle f, g \rangle: X \to P\) making the following diagram commute:
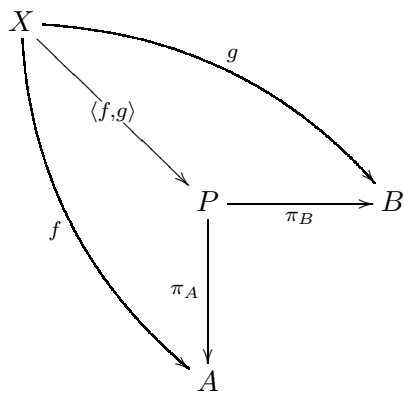
\[
\xymatrix@R=5em@C=5em{
X 
  \ar@/^1.5pc/[drr]^-{g} 
  \ar[dr]|-{\langle f, g \rangle} 
  \ar@/_1.5pc/[ddr]_-{f} & & \\
& P \ar[r]_-{\pi_B} \ar[d]_-{\pi_A} & B \\
& A &
}
\]
That is, the equations
\[
\pi_A \circ \langle f, g \rangle = f \quad \text{and} \quad \pi_B \circ \langle f, g \rangle = g
\]
hold. The uniqueness of \(\langle f, g \rangle\) guarantees that \(P\) is, up to unique isomorphism, the most general object through which every pair of morphisms \(f\) and \(g\) factor.

\subsection{Construction and Examples}
\label{subsec:product_construction_examples}

In this subsection, we present a concrete construction of the product category and illustrate its application with typical examples.

First, recall that for any two objects \(A\) and \(B\) in a category \(\mathcal{C}\), a product \(A \times B\) is defined by an object together with projection morphisms
\[
\pi_A: A \times B \to A \quad \text{and} \quad \pi_B: A \times B \to B,
\]
satisfying the following universal property: for any object \(X\) and any pair of morphisms \(f: X \to A\) and \(g: X \to B\), there exists a unique morphism
\[
\langle f, g \rangle: X \to A \times B
\]
such that
\[
\pi_A \circ \langle f, g \rangle = f \quad \text{and} \quad \pi_B \circ \langle f, g \rangle = g.
\]

\paragraph{Construction in the Category of Sets:}
In the familiar category \(\mathbf{Set}\), the product of two sets \(A\) and \(B\) is the Cartesian product \(A \times B\). The projection maps are the standard coordinate projections:
\[
\pi_A(a,b) = a \quad \text{and} \quad \pi_B(a,b) = b.
\]
Given any set \(X\) and functions \(f: X \to A\) and \(g: X \to B\), the unique function \(\langle f, g \rangle: X \to A \times B\) is defined by
\[
\langle f, g \rangle (x) = (f(x), g(x)).
\]
This construction clearly satisfies the universal property of the product.

\paragraph{Diagrammatic Representation:}
The following commutative diagram summarizes the universal property:
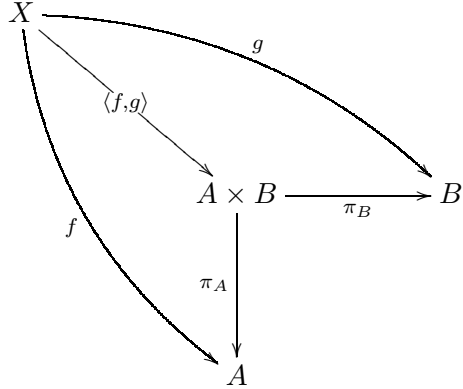
\[
\xymatrix@R=5em@C=5em{
X 
  \ar@/^1.5pc/[drr]^-{g} 
  \ar[dr]|-{\langle f, g \rangle} 
  \ar@/_1.5pc/[ddr]_-{f} & & \\
& A \times B \ar[r]_-{\pi_B} \ar[d]_-{\pi_A} & B \\
& A &
}
\]

\paragraph{Application Example: Logical Conjunction in Categorical Semantics:}
In categorical logic, the product operation is used to model the logical conjunction (\(\land\)). If propositions are interpreted as objects in a category, then the product of two propositions \(A\) and \(B\) (representing “\(A\) and \(B\)”) is modeled by their categorical product \(A \times B\). The universal property guarantees that any evidence proving \(A\) and \(B\) together factors uniquely through the product, aligning with the intuition behind conjunction.

\paragraph{Summary:}
The construction of the product category, exemplified here in the category of sets, provides a robust framework for modeling conjunction in logic and underpins many other applications in categorical semantics. This concrete example serves as a stepping stone to more complex constructions, such as integrating various local categories in a unified semantic framework.

\section{Coproduct Category}
\subsection{Definition of the Coproduct Structure}
\label{subsec:coproduct_definition}

For any objects \(A\) and \(B\) in a category \(\mathcal{C}\), a \emph{coproduct} of \(A\) and \(B\) is an object \(A+B\) together with two injection morphisms
\[
\iota_A: A \to A+B \quad \text{and} \quad \iota_B: B \to A+B,
\]
satisfying the following universal property:

\begin{defn}
Let \(A\) and \(B\) be objects in \(\mathcal{C}\). A \emph{coproduct} of \(A\) and \(B\) is an object \(A+B\) with injections \(\iota_A\) and \(\iota_B\) such that for any object \(X\) and any pair of morphisms
\[
f: A \to X \quad \text{and} \quad g: B \to X,
\]
there exists a unique morphism \([f, g]: A+B \to X\) making the diagram commute:
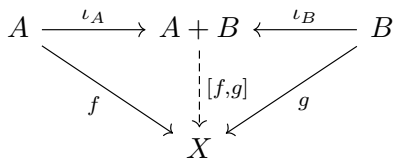
\[
\begin{tikzcd}[column sep=large, row sep=large]
A \arrow[r, "\iota_A"] \arrow[dr, "f"'] & A+B \arrow[d, dashed, "{[f, g]}"] & B \arrow[l, "\iota_B"'] \arrow[dl, "g"] \\
& X &
\end{tikzcd}
\]
That is, the equations
\[
[f, g] \circ \iota_A = f \quad \text{and} \quad [f, g] \circ \iota_B = g
\]
hold.
\end{defn}

\paragraph{Interpretation:}  
This definition encapsulates the idea that \(A+B\) is the most general object that “contains” both \(A\) and \(B\) via the injection maps. Any pair of morphisms from \(A\) and \(B\) to another object \(X\) factors uniquely through \(A+B\), mirroring the logical interpretation of disjunction.

\subsection{Universal Property of the Coproduct}
\label{subsec:universal_property_coproduct}

Let \(A\) and \(B\) be objects in a category \(\mathcal{C}\). A \emph{coproduct} of \(A\) and \(B\) is an object \(A+B\) equipped with injection morphisms
\[
\iota_A: A \to A+B \quad \text{and} \quad \iota_B: B \to A+B,
\]
satisfying the following universal property:

For any object \(X\) and any pair of morphisms
\[
f: A \to X \quad \text{and} \quad g: B \to X,
\]
there exists a unique morphism \([f, g]: A+B \to X\) such that the following diagram commutes:
\[
\begin{tikzcd}[column sep=large, row sep=large]
A \arrow[r, "\iota_A"] \arrow[dr, "f"'] & A+B \arrow[d, dashed, "{[f, g]}"] & B \arrow[l, "\iota_B"'] \arrow[dl, "g"] \\
& X &
\end{tikzcd}
\]
That is, the equations
\[
[f, g] \circ \iota_A = f \quad \text{and} \quad [f, g] \circ \iota_B = g
\]
hold. The uniqueness of \([f, g]\) ensures that the coproduct \(A+B\) is unique up to unique isomorphism.

\subsection{Construction and Examples}
\label{subsec:coproduct_construction_examples}

In this subsection, we present concrete constructions of the coproduct in a local category and illustrate its application through typical examples.

\paragraph{Construction in the Category of Sets:}  
In \(\mathbf{Set}\), the coproduct of two sets \(A\) and \(B\) is given by the disjoint union \(A \sqcup B\). In order to ensure the disjointness of \(A\) and \(B\) (even if they have elements in common), one typically tags the elements, for example:
\[
A \sqcup B = \{ (a,0) \mid a \in A \} \cup \{ (b,1) \mid b \in B \}.
\]
The injection maps are defined as:
\[
\iota_A: A \to A \sqcup B, \quad a \mapsto (a,0), \qquad \iota_B: B \to A \sqcup B, \quad b \mapsto (b,1).
\]
For any set \(X\) and functions \(f: A \to X\) and \(g: B \to X\), the unique function \([f, g]: A \sqcup B \to X\) is defined by:
\[
[f, g](a,0) = f(a) \quad \text{and} \quad [f, g](b,1) = g(b).
\]
This construction satisfies the universal property of the coproduct.

\paragraph{Diagrammatic Representation:}  
The universal property is captured by the following commutative diagram:
\[
\begin{tikzcd}[column sep=large, row sep=large]
A \arrow[r, "\iota_A"] \arrow[dr, "f"'] & A\sqcup B \arrow[d, dashed, "{[f, g]}"] & B \arrow[l, "\iota_B"'] \arrow[dl, "g"] \\
& X &
\end{tikzcd}
\]

\paragraph{Other Examples and Applications:}  
\begin{itemize}
  \item \textbf{Free Products in Group Theory:}  
  In the category of groups, the coproduct is given by the free product of groups. If \(G\) and \(H\) are groups, their free product \(G * H\) satisfies a universal property analogous to that of the disjoint union in \(\mathbf{Set}\), but within the context of group homomorphisms.
  
  \item \textbf{Logical Disjunction:}  
  In categorical logic, the coproduct models the logical disjunction (\(\lor\)). When propositions are interpreted as objects, the disjunction corresponds to a coproduct where the injection maps represent the inclusion of each alternative. This allows one to reason about “either/or” statements in a structurally robust manner.
  
  \item \textbf{Topos Theory:}  
  In a topos, coproducts (or disjoint unions) play a key role in constructing and interpreting internal logical formulas, further bridging the gap between categorical structures and logical semantics.
\end{itemize}

\paragraph{Summary:}  
The construction of the coproduct in various concrete categories—such as sets, groups, and topoi—demonstrates how the universal property facilitates the unique factorization of morphisms through the disjoint union or free product. This not only mirrors the logical operation of disjunction but also provides a powerful tool in the categorical analysis of logical systems.

\newpage
\section{Exponential Category}
\subsection{Definition of Exponential Objects}
\label{subsec:exponential_definition}

In a cartesian closed category \(\mathcal{C}\), for any objects \(A\) and \(B\), an \emph{exponential object} \(B^A\) is defined together with an evaluation morphism
\[
\mathrm{ev}: B^A \times A \longrightarrow B,
\]
which satisfies the following universal property:

\begin{defn}
Let \(\mathcal{C}\) be a cartesian closed category and \(A, B \in \operatorname{Ob}(\mathcal{C})\). The object \(B^A\) is called an \emph{exponential object} (or \emph{function space}) from \(A\) to \(B\) if there exists a morphism
\[
\mathrm{ev}: B^A \times A \to B
\]
such that for every object \(X\) and for every morphism 
\[
f: X \times A \to B,
\]
there exists a unique morphism
\[
\tilde{f}: X \to B^A
\]
making the following diagram commute:
\[
\begin{tikzcd}[column sep=large, row sep=large]
X \times A \arrow[r, "f"] \arrow[d, "\tilde{f}\times \mathrm{id}_A"'] & B \\
B^A \times A \arrow[ur, "\mathrm{ev}"'] &
\end{tikzcd}
\]
That is, the equation
\[
f = \mathrm{ev} \circ (\tilde{f} \times \mathrm{id}_A)
\]
holds. The unique morphism \(\tilde{f}\) is often referred to as the \emph{currying} of \(f\).
\end{defn}

\paragraph{Interpretation:}  
The exponential object \(B^A\) encapsulates the notion of "functions from \(A\) to \(B\)" within the category \(\mathcal{C}\). The evaluation morphism \(\mathrm{ev}\) plays the role of "function application," and the universal property ensures that any morphism \(f: X \times A \to B\) factors uniquely through \(B^A\) via currying. This construction is fundamental in the categorical interpretation of logical implication and functional abstraction.

\subsection{Universal Property and Currying}
\label{subsec:universal_property_currying}

In a cartesian closed category \(\mathcal{C}\), the exponential object \(B^A\) together with the evaluation morphism
\[
\mathrm{ev}: B^A \times A \longrightarrow B,
\]
satisfies the following universal property, which encapsulates the notion of \emph{currying}.

\begin{defn}
\label{defn:currying_universal_property}
Let \(A, B, X \in \operatorname{Ob}(\mathcal{C})\) and let \(B^A\) be the exponential object with evaluation map \(\mathrm{ev}: B^A \times A \to B\). Then for every morphism
\[
f: X \times A \to B,
\]
there exists a unique morphism
\[
\tilde{f}: X \to B^A,
\]
called the \emph{currying} of \(f\), such that the following diagram commutes:
\[
\begin{tikzcd}[column sep=large, row sep=large]
X \times A \arrow[r, "f"] \arrow[d, "\tilde{f}\times\mathrm{id}_A"'] & B \\
B^A \times A \arrow[ur, "\mathrm{ev}"'] &
\end{tikzcd}
\]
That is,
\[
f = \mathrm{ev} \circ (\tilde{f} \times \mathrm{id}_A).
\]
\end{defn}

\paragraph{Interpretation:}  
This universal property asserts that the process of currying provides a one-to-one correspondence between morphisms \(f: X \times A \to B\) and morphisms \(\tilde{f}: X \to B^A\). The unique morphism \(\tilde{f}\) encapsulates the idea that a function of two arguments can be reinterpreted as a function returning another function. This concept is central to both the categorical understanding of implication and the functional abstraction in logic.

\noindent
In summary, the universal property of the exponential object ensures that any mapping out of a product \(X \times A\) factors uniquely through the evaluation map, thereby formalizing the process of currying in a categorical setting.

\subsection{Construction and Examples}
\label{subsec:exponential_construction_examples}

In a cartesian closed category, exponential objects capture the notion of function spaces and enable the process of currying. We now illustrate the construction of exponential objects and provide concrete examples of currying, using the familiar category \(\mathbf{Set}\) as our primary example.

\paragraph{Construction in \(\mathbf{Set}\):}  
For any two sets \(A\) and \(B\), the exponential object \(B^A\) is defined as the set of all functions from \(A\) to \(B\):
\[
B^A = \{ h \mid h: A \to B \}.
\]
The evaluation morphism
\[
\mathrm{ev}: B^A \times A \to B
\]
is given by
\[
\mathrm{ev}(h, a) = h(a) \quad \text{for all } h \in B^A \text{ and } a \in A.
\]

\paragraph{Currying Process:}  
Given any set \(X\) and any function
\[
f: X \times A \to B,
\]
we define the \emph{currying} of \(f\) as the unique function
\[
\tilde{f}: X \to B^A,
\]
where \(\tilde{f}(x)\) is the function from \(A\) to \(B\) defined by
\[
\tilde{f}(x)(a) = f(x, a) \quad \text{for all } x \in X \text{ and } a \in A.
\]
The universal property of the exponential object asserts that this \(\tilde{f}\) is unique and satisfies the equation:
\[
f = \mathrm{ev} \circ (\tilde{f} \times \mathrm{id}_A).
\]

\paragraph{Commutative Diagram:}  
The universal property is illustrated by the following diagram:
\[
\begin{tikzcd}[column sep=large, row sep=large]
X \times A \arrow[r, "f"] \arrow[d, "\tilde{f} \times \mathrm{id}_A"'] & B \\
B^A \times A \arrow[ur, "\mathrm{ev}"'] &
\end{tikzcd}
\]
This diagram commutes, meaning that for all \(x \in X\) and \(a \in A\), we have:
\[
f(x, a) = \mathrm{ev}(\tilde{f}(x), a).
\]

\paragraph{Abstract Perspective:}  
In any cartesian closed category, the exponential object \(B^A\) and the evaluation morphism \(\mathrm{ev}\) give rise to a natural isomorphism:
\[
\operatorname{Hom}(X \times A, B) \cong \operatorname{Hom}(X, B^A),
\]
which formalizes the currying process. This correspondence is central to the interpretation of logical implication in categorical logic, where \(B^A\) represents the internal hom and the process of currying mirrors the abstraction of functions in logic and computer science.

\paragraph{Summary:}  
The construction of exponential objects in \(\mathbf{Set}\) provides a concrete example of how currying works in a cartesian closed category. The unique factorization through the evaluation map ensures that any function \(f: X \times A \to B\) corresponds uniquely to a curried function \(\tilde{f}: X \to B^A\). This universal property is a cornerstone in categorical semantics, linking logical implication to the structure of function spaces.

\newpage
\section{Summary and Transition to 2-Category Integration}
\label{sec:summary_transition_2cat}

In this section, we summarize the key aspects of each local category and discuss how these structures will serve as the building blocks for the 2-category integration presented in the following chapters.

\subsection*{Summary of Local Categories}
Each local category constructed in this work encapsulates the semantics of a specific logical connective:
\begin{itemize}
  \item \textbf{Negation Category:} Models logical negation using a dualizing object and a universal mapping property that captures the notion of complementarity.
  \item \textbf{Product Category:} Represents logical conjunction via the categorical product. The product object is equipped with projection morphisms and satisfies the universal property ensuring unique factorization of any pair of morphisms from an arbitrary object.
  \item \textbf{Coproduct Category:} Models logical disjunction through the coproduct (or disjoint union), characterized by injection morphisms and a universal property that guarantees a unique mediating morphism.
  \item \textbf{Exponential Category:} Captures logical implication using exponential objects in a cartesian closed category. The evaluation morphism and the process of currying provide a categorical formulation of functional abstraction.
\end{itemize}

\subsection*{Transition to 2-Category Integration}
The local categories described above provide robust semantic interpretations for their respective logical connectives. However, to integrate these diverse structures into a cohesive framework, we extend them into a 2-categorical setting. In the following chapters, we will:
\begin{itemize}
  \item Extend each local category by introducing 2-morphisms (such as natural transformations), thereby forming 2-categories or bicategories.
  \item Integrate these extended structures into a unified 2-category that preserves the universal properties of the local categories while accommodating flexible composition via coherent 2-morphisms.
  \item Apply strictification techniques to convert the resulting weak structure into a strict 2-category, thus simplifying the verification of coherence conditions and enabling more streamlined categorical reasoning.
\end{itemize}

This transition not only maintains the integrity of the universal properties in each local category but also creates a solid foundation for a unified categorical logic that operates at the level of 2-categories.

\chapter{2-Category Extension and Natural Isomorphisms}
\section{Motivation for Extending Local Categories to a 2-Category Framework}
\subsection{Limitations of 1-Category Approaches}
\label{subsec:limitations_1cat}

While 1-categories provide a solid foundation for many mathematical constructions, they often lack the flexibility needed to capture certain universal properties and coherence conditions that arise naturally in more complex settings. In particular:

\paragraph{Strict Universal Properties:}
In a 1-category, universal properties (such as those defining products, coproducts, and exponentials) are formulated with strict commutativity. That is, for a given construction, the mediating morphism is required to make a diagram commute exactly. However, many natural constructions only satisfy these properties up to a unique isomorphism rather than by strict equality. This strictness may be overly rigid in contexts where the uniqueness should only be determined up to isomorphism.

\paragraph{Lack of Coherence Mechanisms:}
Coherence conditions ensure that all different ways of composing morphisms yield equivalent results. In a 1-category, the absence of higher morphisms (such as 2-morphisms) means that there is no systematic way to express or manage the flexibility needed to account for associativity or unit laws up to isomorphism. Consequently, complex structures that inherently require coherent isomorphisms (for instance, those found in bicategories) cannot be adequately modeled within the 1-categorical framework.

\noindent
These limitations motivate the extension to higher categorical frameworks, such as 2-categories and bicategories, where universal properties and coherence conditions can be expressed in a more flexible and natural manner.

\subsection{Advantages of 2-Categories}
\label{subsec:advantages_2cat}

One significant advantage of working within a 2-category is the ability to express universal properties via natural isomorphisms—namely, 2-morphisms. In a 2-category, the requirements for universal constructions such as products, coproducts, or exponentials are relaxed: instead of demanding strict equality, the necessary diagrams need only commute up to a specified natural isomorphism. This feature brings several benefits:

\begin{itemize}
  \item \textbf{Flexible Universal Constructions:}  
  In 2-categories, the uniqueness of mediating morphisms is understood up to a natural isomorphism. This means that for a universal property, while there may not be a unique morphism in the strict sense, there is a unique morphism up to coherent isomorphism. This reflects more accurately many mathematical and logical situations where uniqueness is inherently “up-to-isomorphism.”
  
  \item \textbf{Explicit Coherence Data:}  
  The introduction of 2-morphisms allows for a precise expression of coherence conditions. Associativity and unit laws, which must hold strictly in 1-categories, are relaxed to hold up to coherent isomorphisms (associators and unitors) that satisfy standard coherence conditions like the pentagon and triangle identities.
  
  \item \textbf{Simplification of Complex Structures:}  
  Many naturally occurring structures, such as bicategories, inherently involve weak compositions. By working in a 2-category framework, one can capture these weak structures directly, enabling a more natural integration of various local categories. This simplifies reasoning about transformations between different compositional paths.
\end{itemize}

Overall, the use of 2-categories allows us to model universal properties in a way that is both more flexible and more faithful to many real-world mathematical phenomena. The natural isomorphisms provided by 2-morphisms serve as the key tool to manage and simplify the coherence issues that are unavoidable in higher categorical structures.

\newpage
\subsection{Basic Definitions and Examples}
\label{subsec:basic_definitions_examples}

In this section, we introduce the foundational concepts of 2-categories and bicategories, focusing on their components—objects, 1-morphisms, and 2-morphisms—and providing illustrative examples to elucidate these structures.

\paragraph{2-Categories:}

A \textbf{2-category} extends the notion of a category by incorporating morphisms between morphisms, known as 2-morphisms. Formally, a 2-category consists of:

- **Objects**: Entities denoted as \( A, B, C, \ldots \).
- **1-Morphisms**: Also called arrows or morphisms, these are maps between objects. For objects \( A \) and \( B \), a 1-morphism \( f \) is represented as \( f: A \to B \).
- **2-Morphisms**: Also referred to as 2-cells, these are morphisms between 1-morphisms. Given two 1-morphisms \( f, g: A \to B \), a 2-morphism \( \alpha \) from \( f \) to \( g \) is denoted \( \alpha: f \Rightarrow g \).

These components are subject to two types of composition:

1. **Vertical Composition**: The composition of 2-morphisms between the same pair of 1-morphisms. For \( \alpha: f \Rightarrow g \) and \( \beta: g \Rightarrow h \), their vertical composition \( \beta \circ \alpha: f \Rightarrow h \) satisfies associativity and has identity 2-morphisms as units.

2. **Horizontal Composition**: The composition of 1-morphisms and their corresponding 2-morphisms. For \( f: A \to B \) and \( g: B \to C \), their horizontal composition is \( g \circ f: A \to C \). For 2-morphisms \( \alpha: f \Rightarrow f' \) and \( \beta: g \Rightarrow g' \), the horizontal composition \( \beta \ast \alpha: g \circ f \Rightarrow g' \circ f' \) is defined, adhering to associativity and unit laws.

An essential coherence condition in 2-categories is the \textbf{interchange law}, which governs the interaction between vertical and horizontal compositions of 2-morphisms:

\[
(\beta \circ \alpha) \ast (\delta \circ \gamma) = (\beta \ast \delta) \circ (\alpha \ast \gamma)
\]

\paragraph{Example: The 2-Category \(\mathbf{Cat}\):}

A canonical example of a 2-category is \(\mathbf{Cat}\), defined as follows:

- **Objects**: Small categories.
- **1-Morphisms**: Functors between categories.
- **2-Morphisms**: Natural transformations between functors.

In \(\mathbf{Cat}\), vertical composition corresponds to the composition of natural transformations, while horizontal composition corresponds to the composition of functors, with natural transformations composing accordingly.

\paragraph{Bicategories:}

A \textbf{bicategory} generalizes the concept of a 2-category by relaxing the strictness of composition:

- **Associativity and Unit Constraints**: In a bicategory, the composition of 1-morphisms is associative and unital only up to specified natural isomorphisms (called associators and unitors), which themselves satisfy coherence conditions, such as the pentagon and triangle identities.

\paragraph{Example: The Bicategory \(\mathbf{Rel}\):}

An illustrative example of a bicategory is \(\mathbf{Rel}\):

- **Objects**: Sets.
- **1-Morphisms**: Relations between sets.
- **2-Morphisms**: Inclusions of relations.

In \(\mathbf{Rel}\), the composition of relations is associative up to isomorphism, exemplifying the relaxed associativity inherent in bicategories.

\paragraph{Summary:}

Understanding the structures of 2-categories and bicategories is fundamental for modeling contexts where morphisms between morphisms and coherence conditions play a critical role. The examples of \(\mathbf{Cat}\) and \(\mathbf{Rel}\) demonstrate how these higher-categorical frameworks naturally extend traditional category theory, accommodating more complex compositional structures and relationships.

\subsection{Bicategories and Their Weak Structures}
\label{subsec:bicategories_weak}

A \emph{bicategory} is a generalization of a 2-category in which the associativity and identity laws hold only up to specified natural isomorphisms rather than strictly. In a bicategory \(\mathcal{B}\), the structure is given by:

\begin{itemize}
  \item \textbf{Objects:} The objects of \(\mathcal{B}\) are the same as in a 1-category.
  \item \textbf{1-Morphisms:} For any two objects \(A\) and \(B\), the hom-category \(\mathcal{B}(A,B)\) contains the 1-morphisms \(f: A \to B\).
  \item \textbf{2-Morphisms:} For any pair of 1-morphisms \(f, g: A \to B\), 2-morphisms \(\alpha: f \Rightarrow g\) are the morphisms in the hom-category \(\mathcal{B}(A,B)\).
\end{itemize}

In contrast to strict 2-categories, the composition in a bicategory is governed by the following weak structures:

\paragraph{Weak Associativity:}  
For any three composable 1-morphisms
\[
f: A \to B,\quad g: B \to C,\quad h: C \to D,
\]
the two possible ways to compose them, namely \((h \circ g) \circ f\) and \(h \circ (g \circ f)\), are not necessarily equal but are related by an \emph{associator} isomorphism
\[
a_{f,g,h}: (h \circ g) \circ f \xrightarrow{\sim} h \circ (g \circ f).
\]
These associator isomorphisms must satisfy the well-known \textbf{pentagon coherence condition} for any four composable 1-morphisms.

\paragraph{Weak Identity Laws:}  
For each 1-morphism \(f: A \to B\), there exist \emph{left} and \emph{right unitors}:
\[
l_f: \mathrm{id}_B \circ f \xrightarrow{\sim} f \quad \text{and} \quad r_f: f \circ \mathrm{id}_A \xrightarrow{\sim} f,
\]
which are required to satisfy the \textbf{triangle identity} in conjunction with the associators.

\paragraph{Need for Coherence:}  
Since the associativity and unit laws hold only up to natural isomorphism, a variety of different composites can arise. Coherence conditions, such as the pentagon and triangle identities, ensure that all these different ways of composing morphisms are suitably equivalent. In essence, these conditions guarantee that any diagram constructed from associators and unitors commutes, which is crucial for the internal consistency of the bicategory.

\noindent
In summary, bicategories capture the notion of “weak” categorical structures, where the rigidity of strict composition is relaxed in favor of a flexible framework that still retains coherent behavior through a system of natural isomorphisms. This approach is essential for modeling many naturally occurring mathematical phenomena where strict associativity or identity does not hold.

\newpage
\subsection{Strategy for Extension}
\label{subsec:strategy_extension}

To lift the local category structures (negation, product, coproduct, exponential) into a unified 2-categorical framework, we adopt a multi-step strategy that ensures both the preservation of universal properties and the management of coherence via 2-morphisms. The strategy consists of the following key steps:

\begin{enumerate}
  \item \textbf{Enrich Local Categories with 2-Morphisms:}  
  Each local category is extended to a 2-category by introducing 2-morphisms that capture the natural isomorphisms inherent in their universal properties. For example, while a product in a 1-category requires strict commutativity, in the 2-categorical setting, the associativity of the product is mediated by a natural isomorphism (the associator).

  \item \textbf{Define Pseudofunctors for Integration:}  
  Construct pseudofunctors that map the objects, 1-morphisms, and 2-morphisms from each local category into a global 2-category. These pseudofunctors preserve the structure of the local categories up to coherent isomorphism, ensuring that universal constructions (like products or exponentials) remain valid.

  \item \textbf{Establish Coherence Data:}  
  Introduce appropriate coherence isomorphisms (associators, unitors, etc.) to control the weak composition laws. These isomorphisms must satisfy standard coherence conditions (such as the pentagon and triangle identities), which guarantee that different composition orders yield equivalent results.

  \item \textbf{Employ Pseudo-Limits and Pseudo-Colimits:}  
  Use pseudo-limits and pseudo-colimits to glue together the local categories. The flexibility of pseudo-limits/colimits allows the integrated structure to maintain the universal properties of the individual local categories, even when these properties hold only up to isomorphism.

  \item \textbf{Strictification (Optional):}  
  When needed for further applications, apply strictification techniques to convert the weak 2-categorical structure into a strict 2-category without losing the essential universal properties. This step simplifies the verification of coherence conditions in later stages.
\end{enumerate}

By following this strategy, the local categories—each capturing a specific logical connective—are effectively "lifted" into a 2-categorical setting. This integrated framework not only preserves the universal properties of each local structure but also provides the necessary flexibility to manage coherence via 2-morphisms.

\subsection{Incorporating Natural Isomorphisms}
\label{subsec:incorporating_nat_iso}

In the 2-categorical framework, universal properties are often satisfied only up to a natural isomorphism. In other words, rather than requiring strict commutativity of diagrams, we allow them to commute up to a specified 2-morphism. This flexibility is essential for capturing many naturally occurring mathematical phenomena where the uniqueness of mediating morphisms holds only up to isomorphism.

Consider a universal construction in a 1-category where, for any object \(X\) and any pair of morphisms satisfying certain conditions, there exists a unique mediating morphism \(u: X \to U\) making a diagram strictly commute. In a 2-category, the analogous situation involves a 1-morphism \(u: X \to U\) together with a natural isomorphism
\[
\theta: \pi \circ u \xRightarrow{\sim} f,
\]
where \(\pi\) represents the family of projection morphisms from the universal object \(U\), and \(f\) is the given cone over a diagram \(D\). Here, the isomorphism \(\theta\) is a 2-morphism, and it constitutes the coherence data ensuring that the universal property is satisfied up to isomorphism.

This incorporation of natural isomorphisms has several key advantages:
\begin{itemize}
  \item \textbf{Flexibility in Universal Properties:} The universal constructions (such as products, coproducts, and exponentials) need only satisfy their defining properties up to natural isomorphism, which aligns with the inherent variability in many mathematical settings.
  \item \textbf{Management of Coherence:} Natural isomorphisms serve to manage the coherence conditions in a 2-category. They ensure that all reasonable compositions of 1-morphisms yield results that are equivalent in a coherent fashion.
  \item \textbf{Simplification of Integration:} When lifting local categories into a global 2-categorical framework, natural isomorphisms allow different parts of the structure to be integrated without forcing strict equality, thereby simplifying the overall construction.
\end{itemize}

In summary, by introducing natural isomorphisms (2-morphisms), we capture the universal properties of categorical constructions in a flexible manner. This approach not only preserves the essential features of the original universal property but also provides the necessary framework to handle coherence in higher categorical structures.

\subsection{Formal Construction}
\label{subsec:formal_construction}

We now present a formal construction of the integrated 2-category \(\mathcal{D}\) obtained by "lifting" and combining the local categories \(\{\mathcal{C}_i\}_{i \in I}\) (each modeling a logical connective such as negation, product, coproduct, or exponential) using pseudo-limits and pseudo-colimits. The objective is to build a 2-category that preserves the universal properties of the local categories while managing the necessary coherence conditions via 2-morphisms.

\begin{defn}
\label{defn:integrated_2category}
Let \(\{\mathcal{C}_i\}_{i \in I}\) be a collection of local categories. The \emph{integrated 2-category} \(\mathcal{D}\) is defined by the following data:
\begin{enumerate}
  \item \textbf{Objects:} The objects of \(\mathcal{D}\) are given by the disjoint union of the objects of the local categories:
  \[
  \operatorname{Ob}(\mathcal{D}) = \bigsqcup_{i \in I} \operatorname{Ob}(\mathcal{C}_i).
  \]
  \item \textbf{1-Morphisms:} For any two objects \(A\) and \(B\) in \(\mathcal{D}\), the hom-category \(\mathcal{D}(A,B)\) is constructed as a pseudo-limit or pseudo-colimit of the diagram formed by the corresponding hom-categories of the local categories. This construction ensures that the universal properties (e.g., those defining products or exponentials) hold up to a coherent natural isomorphism.
  \item \textbf{2-Morphisms:} The 2-morphisms in \(\mathcal{D}(A,B)\) are defined as the natural transformations between the 1-morphisms, together with additional coherence data inherited from the pseudo-limiting or pseudo-colimiting process.
  \item \textbf{Composition:} Composition of 1-morphisms (and 2-morphisms) in \(\mathcal{D}\) is defined via the universal properties of pseudo-limits and pseudo-colimits. That is, given composable diagrams in the local categories, the composite in \(\mathcal{D}\) is the unique (up to coherent isomorphism) mediating morphism provided by the pseudo-limit construction.
\end{enumerate}
\end{defn}

\paragraph{Role of Pseudo-Limits and Pseudo-Colimits:}  
The key to this construction is the use of pseudo-limits and pseudo-colimits:
\begin{itemize}
  \item \textbf{Pseudo-Limits:} These allow us to define limits in a 2-category where the universal property holds up to natural isomorphism rather than strictly. This flexibility is essential when integrating local categories that satisfy their universal properties only up to isomorphism.
  \item \textbf{Pseudo-Colimits:} Similarly, pseudo-colimits enable the construction of colimits in a way that accommodates the inherent "weakness" (i.e., up-to-isomorphism uniqueness) of the local structures.
\end{itemize}

\paragraph{Conclusion:}  
The integrated 2-category \(\mathcal{D}\) defined in Definition~\ref{defn:integrated_2category} provides a unifying framework in which the universal properties of local categories are preserved via pseudo-limits and pseudo-colimits. This construction not only maintains the essential logical semantics encoded in each local category but also ensures that coherence is managed through the introduction of appropriate 2-morphisms, setting the stage for further integration and strictification in subsequent chapters.

\newpage
\subsection{Definition and Role}
\label{subsec:pseudo_limits_colimits_definition_role}

In a 2-category, the concept of limits and colimits is generalized by allowing the required universal properties to hold only up to a specified natural isomorphism. These generalized constructions are known as \emph{pseudo-limits} and \emph{pseudo-colimits}.

\begin{defn}
Let \(D: \mathcal{J} \to \mathcal{C}\) be a 2-functor from an index 2-category \(\mathcal{J}\) to a 2-category \(\mathcal{C}\). A \emph{pseudo-limit} of \(D\) consists of:
\begin{enumerate}
  \item An object \(L\) in \(\mathcal{C}\),
  \item A family of 1-morphisms \(\{\pi_j: L \to D(j)\}_{j \in \operatorname{Ob}(\mathcal{J})}\),
  \item For each morphism \(u: j \to k\) in \(\mathcal{J}\), an invertible 2-morphism
  \[
  \theta_u: \pi_k \Longrightarrow D(u) \circ \pi_j,
  \]
\end{enumerate}
satisfying appropriate coherence conditions. Moreover, for any object \(X\) with a cone \(\{f_j: X \to D(j)\}\) and invertible 2-morphisms relating \(f_j\) with \(D(u) \circ f_k\) for every \(u: j \to k\), there exists a unique (up to a coherent isomorphism) 1-morphism \(u: X \to L\) such that
\[
f_j \cong \pi_j \circ u,
\]
for all \(j\in\operatorname{Ob}(\mathcal{J})\).
\end{defn}

\begin{defn}
Dually, a \emph{pseudo-colimit} of a 2-functor \(D: \mathcal{J} \to \mathcal{C}\) consists of:
\begin{enumerate}
  \item An object \(C\) in \(\mathcal{C}\),
  \item A family of 1-morphisms \(\{\iota_j: D(j) \to C\}_{j \in \operatorname{Ob}(\mathcal{J})}\),
  \item For each morphism \(u: j \to k\) in \(\mathcal{J}\), an invertible 2-morphism
  \[
  \psi_u: \iota_j \Longrightarrow \iota_k \circ D(u),
  \]
\end{enumerate}
satisfying dual coherence conditions. For any object \(X\) and cocone \(\{g_j: D(j) \to X\}\) with appropriate invertible 2-morphisms, there exists a unique 1-morphism \(v: C \to X\) (up to coherent isomorphism) such that
\[
g_j \cong v \circ \iota_j,
\]
for all \(j\in\operatorname{Ob}(\mathcal{J})\).
\end{defn}

\paragraph{Role and Relevance:}  
Pseudo-limits and pseudo-colimits allow us to capture universal properties in settings where strict equalities are too rigid. In many naturally occurring 2-categorical constructions, the requirement that diagrams commute strictly is unrealistic; instead, commutativity up to a coherent natural isomorphism is both necessary and sufficient. This flexibility is crucial when integrating local categories that model logical connectives:
\begin{itemize}
  \item They enable the preservation of universal properties (e.g., those of products, coproducts, and exponentials) even when the underlying structures are defined only up to isomorphism.
  \item They provide a systematic way to manage the coherence data arising from weak composition laws, ensuring that all associators, unitors, and related 2-morphisms interact consistently.
\end{itemize}

In summary, pseudo-limits and pseudo-colimits serve as essential tools in higher category theory, allowing for the construction of integrated 2-categories where universal properties are maintained in a flexible, yet coherent, manner.

\subsection{Examples in the Context of Logical Connectives}
\label{subsec:pseudo_limits_logical_connectives}

Pseudo-limits and pseudo-colimits allow us to express universal properties in a flexible way—up to coherent natural isomorphism—which is especially useful when modeling logical connectives in higher category theory. Below are examples that illustrate how these constructions capture the universality of products, coproducts, and exponentials.

\paragraph{Product as Logical Conjunction:}  
In a 2-category, the product \(A \times B\) is characterized by the universal property that for any object \(X\) and any pair of morphisms
\[
f: X \to A \quad \text{and} \quad g: X \to B,
\]
there exists a unique mediating morphism \(\langle f, g \rangle: X \to A \times B\) making the diagram commute up to natural isomorphism:
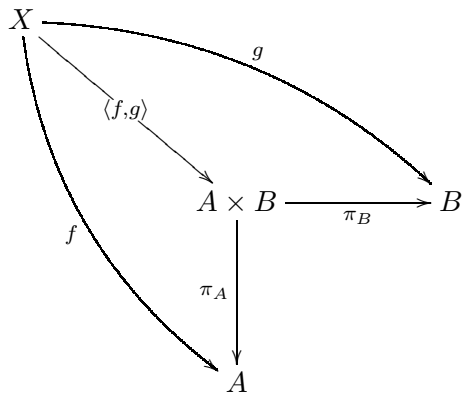
\[
\xymatrix@R=5em@C=5em{
X 
  \ar@/^1.5pc/[drr]^-{g} 
  \ar[dr]|-{\langle f, g \rangle} 
  \ar@/_1.5pc/[ddr]_-{f} & & \\
& A \times B \ar[r]_-{\pi_B} \ar[d]_-{\pi_A} & B \\
& A &
}
\]
This diagram expresses the pseudo-limit property of the product, corresponding to logical conjunction.

\paragraph{Coproduct as Logical Disjunction:}  
Dually, the coproduct \(A+B\) models logical disjunction. Given any object \(X\) and morphisms
\[
f: A \to X \quad \text{and} \quad g: B \to X,
\]
the universal property of the coproduct states that there is a unique mediating morphism \([f, g]: A+B \to X\) (up to coherent isomorphism) such that:
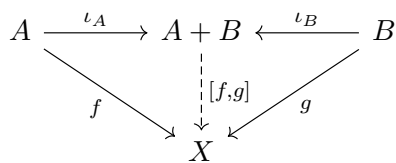
\[
\begin{tikzcd}[column sep=large, row sep=large]
A \arrow[r, "\iota_A"] \arrow[dr, "f"'] & A+B \arrow[d, dashed, "{[f, g]}"] & B \arrow[l, "\iota_B"'] \arrow[dl, "g"] \\
& X &
\end{tikzcd}
\]
This pseudo-colimit formulation captures the essence of logical disjunction.

\paragraph{Exponential as Logical Implication:}  
In a cartesian closed category, the exponential object \(B^A\) models logical implication. Its universal property is expressed by the existence of a unique (up to natural isomorphism) morphism \(\tilde{f}: X \to B^A\) for any morphism
\[
f: X \times A \to B,
\]
such that:
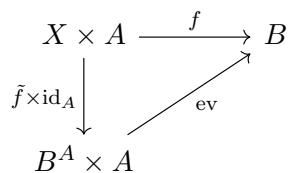
\[
\begin{tikzcd}[column sep=large, row sep=large]
X \times A \arrow[r, "f"] \arrow[d, "\tilde{f}\times\mathrm{id}_A"'] & B \\
B^A \times A \arrow[ur, "\mathrm{ev}"'] &
\end{tikzcd}
\]
This diagram illustrates the currying process, where \(B^A\) acts as the function space from \(A\) to \(B\), corresponding to logical implication.

\paragraph{Discussion:}  
These examples demonstrate that pseudo-limits and pseudo-colimits provide the necessary flexibility to capture universal properties in a 2-categorical framework. By allowing diagrams to commute up to a coherent natural isomorphism, they enable the modeling of logical connectives—such as conjunction, disjunction, and implication—in a way that mirrors the inherent "up-to-isomorphism" nature of many mathematical constructions.

\newpage
\section{Case Studies: Extending Specific Logical Connectives}
\subsection{Negation Category Extension}
\label{subsec:negation_category_extension}

In the classical negation category, each object is a pair \((A,\eta_A)\) where \(A\) is an object in the underlying category \(\mathcal{C}\) and \(\eta_A: A \to D\) is a fixed morphism to a dualizing object \(D\). To extend this structure into a 2-categorical context, we enrich both the morphisms and the universal properties by incorporating 2-morphisms that capture natural isomorphisms.

In the extended negation category \(\mathcal{N}\):
\begin{enumerate}
  \item \textbf{Objects:} Remain as pairs \((A,\eta_A)\) with \(\eta_A: A \to D\) representing the negation structure.
  \item \textbf{1-Morphisms:} A 1-morphism \(f: (A,\eta_A) \to (B,\eta_B)\) is a morphism \(f: A \to B\) in \(\mathcal{C}\) such that the following diagram commutes up to a specified natural isomorphism:
  \[
  \begin{tikzcd}[column sep=large, row sep=large]
  A \arrow[r, "f"] \arrow[d, "\eta_A"'] & B \arrow[d, "\eta_B"] \\
  D \arrow[r, equals] & D
  \end{tikzcd}
  \]
  \item \textbf{2-Morphisms:} For two 1-morphisms \(f, g: (A,\eta_A) \to (B,\eta_B)\), a 2-morphism \(\alpha: f \Rightarrow g\) is a 2-cell in the ambient 2-category which, together with the coherence data, guarantees that the negation structure is preserved. That is, the following diagram commutes up to higher coherence isomorphisms:
  \[
  \begin{tikzcd}[column sep=large, row sep=large]
  A \arrow[r, shift left=0.5ex, "f"] \arrow[r, shift right=0.5ex, "g"'] \arrow[d, "\eta_A"'] & B \arrow[d, "\eta_B"] \\
  D \arrow[r, equals] & D
  \end{tikzcd}
  \]
\end{enumerate}

\paragraph{Duality and Complementarity in the Extended Setting:}  
In this 2-categorical framework, the notion of duality is enhanced by the introduction of 2-morphisms. The weak commutativity in the defining diagram for negation is now witnessed by a natural isomorphism rather than strict equality. This approach captures the idea that the universal property of negation holds uniquely only up to a coherent isomorphism. As a result:
\begin{itemize}
  \item Different 1-morphisms that preserve the negation structure can be compared via 2-morphisms, allowing for a more flexible interpretation of logical complement.
  \item The coherence data provided by 2-morphisms ensures that all diagrams expressing the negation universal property commute in a weak sense, preserving the duality between an object and its complement.
\end{itemize}

\paragraph{Conclusion:}  
By lifting the classical negation category into a 2-categorical context, we obtain a framework where logical negation is modeled more faithfully. The duality and complementary aspects are preserved and enriched via natural isomorphisms, thus enabling a robust treatment of negation that integrates seamlessly with other logical connectives in a higher categorical setting.

\subsection{Product, Coproduct, and Exponential Categories}
\label{subsec:lifting_standard_categories}

The process of lifting product, coproduct, and exponential categories into a 2-categorical framework involves enriching the classical (1-categorical) constructions with additional layers of 2-morphisms. These 2-morphisms serve as natural isomorphisms that witness the universal properties in a flexible manner, ensuring that the essential features of these constructions are preserved even when strict commutativity is relaxed.

\paragraph{Lifting Process:}
For each of the standard constructions—product, coproduct, and exponential—the lifting process follows a similar strategy:
\begin{enumerate}
  \item \textbf{Enrichment of Hom-Categories:}  
  Replace the hom-sets in the classical category with hom-categories. In these enriched categories, 1-morphisms remain as the original morphisms, while 2-morphisms are natural transformations between these morphisms.
  
  \item \textbf{Universal Property Up to Isomorphism:}  
  In the 1-categorical setting, the universal property (e.g., the existence and uniqueness of the mediating morphism) holds strictly. When lifted to a 2-category, this property is weakened so that the relevant diagrams commute only up to a specified natural isomorphism. For example, for the product \(A \times B\), given any object \(X\) and morphisms \(f: X \to A\) and \(g: X \to B\), there exists a unique mediating 1-morphism \(\langle f, g \rangle: X \to A \times B\) along with a natural isomorphism ensuring that
  \[
  \pi_A \circ \langle f, g \rangle \cong f \quad \text{and} \quad \pi_B \circ \langle f, g \rangle \cong g.
  \]
  
  \item \textbf{Coherence via 2-Morphisms:}  
  The natural isomorphisms (2-morphisms) introduced in the lifting process provide the coherence data required to manage the weak commutativity of diagrams. These 2-morphisms ensure that any two different ways of constructing the universal morphisms are coherently isomorphic, thus preserving the intended universal property.
\end{enumerate}

\paragraph{Examples:}
\begin{itemize}
  \item \textbf{Product Categories:}  
  In the lifted product category, the classical product \(A \times B\) is enriched by defining 2-morphisms between any two mediating morphisms that satisfy the projection conditions. The associativity of products, for instance, is now governed by an associator 2-morphism, making the universal property hold up to coherent isomorphism.
  
  \item \textbf{Coproduct Categories:}  
  Similarly, the coproduct \(A+B\) is lifted by replacing the strict uniqueness of the mediating morphism with uniqueness up to a natural isomorphism. The injection morphisms \(\iota_A\) and \(\iota_B\) now come with coherence 2-morphisms that ensure any pair of morphisms from \(A\) and \(B\) into another object \(X\) factors uniquely (up to isomorphism) through \(A+B\).
  
  \item \textbf{Exponential Categories:}  
  In a cartesian closed category, the exponential object \(B^A\) is defined together with an evaluation morphism. When lifted into the 2-categorical setting, the universal property of exponentials (i.e., the currying correspondence between morphisms \(X \times A \to B\) and \(X \to B^A\)) holds up to a coherent natural isomorphism. This ensures that the process of currying is well-defined even when the underlying diagrams commute only weakly.
\end{itemize}

\paragraph{Conclusion:}
By enriching local categories with 2-morphisms and using pseudo-limits and pseudo-colimits, the lifting process transforms strict universal properties into ones that hold up to natural isomorphism. This approach preserves the essence of product, coproduct, and exponential constructions while providing the flexibility necessary for managing coherence in higher categorical structures.

\newpage
\section{Summary and Transition}
\label{sec:summary_transition_4}

In this chapter, we have extended our classical local categories into the 2-categorical framework. The key points of this extension include:

\begin{itemize}
  \item \textbf{Enrichment with 2-Morphisms:}  
  Each local category—whether modeling negation, product, coproduct, or exponential structures—has been enriched with 2-morphisms. These natural isomorphisms allow the universal properties to hold up to coherent isomorphism rather than strict equality.

  \item \textbf{Pseudo-Limits and Pseudo-Colimits:}  
  By employing pseudo-limits and pseudo-colimits, we have ensured that the essential universal properties of these constructions are maintained even in a weak (2-categorical) setting.

  \item \textbf{Preservation of Coherence:}  
  The introduction of associators, unitors, and other coherence isomorphisms guarantees that various compositions in the 2-categorical framework are compatible, thereby overcoming the limitations of strict 1-categories.

  \item \textbf{Integration Readiness:}  
  With local categories now lifted into a 2-categorical context, we have established a robust foundation for further integration. The flexible structure provided by 2-morphisms paves the way for unifying these diverse logical constructs into a cohesive global framework.
\end{itemize}

In the next chapter, we will delve into 2-categorical composition and integration, where the enriched structure of our local categories will be composed and combined. This will involve detailed analysis of the horizontal and vertical compositions of 2-morphisms, as well as the methods to ensure coherence across the integrated framework.

\chapter{2-Category Composition and Integration of Local Categories}
\section{Introduction to 2-Category Composition and Integration}
\subsection{Motivation and Objectives}
\label{subsec:motivation_objectives}

In modern categorical logic, local categories have been constructed to model individual logical connectives—such as negation, conjunction, disjunction, and implication—each with its own universal property. However, to capture the full semantics of logical systems, it is essential to integrate these diverse structures into a unified framework. This motivates the extension of local categories into a 2-categorical setting.

The main objectives of this integration are:
\begin{itemize}
  \item \textbf{Preservation of Universal Properties:}  
  Ensure that the universal properties (e.g., those of products, coproducts, and exponentials) are maintained up to coherent natural isomorphism when local categories are lifted into the 2-categorical framework.
  
  \item \textbf{Ensuring Coherence via 2-Morphisms:}  
  Introduce 2-morphisms that serve as natural isomorphisms to manage the inherent flexibility of composition. These 2-morphisms guarantee that all associativity and unit laws hold up to coherent isomorphism, thereby resolving the limitations of strict equality in 1-categories.
  
  \item \textbf{Unified Structural Integration:}  
  Combine the enriched local categories into a global 2-category that facilitates seamless interactions among the various logical connectives, enabling a coherent categorical semantics for complex logical systems.
\end{itemize}

By achieving these goals, the integration process will provide a robust framework where the logical operations and their interrelations are modeled in a manner that reflects both their individual universal properties and their global coherence.

\newpage
\subsection{Overview of the Integration Process}
\label{subsec:integration_overview}

The integration process of local categories into a unified 2-category is achieved by systematically combining various composition techniques. The following roadmap outlines the main steps and tools used in this integration:

\begin{enumerate}
  \item \textbf{Direct Product Composition:}  
  Initially, local categories are combined via the direct product, which provides a straightforward way to merge objects and morphisms from different sources while preserving their structure.

  \item \textbf{Bifunctors:}  
  Bifunctors are then employed to mediate between pairs of local categories. These functors act on both objects and morphisms simultaneously, ensuring that the interactions between different logical connectives (e.g., conjunction and disjunction) are coherently managed.

  \item \textbf{Pseudo-Limits and Pseudo-Colimits:}  
  Finally, pseudo-limits and pseudo-colimits are used to glue together the local categories in a flexible manner. Since universal properties in higher category theory often hold only up to natural isomorphism, these constructions allow the integrated structure to maintain its universal properties and coherence without requiring strict equality.
\end{enumerate}

Collectively, these techniques yield an integrated 2-category that faithfully reflects the universal properties and logical semantics inherent in the individual local categories.

\newpage
\subsection{Direct Product and Its Role in Integration}
\label{subsec:direct_product_integration}

The direct product is a fundamental construction that provides a natural way to combine objects from different local categories. In a 1-category, given objects \(A\) and \(B\) in a category \(\mathcal{C}\), their direct product \(A \times B\) is defined by the existence of projection morphisms
\[
\pi_A: A \times B \to A \quad \text{and} \quad \pi_B: A \times B \to B,
\]
satisfying the universal property: for any object \(X\) and any pair of morphisms \(f: X \to A\) and \(g: X \to B\), there exists a unique mediating morphism \(\langle f, g \rangle: X \to A \times B\) such that
\[
\pi_A \circ \langle f, g \rangle = f \quad \text{and} \quad \pi_B \circ \langle f, g \rangle = g.
\]

When we lift local categories into a 2-categorical setting, the direct product construction is enriched in the following ways:
\begin{itemize}
  \item \textbf{Enrichment of Hom-Categories:}  
  In the 2-category, the hom-sets are replaced by hom-categories where 1-morphisms are as in the 1-category, and 2-morphisms (natural isomorphisms) relate different factorizations. The universal property of the product then holds up to coherent isomorphism.
  
  \item \textbf{Preservation of Universal Properties:}  
  The projection 1-morphisms \(\pi_A\) and \(\pi_B\) now come with associated 2-morphisms ensuring that for any object \(X\) with 1-morphisms \(f: X \to A\) and \(g: X \to B\), there exists a unique mediating 1-morphism (up to a coherent natural isomorphism) that factors the cone through \(A \times B\).
  
  \item \textbf{Foundation for Integration:}  
  The direct product serves as a basic building block in the integration process. By combining local structures through direct products, we set the stage for more complex constructions (e.g., pseudo-limits and bifunctors) that integrate these structures into a unified 2-category.
\end{itemize}

In summary, the direct product not only merges objects from different local categories but also preserves their universal properties via natural isomorphisms. This enrichment is essential for integrating local categories into a global 2-categorical framework where the interplay of 1-morphisms and 2-morphisms ensures overall coherence.

\subsection{Bifunctors for Mapping Local Categories}
\label{subsec:bifunctors_mapping}

Bifunctors provide a systematic way to integrate local categories into a unified 2-categorical framework. In our setting, a bifunctor acts on pairs of objects and morphisms from two local categories and maps them into the integrated category while preserving the structure and universal properties inherent in each local category.

Suppose we have two local categories \(\mathcal{C}\) and \(\mathcal{D}\). A bifunctor
\[
F: \mathcal{C} \times \mathcal{D} \to \mathcal{E}
\]
is defined by:
\begin{itemize}
  \item \textbf{On Objects:} For each pair \((A, B)\) with \(A \in \operatorname{Ob}(\mathcal{C})\) and \(B \in \operatorname{Ob}(\mathcal{D})\), the bifunctor assigns an object \(F(A, B) \in \operatorname{Ob}(\mathcal{E})\).
  \item \textbf{On Morphisms:} For morphisms \(f: A \to A'\) in \(\mathcal{C}\) and \(g: B \to B'\) in \(\mathcal{D}\), it assigns a morphism
  \[
  F(f, g): F(A, B) \to F(A', B')
  \]
  in \(\mathcal{E}\).
\end{itemize}

The bifunctor \(F\) satisfies the following coherence conditions:
\begin{align*}
F(\mathrm{id}_A, \mathrm{id}_B) &= \mathrm{id}_{F(A,B)}, \quad \text{and} \\
F(f_2 \circ f_1,\, g_2 \circ g_1) &= F(f_2, g_2) \circ F(f_1, g_1),
\end{align*}
ensuring that the composition and identity structures of the local categories are preserved in \(\mathcal{E}\).

\paragraph{Preservation of Universal Properties:}  
When local categories come equipped with universal properties—such as the existence of products, coproducts, or exponentials—the corresponding universal diagrams are mapped into \(\mathcal{E}\) via \(F\) in such a way that these properties hold up to natural isomorphism. For example, if \(\mathcal{C}\) has a product \(A \times B\) with its universal cone, then \(F\) maps this cone to a cone in \(\mathcal{E}\) that satisfies the product universal property up to a coherent 2-isomorphism.

\paragraph{Role in Integration:}  
By mapping local categories into a global integrated structure, bifunctors serve as the glue that connects different logical connectives. They ensure that the key properties—such as the unique factorization of morphisms (up to natural isomorphism)—are maintained throughout the integration process. In this way, bifunctors enable us to:
\begin{itemize}
  \item Combine the structures of different local categories without losing their intrinsic universal properties.
  \item Systematically lift 1-categorical data into the 2-categorical setting, where natural isomorphisms (2-morphisms) handle the necessary flexibility and coherence.
\end{itemize}

Overall, the construction of bifunctors is essential for integrating local categories into a unified 2-category, as they provide the framework needed to preserve both structure and universal properties across the integrated system.

\newpage
\subsection{Definition and Theoretical Background}
\label{subsec:pseudo_limits_colimits_background}

In a 2-category, the classical notions of limits and colimits are generalized to \emph{pseudo-limits} and \emph{pseudo-colimits} in order to accommodate the inherent weakness of 2-categorical structures. Rather than requiring that diagrams commute strictly, these constructions allow them to commute up to a coherent natural isomorphism. This relaxation is essential when working with 2-categories and bicategories, where composition and identities hold only up to specified 2-morphisms.

\begin{defn}
\label{defn:pseudo_limit}
Let \(D: \mathcal{J} \to \mathcal{C}\) be a 2-functor from an index 2-category \(\mathcal{J}\) to a 2-category \(\mathcal{C}\). A \emph{pseudo-limit} of \(D\) consists of:
\begin{enumerate}
  \item An object \(L\) of \(\mathcal{C}\),
  \item A family of 1-morphisms \(\{\pi_j: L \to D(j)\}_{j \in \operatorname{Ob}(\mathcal{J})}\),
  \item For each 1-morphism \(u: j \to k\) in \(\mathcal{J}\), an invertible 2-morphism 
  \[
  \theta_u: \pi_k \Longrightarrow D(u) \circ \pi_j,
  \]
\end{enumerate}
such that for any other object \(X\) equipped with a cone \(\{f_j: X \to D(j)\}\) and invertible 2-morphisms \(\{\phi_u: f_k \Longrightarrow D(u) \circ f_j\}\) satisfying the analogous coherence conditions, there exists a unique (up to a unique invertible 2-morphism) 1-morphism \(u: X \to L\) making all the corresponding diagrams commute up to the given 2-morphisms.
\end{defn}

\begin{defn}
\label{defn:pseudo_colimit}
Dually, a \emph{pseudo-colimit} of a 2-functor \(D: \mathcal{J} \to \mathcal{C}\) is defined by an object \(C\) in \(\mathcal{C}\) together with:
\begin{enumerate}
  \item A family of 1-morphisms \(\{\iota_j: D(j) \to C\}_{j \in \operatorname{Ob}(\mathcal{J})}\),
  \item For each 1-morphism \(u: j \to k\) in \(\mathcal{J}\), an invertible 2-morphism 
  \[
  \psi_u: \iota_j \Longrightarrow \iota_k \circ D(u),
  \]
\end{enumerate}
such that for any other object \(X\) with a cocone \(\{g_j: D(j) \to X\}\) and corresponding coherent 2-morphisms, there exists a unique (up to a unique invertible 2-morphism) 1-morphism \(v: C \to X\) satisfying
\[
g_j \cong v \circ \iota_j
\]
for all \(j\in\operatorname{Ob}(\mathcal{J})\).
\end{defn}

\paragraph{Relevance in Ensuring Flexible Universal Properties:}  
Pseudo-limits and pseudo-colimits are indispensable in higher category theory because they allow us to retain the essence of universal constructions without the rigidity imposed by strict equalities. Instead, the universal properties are satisfied up to coherent isomorphisms, which is more in line with the natural behavior of many mathematical structures. This flexibility is particularly critical when integrating local categories into a global 2-categorical framework, as it:
\begin{itemize}
  \item Preserves the universal properties of the local structures (e.g., products, coproducts, exponentials) in a manner that respects their inherent "up-to-isomorphism" uniqueness.
  \item Facilitates the management of coherence by systematically employing 2-morphisms to mediate between different compositional paths.
\end{itemize}

In summary, pseudo-limits and pseudo-colimits provide the theoretical foundation needed to extend universal constructions from 1-categories to 2-categories, thereby ensuring that the essential properties of local categories are maintained within a flexible yet coherent global framework.

\subsection{Application to Logical Connectives}
\label{subsec:application_logical_connectives}

Pseudo-limits and pseudo-colimits are instrumental in capturing the flexible universal properties of standard constructions—such as products, coproducts, and exponential objects—in a 2-categorical setting. In the context of logical connectives, these constructions provide a categorical semantics where the universal properties hold up to coherent isomorphism.

\paragraph{Products (Logical Conjunction):}
In a 2-category, the product \(A \times B\) is defined with projection 1-morphisms
\[
\pi_A: A \times B \to A \quad \text{and} \quad \pi_B: A \times B \to B,
\]
such that for any object \(X\) and a pair of 1-morphisms \(f: X \to A\) and \(g: X \to B\), there exists a unique mediating 1-morphism \(\langle f, g \rangle: X \to A \times B\) along with a coherent 2-morphism expressing that
\[
\pi_A \circ \langle f, g \rangle \cong f \quad \text{and} \quad \pi_B \circ \langle f, g \rangle \cong g.
\]
Here, the pseudo-limit aspect allows the diagram to commute up to a natural isomorphism, which mirrors the logical interpretation of conjunction that “\(A\) and \(B\)” holds if both \(A\) and \(B\) are satisfied, modulo a coherent equivalence.

\paragraph{Coproducts (Logical Disjunction):}
Dually, the coproduct \(A+B\) is constructed with injection 1-morphisms
\[
\iota_A: A \to A+B \quad \text{and} \quad \iota_B: B \to A+B.
\]
For any object \(X\) and any pair of 1-morphisms \(f: A \to X\) and \(g: B \to X\), the pseudo-colimit property guarantees a unique 1-morphism \([f, g]: A+B \to X\) such that
\[
[f, g] \circ \iota_A \cong f \quad \text{and} \quad [f, g] \circ \iota_B \cong g.
\]
This weak universal property captures the essence of logical disjunction, where “\(A\) or \(B\)” is satisfied in a way that is unique up to a coherent natural isomorphism.

\paragraph{Exponential Objects (Logical Implication):}
In a cartesian closed 2-category, the exponential object \(B^A\) is defined along with an evaluation 1-morphism
\[
\mathrm{ev}: B^A \times A \to B.
\]
For any object \(X\) and any 1-morphism \(f: X \times A \to B\), there is a unique 1-morphism \(\tilde{f}: X \to B^A\) (the currying of \(f\)) such that the diagram
\[
\begin{tikzcd}[column sep=large, row sep=large]
X \times A \arrow[r, "f"] \arrow[d, "\tilde{f}\times \mathrm{id}_A"'] & B \\
B^A \times A \arrow[ur, "\mathrm{ev}"'] &
\end{tikzcd}
\]
commutes up to a natural isomorphism. The pseudo-limit structure here ensures that the equivalence between the mapping \(f\) and its curried form \(\tilde{f}\) holds in a flexible manner, corresponding to the logical notion of implication.

\paragraph{Summary:}
By allowing the universal properties to hold up to coherent isomorphism rather than strictly, pseudo-limits and pseudo-colimits provide the necessary flexibility to model logical connectives in a higher categorical setting. This approach preserves the essential behavior of products, coproducts, and exponentials—thus faithfully representing logical conjunction, disjunction, and implication—while accommodating the inherent "up-to-isomorphism" nature of 2-categorical structures.

\subsection{Formal Construction of the Integrated Category}
\label{subsec:formal_construction_integrated}

We now present the formal procedure for constructing the integrated 2-category \(\mathcal{I}\) from a collection of local categories \(\{\mathcal{C}_i\}_{i\in I}\), each endowed with its own universal constructions (products, coproducts, exponentials, negation, etc.). The key idea is to use pseudo-limits and pseudo-colimits to “glue” these local structures together in a way that preserves their universal properties up to coherent natural isomorphism.

\paragraph{Step 1: Object Formation.}  
Define the objects of \(\mathcal{I}\) as the disjoint union of the objects of each local category:
\[
\operatorname{Ob}(\mathcal{I}) = \bigsqcup_{i\in I}\operatorname{Ob}(\mathcal{C}_i).
\]
Each object retains its local structure (e.g., the negation mapping, product structure, etc.) within its originating category.

\paragraph{Step 2: Hom-Categories via Pseudo-Limits/Colimits.}  
For any two objects \(A\) and \(B\) in \(\mathcal{I}\), define the hom-category \(\mathcal{I}(A,B)\) as follows:
\begin{itemize}
  \item If \(A\) and \(B\) both belong to the same local category \(\mathcal{C}_i\), then set
  \[
  \mathcal{I}(A,B) := \mathcal{C}_i(A,B).
  \]
  \item If \(A\) and \(B\) belong to different local categories (or when integrating overlapping structure from different connectives), construct \(\mathcal{I}(A,B)\) as a pseudo-limit or pseudo-colimit of the diagram formed by the corresponding hom-categories. That is, let \(D: \mathcal{J} \to \mathbf{Cat}\) be a diagram whose objects are the relevant hom-categories, and define
  \[
  \mathcal{I}(A,B) = \mathrm{pLim}\, D \quad \text{or} \quad \mathcal{I}(A,B) = \mathrm{pColim}\, D,
  \]
  ensuring that the universal properties hold up to natural isomorphism.
\end{itemize}

\paragraph{Step 3: Composition of 1-Morphisms and 2-Morphisms.}  
The composition in \(\mathcal{I}\) is defined using the universal properties provided by the pseudo-limits and pseudo-colimits. Specifically, for composable 1-morphisms
\[
f: A \to B \quad \text{and} \quad g: B \to C,
\]
their composite \(g\circ f: A \to C\) is given by the unique (up to coherent isomorphism) mediating 1-morphism obtained from the pseudo-limit (or pseudo-colimit) construction on the corresponding diagram. Similarly, the vertical and horizontal compositions of 2-morphisms are inherited from the corresponding compositions in the local hom-categories, with coherence ensured by the natural isomorphisms in the pseudo-constructions.

\paragraph{Step 4: Coherence Data and Compatibility.}  
To complete the construction, we specify coherence isomorphisms (associators, unitors, etc.) in \(\mathcal{I}\) by:
\begin{enumerate}
  \item Invoking the coherence isomorphisms already present in the local categories.
  \item Extending these isomorphisms to the integrated hom-categories via the universal property of the pseudo-limits/colimits.  
\end{enumerate}
These coherence data ensure that all diagrams (such as the pentagon and triangle diagrams) commute up to a coherent 2-isomorphism, thereby maintaining the integrity of the integrated structure.

\paragraph{Conclusion:}  
The integrated 2-category \(\mathcal{I}\) constructed via the above steps successfully unifies the local categories by:
\begin{itemize}
  \item Combining their objects in a disjoint union,
  \item Forming enriched hom-categories via pseudo-limits and pseudo-colimits,
  \item Defining composition that respects the universal properties up to coherent isomorphism, and
  \item Incorporating the necessary coherence data.
\end{itemize}
This framework preserves the universal properties of products, coproducts, exponentials, and other logical constructions while accommodating the flexibility required by weak (2-categorical) structures.

\newpage
\subsection{Coherence Diagrams in the Integrated Setting}
\label{subsec:coherence_diagrams}

To ensure that the integrated 2-category maintains its logical and categorical structure, several key coherence diagrams must commute (up to natural isomorphism). These diagrams verify that the various ways of composing 1-morphisms and 2-morphisms yield equivalent results. Below are a couple of representative examples.

\paragraph{Triangle Coherence Diagram (Unitors):}  
In any bicategory or 2-category, the left and right unitors must satisfy the triangle identity. For any 1-morphisms \(f: A \to B\) and \(g: B \to C\), the following triangle must commute up to a natural isomorphism:
\[
\begin{tikzcd}[column sep=huge]
(f \circ \mathrm{id}_A) \circ g \arrow[rr, "a_{f,\mathrm{id}_A,g}"] \arrow[dr, "r_f \circ \mathrm{id}_g"'] & & f \circ (\mathrm{id}_A \circ g) \arrow[dl, "\mathrm{id}_f \circ l_g"] \\
& f \circ g &
\end{tikzcd}
\]
Here, \(a_{f,\mathrm{id}_A,g}\) denotes the associator 2-morphism, while \(r_f\) and \(l_g\) denote the right and left unitors, respectively. This diagram ensures that the unit constraints are compatible with the associativity of composition.

\paragraph{Square Coherence Diagram (Interchange Law):}  
Another fundamental coherence condition is the interchange law, which governs the interaction between vertical and horizontal compositions of 2-morphisms. For 2-morphisms \(\alpha, \alpha': f \Rightarrow f'\) and \(\beta, \beta': g \Rightarrow g'\), the following square illustrates that the horizontal composition of the vertical composites equals the vertical composite of the horizontal compositions:
\[
\begin{tikzcd}[column sep=huge, row sep=huge]
g \circ f \arrow[r, Rightarrow, "\alpha \ast \beta"] \arrow[d, Rightarrow, "\beta' \circ \alpha"'] & g' \circ f \arrow[d, Rightarrow, "\beta' \ast \alpha"] \\
g \circ f' \arrow[r, Rightarrow, "\beta \ast \alpha'"'] & g' \circ f'
\end{tikzcd}
\]
(Here, the double arrows represent the 2-morphisms and their compositions.) This diagram encapsulates the interchange law, a key condition ensuring that the two ways of composing 2-morphisms (horizontally then vertically, or vice versa) are coherently equivalent.

\paragraph{Discussion:}  
These coherence diagrams are not merely technical artifacts—they ensure that the universal properties, as well as the associativity and unit conditions of the integrated 2-category, are maintained in a flexible manner. By requiring these diagrams to commute up to natural isomorphism, the framework accommodates the inherent "weakness" of higher categorical structures while still preserving a robust and consistent semantic foundation.

\newpage
\subsection{Verification of Multiple Composition Paths}
\label{subsec:verification_multiple_paths}

Consider three composable 1-morphisms in a 2-category:
\[
f: A \to B,\quad g: B \to C,\quad h: C \to D.
\]
There are two natural ways to compose these morphisms:
\[
u = (h \circ g) \circ f \quad \text{and} \quad v = h \circ (g \circ f).
\]
In a strict 2-category these would be equal; however, in a weak setting they are related by the \emph{associator} 2-morphism
\[
a_{f,g,h}: (h \circ g) \circ f \Rightarrow h \circ (g \circ f).
\]
This associator is a natural isomorphism that ensures that all different composition routes are coherently equivalent.

A typical coherence diagram illustrating this equivalence is given by:
\[
\begin{tikzcd}[column sep=large, row sep=large]
(h \circ g) \circ f \arrow[rr, "a_{f,g,h}"] \arrow[dr, dashed, "\alpha"'] & & h \circ (g \circ f) \\
& h \circ (g \circ f) \arrow[ur, equal] &
\end{tikzcd}
\]
Here, the dashed arrow represents a (possibly trivial) 2-morphism showing that the two paths from \((h \circ g) \circ f\) to \(h \circ (g \circ f)\) are equivalent. In a fully coherent 2-category, such equivalences are required to satisfy additional coherence conditions (like the pentagon identity) ensuring that all higher compositions are compatible.

\paragraph{Interpretation:}  
The existence of the associator \(a_{f,g,h}\) and similar coherence 2-morphisms (such as unitors) demonstrates that even though there are multiple ways to compose 1-morphisms in a weak 2-category, they are all equivalent up to a unique, coherent natural isomorphism. This flexibility is essential for modeling complex structures where strict associativity or identity does not hold, while still maintaining overall consistency.

\subsection{Discussion on the Role of Natural Isomorphisms}
\label{subsec:role_nat_iso}

Natural isomorphisms (2-morphisms) play a pivotal role in ensuring the overall consistency of an integrated 2-categorical structure. In contrast to strict categories, where equations hold exactly, 2-categories and bicategories allow equations to hold up to a specified natural isomorphism. This flexibility is essential in capturing the true behavior of many mathematical constructions and logical connectives. The following points elaborate on their significance:

\paragraph{Ensuring Coherence:}  
Natural isomorphisms provide the necessary coherence data by linking different composition routes. For example, while the composition of 1-morphisms might not be strictly associative, the existence of associator 2-morphisms (and corresponding unitors) guarantees that all ways of composing a sequence of 1-morphisms are coherently isomorphic. This is formally encapsulated in coherence conditions such as the pentagon and triangle identities, which ensure that any diagram constructed using associators and unitors commutes up to a unique natural isomorphism.

\paragraph{Preserving Universal Properties:}  
When universal properties (such as those defining products, coproducts, and exponentials) are extended into the 2-categorical setting, the uniqueness conditions are replaced by uniqueness up to natural isomorphism. The natural isomorphisms ensure that even though the mediating morphisms may not be strictly unique, they are uniquely determined in a coherent manner. This preservation is crucial for maintaining the logical semantics associated with these constructions.

\paragraph{Facilitating Integration:}  
In the process of integrating local categories into a global 2-category, various structures (each with its own universal properties) must be combined. Natural isomorphisms serve as the "glue" that binds these disparate structures together, ensuring that the transitions between local and integrated levels respect the intended logical relationships. They allow for the flexible composition of morphisms from different local categories, while ensuring that the resulting integrated structure is consistent and well-behaved.

\paragraph{Abstract Consistency:}  
Overall, natural isomorphisms embody the notion that “equality” in higher category theory is replaced by a coherent equivalence. This shift from strict equality to equivalence up to isomorphism is not a loss of information but rather a more accurate reflection of many naturally occurring mathematical phenomena. In this way, the integrated structure remains robust and reliable, as all potential discrepancies in composition are systematically resolved via natural isomorphisms.

\noindent
In summary, natural isomorphisms are indispensable in the integrated 2-category, ensuring that the universal properties and coherence conditions are maintained across various layers of the structure. Their role is fundamental in achieving a consistent, flexible, and logically sound framework in higher category theory.

\newpage
\section{Case Studies and Examples}
\subsection{Integration of Product and Coproduct Categories}
\label{subsec:integration_prod_coproduct}

In our integrated 2-categorical framework, the local structures modeling product and coproduct operations are combined via pseudo-limit and pseudo-colimit constructions. Recall that in the product category, for any objects \(A\) and \(B\), the product \(A \times B\) is defined by the existence of projection 1-morphisms
\[
\pi_A: A \times B \to A \quad \text{and} \quad \pi_B: A \times B \to B,
\]
satisfying the universal property: for any object \(X\) with 1-morphisms \(f: X \to A\) and \(g: X \to B\), there exists a unique mediating 1-morphism \(\langle f, g \rangle: X \to A \times B\) (up to a coherent natural isomorphism) making
\[
\pi_A \circ \langle f, g \rangle \cong f \quad \text{and} \quad \pi_B \circ \langle f, g \rangle \cong g.
\]

Dually, in the coproduct category, the coproduct \(A+B\) is defined by injection 1-morphisms
\[
\iota_A: A \to A+B \quad \text{and} \quad \iota_B: B \to A+B,
\]
with the universal property that for any object \(X\) and any pair of 1-morphisms \(f: A \to X\) and \(g: B \to X\), there exists a unique 1-morphism \([f, g]: A+B \to X\) (again, up to a coherent natural isomorphism) such that
\[
[f, g] \circ \iota_A \cong f \quad \text{and} \quad [f, g] \circ \iota_B \cong g.
\]

\paragraph{Integration via Pseudo-Constructions:}  
To integrate these local structures into a unified 2-category \(\mathcal{I}\), we proceed as follows:
\begin{enumerate}
  \item \textbf{Object Formation:}  
  The objects of \(\mathcal{I}\) are taken to be the disjoint union of the objects from the product and coproduct categories.
  
  \item \textbf{Hom-Categories via Pseudo-Limits/Colimits:}  
  For any pair of objects in \(\mathcal{I}\), the hom-category is constructed using pseudo-limit (or pseudo-colimit) techniques, ensuring that the universal properties of products and coproducts are preserved up to a natural isomorphism.
  
  \item \textbf{Compatibility via 2-Morphisms:}  
  Additional 2-morphisms are introduced to relate the product and coproduct structures, guaranteeing that any morphism interacting with these constructions factors uniquely (up to coherent isomorphism) through the integrated structure.
\end{enumerate}

\paragraph{Illustrative Examples:}  
The following diagrams represent the universal properties of the product and coproduct constructions, respectively, in the integrated setting.

\medskip

\textbf{Product Diagram:}
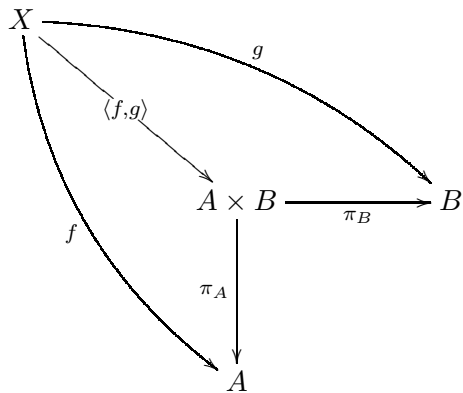
\[
\xymatrix@R=5em@C=5em{
X 
  \ar@/^1.5pc/[drr]^-{g} 
  \ar[dr]|-{\langle f, g \rangle} 
  \ar@/_1.5pc/[ddr]_-{f} & & \\
& A \times B \ar[r]_-{\pi_B} \ar[d]_-{\pi_A} & B \\
& A &
}
\]

\medskip

\textbf{Coproduct Diagram:}
\[
\begin{tikzcd}[column sep=large, row sep=large]
A \arrow[r, "\iota_A"] \arrow[dr, "f"'] & A+B \arrow[d, dashed, "{[f, g]}"] & B \arrow[l, "\iota_B"'] \arrow[dl, "g"] \\
& X &
\end{tikzcd}
\]

\paragraph{Conclusion:}  
By using pseudo-limits and pseudo-colimits, the integrated 2-category \(\mathcal{I}\) retains the universal properties of the product and coproduct constructions—essential for modeling logical conjunction and disjunction—while allowing the flexibility of “up-to-isomorphism” commutativity. The incorporation of natural isomorphisms (2-morphisms) ensures that the overall integrated structure is coherent, providing a robust foundation for further logical and categorical developments.

\subsection{Integration of the Exponential Category}
\label{subsec:integration_exponential}

In a cartesian closed category, the exponential object \(B^A\) together with the evaluation morphism 
\[
\mathrm{ev}: B^A \times A \to B
\]
encapsulates the notion of logical implication via the process of currying. When integrating exponential objects into the global 2-categorical framework, our goal is to preserve their universal property up to coherent natural isomorphism. The integration proceeds as follows:

\begin{enumerate}
  \item \textbf{Lifting to a 2-Categorical Structure:}  
  The local exponential category is enriched by promoting its hom-sets to hom-categories. In this setting, a morphism \(f: X \times A \to B\) has a unique (up to a natural isomorphism) currying \(\tilde{f}: X \to B^A\) such that:
  \[
  f = \mathrm{ev} \circ (\tilde{f} \times \mathrm{id}_A).
  \]
  This currying process is expressed via a natural isomorphism, which is incorporated as a 2-morphism in the integrated structure.
  
  \item \textbf{Preservation of Evaluation:}  
  The evaluation map \(\mathrm{ev}\) is lifted as a 1-morphism in the integrated 2-category. Its role in “applying” a function to an argument is maintained by ensuring that for every \(f: X \times A \to B\) the corresponding diagram
  \[
  \begin{tikzcd}[column sep=large, row sep=large]
  X \times A \arrow[r, "f"] \arrow[d, "\tilde{f}\times \mathrm{id}_A"'] & B \\
  B^A \times A \arrow[ur, "\mathrm{ev}"'] &
  \end{tikzcd}
  \]
  commutes up to a coherent 2-morphism. This guarantees that the evaluation operation remains consistent within the integrated setting.
  
  \item \textbf{Coherence via Pseudo-Limits:}  
  The universal property of exponentials is maintained using pseudo-limits. The pseudo-limit construction allows the currying correspondence to hold up to natural isomorphism, ensuring that any alternative way of factorizing a morphism \(f\) through \(B^A\) is coherently equivalent to \(\tilde{f}\).  
\end{enumerate}

Thus, the integrated exponential category successfully carries over the currying process and evaluation map from the local setting to the global 2-category, preserving the universal properties of exponential objects and enabling a categorical treatment of logical implication in a flexible, coherent framework.

\newpage
\subsection{Comparative Analysis}
\label{subsec:comparative_analysis}

The integration of local categories into a unified 2-categorical framework offers several significant advantages over traditional 1-categorical approaches. Below is a comparative analysis highlighting these benefits:

\paragraph{Flexibility of Universal Properties:}  
In traditional 1-categories, universal properties (such as those of products, coproducts, and exponentials) are required to hold strictly—diagrams must commute exactly. In contrast, the 2-category framework allows these universal properties to hold \emph{up to a natural isomorphism}. This flexibility is crucial because many natural constructions in mathematics and logic only satisfy their universal conditions up to isomorphism, not strict equality.

\paragraph{Enhanced Coherence and Consistency:}  
1-categorical models lack the structure to adequately manage coherence issues arising from multiple composition paths. The introduction of 2-morphisms in 2-categories provides the necessary coherence data (e.g., associators and unitors) that ensure all different compositional routes yield equivalent outcomes. This ensures that the integrated structure remains consistent even when strict commutativity is relaxed.

\paragraph{Modularity in Integration:}  
The 2-categorical approach facilitates the integration of diverse local categories—each modeling different logical connectives—by using pseudo-limits and pseudo-colimits. These constructions allow for a modular assembly where the local universal properties are preserved and coherently interrelated. Traditional methods often struggle to accommodate such modularity without resorting to ad hoc or overly rigid constructions.

\paragraph{Simplified Reasoning and Strictification:}  
Once integrated, the resulting 2-category can be further simplified via strictification, transforming the weak structure into a strict 2-category without losing essential properties. This process greatly simplifies subsequent reasoning and proof construction, which is particularly beneficial when handling complex logical semantics. In a 1-categorical approach, such simplification is typically unavailable or forced by unnatural constraints.

\paragraph{Summary of Advantages:}  
Overall, the 2-category framework:
\begin{itemize}
  \item Captures universal properties more naturally by allowing commutativity up to isomorphism.
  \item Provides explicit coherence data through 2-morphisms that ensure consistency across multiple composition paths.
  \item Offers a modular and flexible method for integrating local structures, making it well-suited for modeling complex logical connectives.
  \item Enables further simplification through strictification, thus streamlining reasoning and proof processes.
\end{itemize}

These advantages illustrate why the 2-categorical integration of local categories is a powerful and more faithful approach to modeling logical semantics compared to traditional 1-categorical methods.

\newpage
\section{Summary and Transition}
\label{sec:summary_transition}

In this chapter, we have detailed the process of integrating local categories into a unified 2-categorical framework. Key points include:

\begin{itemize}
  \item The \textbf{lifting of local categories}—each capturing logical connectives such as negation, product, coproduct, and exponential—into the 2-categorical setting by enriching them with 2-morphisms.
  \item The use of \textbf{pseudo-limits and pseudo-colimits} to preserve universal properties in a flexible manner, ensuring that diagrams commute up to natural isomorphism.
  \item The establishment of \textbf{coherence conditions} via natural isomorphisms (associators, unitors, etc.), which guarantee that multiple composition routes yield equivalent results.
  \item The application of \textbf{bifunctors} to map and integrate the local categories into the global 2-category, preserving both structure and universal properties.
\end{itemize}

These steps ensure that the integrated structure is both consistent and faithful to the logical semantics encoded in the local categories. In the next chapter, we will explore strictification techniques and further evaluate the integrated structure, setting the stage for streamlined reasoning and more robust categorical logic.

\chapter{Coherence Verification in the Integrated Category}
\section{Introduction to Coherence Verification}
\subsection{Motivation and Importance}
\label{subsec:coherence_motivation}

In any integrated 2-categorical framework, ensuring coherence is essential for the following reasons:

\begin{itemize}
  \item \textbf{Consistency of Composition:}  
  In a 2-category, 1-morphisms and 2-morphisms can be composed in multiple ways. Coherence conditions guarantee that all these different composition routes lead to the same result up to a natural isomorphism. Without such conditions, the integrated structure could become ambiguous, undermining the reliability of categorical reasoning.

  \item \textbf{Preservation of Universal Properties:}  
  Universal properties in local categories (such as those of products, coproducts, and exponentials) often hold only up to isomorphism. Coherence ensures that these properties are preserved consistently when local categories are integrated into a global 2-category, allowing us to reason about logical connectives and other constructions in a robust manner.

  \item \textbf{Flexibility in Weak Structures:}  
  In many higher categorical settings, equations are replaced by isomorphisms. Coherence plays a critical role in managing these isomorphisms so that, despite the inherent flexibility, the overall structure behaves in a predictable and controllable way.

  \item \textbf{Facilitating Strictification:}  
  Coherence verification is a prerequisite for strictification procedures, which convert weak 2-categorical structures into strict ones without losing essential properties. This conversion simplifies further constructions and proofs while maintaining the intended semantic content.
\end{itemize}

By ensuring that all necessary diagrams commute up to natural isomorphism, the process of coherence verification underpins the reliability and consistency of the integrated 2-categorical framework. This foundational step is critical for the subsequent stages of strictification and further evaluation of the logical semantics embodied in the integrated category.

\subsection{Overview of Coherence Conditions}
\label{subsec:overview_coherence_conditions}

In the integrated 2-category, several key coherence diagrams must commute (up to natural isomorphism) to ensure that the various ways of composing 1-morphisms and 2-morphisms yield consistent results. Below is an overview of the principal diagrams that are verified in this framework:

\begin{itemize}
  \item \textbf{Triangle Diagram:}  
  This diagram ensures that the unitors interact correctly with the associator. For any 1-morphism \(f: A \to B\), the left and right unitors \(l_f\) and \(r_f\) and the associator \(a\) satisfy the triangle identity:
  \[
  \begin{tikzcd}[column sep=huge]
  f \circ \mathrm{id}_A \arrow[rr, "a_{f,\mathrm{id}_A,\mathrm{id}_A}"] \arrow[dr, "r_f"'] & & f \arrow[dl, "l_f"] \\
  & f &
  \end{tikzcd}
  \]

  \item \textbf{Square Diagram (Interchange Law):}  
  The square diagram captures the interplay between vertical and horizontal compositions of 2-morphisms. It expresses that the horizontal composite of vertical compositions is equal (up to a natural isomorphism) to the vertical composite of horizontal compositions:
  \[
  \begin{tikzcd}[column sep=huge, row sep=huge]
  g \circ f \arrow[r, Rightarrow, "\alpha \ast \beta"] \arrow[d, Rightarrow, "\alpha' \circ \beta"'] & g' \circ f \arrow[d, Rightarrow, "\alpha' \ast \beta"] \\
  g \circ f' \arrow[r, Rightarrow, "\alpha \ast \beta'"'] & g' \circ f'
  \end{tikzcd}
  \]
  (Here, \(\alpha, \alpha'\) and \(\beta, \beta'\) denote appropriate 2-morphisms.)

  \item \textbf{Pentagon Diagram:}  
  For any four composable 1-morphisms
  \[
  f: A \to B,\quad g: B \to C,\quad h: C \to D,\quad k: D \to E,
  \]
  the associators must satisfy the pentagon coherence condition, ensuring that the following diagram commutes up to a unique natural isomorphism:
  \[
  \begin{tikzcd}[column sep=huge, row sep=huge]
  ((k \circ h) \circ g) \circ f \arrow[r, "a_{k,h,g}\circ \mathrm{id}_f"] \arrow[d, "a_{k\circ h,g,f}"'] & (k \circ (h \circ g)) \circ f \arrow[r, "a_{k,h\circ g,f}"] & k \circ ((h \circ g) \circ f) \arrow[d, "\mathrm{id}_k \circ a_{h,g,f}"] \\
  (k \circ h) \circ (g \circ f) \arrow[rr, "a_{k,h,g\circ f}"'] & & k \circ (h \circ (g \circ f))
  \end{tikzcd}
  \]
\end{itemize}

\noindent
These coherence diagrams are fundamental to the integrated structure, as they ensure that all alternative ways of composing morphisms and 2-morphisms are equivalent up to a coherent natural isomorphism. This guarantees that the integrated 2-category is consistent and that its universal properties are preserved in a flexible, yet controlled, manner.

\newpage
\section{Construction of Coherence Diagrams}
\subsection{Diagrammatic Representation of Universal Properties}
\label{subsec:diagrammatic_representation}

In order to verify that universal properties are preserved in the integrated 2-categorical setting, we construct coherence diagrams that express these properties. Below, we outline the general method for building such diagrams for the product, coproduct, and exponential constructions.

\paragraph{1. Product Universal Property:}
For any objects \(A\) and \(B\) in a local category, the product \(A \times B\) is defined by the existence of projection morphisms:
\[
\pi_A: A \times B \to A,\quad \pi_B: A \times B \to B.
\]
The universal property states that for any object \(X\) with morphisms \(f: X \to A\) and \(g: X \to B\), there exists a unique mediating morphism \(\langle f, g \rangle: X \to A \times B\) such that:
\[
\pi_A \circ \langle f, g \rangle = f,\quad \pi_B \circ \langle f, g \rangle = g.
\]
This is represented diagrammatically as follows:
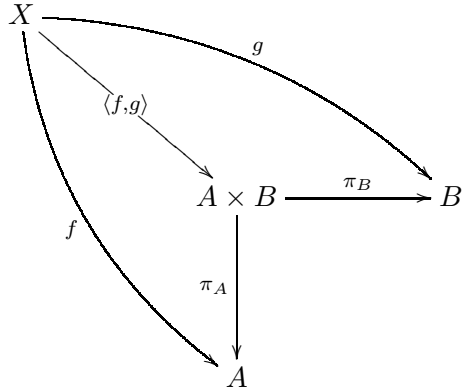
\[
\xymatrix@R=5em@C=5em{
X 
  \ar@/^1.5pc/[drr]^-{g} 
  \ar[dr]|-{\langle f, g \rangle} 
  \ar@/_1.5pc/[ddr]_-{f} & & \\
& A \times B \ar[r]^-{\pi_B} \ar[d]_-{\pi_A} & B \\
& A &
}
\]

\paragraph{2. Coproduct Universal Property:}
Dually, the coproduct \(A+B\) comes with injection morphisms:
\[
\iota_A: A \to A+B,\quad \iota_B: B \to A+B.
\]
For any object \(X\) with morphisms \(f: A \to X\) and \(g: B \to X\), there is a unique mediating morphism \([f, g]: A+B \to X\) satisfying:
\[
[f, g] \circ \iota_A = f,\quad [f, g] \circ \iota_B = g.
\]
Its diagrammatic representation is:
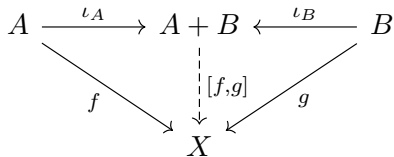
\[
\begin{tikzcd}[column sep=large, row sep=large]
A \arrow[r, "\iota_A"] \arrow[dr, "f"'] & A+B \arrow[d, dashed, "{[f, g]}"] & B \arrow[l, "\iota_B"'] \arrow[dl, "g"] \\
& X &
\end{tikzcd}
\]

\paragraph{3. Exponential Universal Property (Currying):}
In a cartesian closed category, the exponential object \(B^A\) is defined together with an evaluation morphism:
\[
\mathrm{ev}: B^A \times A \to B.
\]
For any object \(X\) and any morphism \(f: X \times A \to B\), there exists a unique (up to natural isomorphism) morphism \(\tilde{f}: X \to B^A\) (the curried form of \(f\)) such that:
\[
f = \mathrm{ev} \circ (\tilde{f} \times \mathrm{id}_A).
\]
The corresponding diagram is:
\[
\begin{tikzcd}[column sep=large, row sep=large]
X \times A \arrow[r, "f"] \arrow[d, "\tilde{f}\times\mathrm{id}_A"'] & B \\
B^A \times A \arrow[ur, "\mathrm{ev}"'] &
\end{tikzcd}
\]

\paragraph{Methodology for Constructing Coherence Diagrams:}
\begin{enumerate}
  \item \textbf{Identify the Universal Property:}  
  Determine the required projections/injections and the unique factorization morphism that characterizes the construction (product, coproduct, or exponential).

  \item \textbf{Construct the Basic Diagram:}  
  Draw the standard commutative diagram representing the universal property, ensuring that all arrows (morphisms) and nodes (objects) are clearly labeled.

  \item \textbf{Incorporate 2-Morphisms:}  
  In the 2-categorical setting, modify the diagram to indicate that the required equalities hold only up to natural isomorphism. Replace strict equalities with commutative diagrams that include 2-morphisms (often depicted as double arrows or labeled isomorphisms).

  \item \textbf{Verify Coherence:}  
  Ensure that additional coherence diagrams (such as triangle, square, or pentagon diagrams) are constructed to show that all different composition paths are coherently isomorphic. These diagrams verify that the universal properties extend to the integrated 2-category.
\end{enumerate}

\noindent
This approach provides a systematic way to represent and verify the universal properties of local categories within the integrated framework, ensuring that all structures—though weakened to hold up to isomorphism—remain consistent and robust.

\subsection{Typical Coherence Diagrams in the Integrated Category}
\label{subsec:typical_coherence_diagrams}

A central aspect of the integrated 2-categorical framework is ensuring that different composition paths yield results that are equivalent up to a coherent natural isomorphism. The following diagrams are typical examples of the coherence conditions that must be verified.

\paragraph{Pentagon Diagram (Associativity Coherence):}  
For any four composable 1-morphisms
\[
f: A \to B,\quad g: B \to C,\quad h: C \to D,\quad k: D \to E,
\]
the associators provide the following coherence, which is captured by the pentagon diagram:
\[
\begin{tikzcd}[column sep=huge, row sep=huge]
((k \circ h) \circ g) \circ f 
  \arrow[r, "{a_{k,h,g}\circ \mathrm{id}_f}"]
  \arrow[d, "a_{k\circ h,g,f}"'] 
  & (k \circ (h \circ g)) \circ f 
    \arrow[r, "a_{k,h\circ g,f}"]
  & k \circ ((h \circ g) \circ f) 
    \arrow[d, "{\mathrm{id}_k\circ a_{h,g,f}}"] \\
(k \circ h) \circ (g \circ f) 
  \arrow[rr, "a_{k,h,g\circ f}"']
  & 
  & k \circ (h \circ (g \circ f))
\end{tikzcd}
\]
This diagram ensures that all ways of associatively composing the four 1-morphisms are coherently isomorphic.

\paragraph{Square Diagram (Interchange Law):}  
Another fundamental coherence condition is the interchange law, which governs the interaction between vertical and horizontal compositions of 2-morphisms. For 2-morphisms \(\alpha: f \Rightarrow f'\) and \(\beta: g \Rightarrow g'\), the following square must commute (up to a specified natural isomorphism):
\[
\begin{tikzcd}[column sep=huge, row sep=huge]
g \circ f \arrow[r, Rightarrow, "\alpha \ast \beta"] \arrow[d, Rightarrow, "\beta' \circ \alpha"'] 
  & g' \circ f \arrow[d, Rightarrow, "\beta' \ast \alpha"] \\
g \circ f' \arrow[r, Rightarrow, "\alpha \ast \beta'"'] 
  & g' \circ f'
\end{tikzcd}
\]
Here, the double arrows represent 2-morphisms, and the diagram expresses that horizontal composition distributes over vertical composition in a coherent manner.

\paragraph{Discussion:}  
These diagrams are instrumental in verifying that the integrated 2-category is well-behaved. The pentagon diagram confirms that different ways of associatively composing 1-morphisms lead to naturally isomorphic results, while the square diagram ensures that the interplay between vertical and horizontal compositions of 2-morphisms is consistent. Together, they form the backbone of the coherence verification in the integrated framework, guaranteeing that all composite morphisms behave predictably up to coherent natural isomorphism.

\subsection{Role of 2-Morphisms in Diagram Commutation}
\label{subsec:role_2morphisms}

In an integrated 2-category, 2-morphisms (or natural isomorphisms) are essential for ensuring that diagrams commute, not strictly, but up to coherent isomorphism. This flexible notion of commutativity allows for the reconciliation of different composition routes that would otherwise yield different outcomes in a strictly 1-categorical setting.

For example, consider two distinct ways of composing a series of 1-morphisms \(f: A \to B\), \(g: B \to C\), and \(h: C \to D\):
\[
u = (h \circ g) \circ f \quad \text{and} \quad v = h \circ (g \circ f).
\]
In a 2-category, these two composites are not necessarily equal, but there exists an associator 2-morphism
\[
a_{f,g,h}: (h \circ g) \circ f \Rightarrow h \circ (g \circ f),
\]
which serves as the bridge between the two. This natural isomorphism guarantees that the diagram relating \(u\) and \(v\) commutes in the weak sense:

\[
\begin{tikzcd}[column sep=huge, row sep=huge]
((h \circ g) \circ f) \arrow[r, "a_{f,g,h}"] \arrow[dr, dashed, "\alpha"'] & h \circ (g \circ f) \\
& h \circ (g \circ f) \arrow[u, equal]
\end{tikzcd}
\]
In the above diagram, the dashed arrow represents any alternative 2-morphism that might arise from different choices in the composition process, and the associator \(a_{f,g,h}\) confirms that all such paths are coherently equivalent.

Similarly, for other constructions—such as those defining products, coproducts, and exponentials—the corresponding universal diagrams commute up to specified natural isomorphisms. These 2-morphisms ensure that the integrated structure respects the intended universal properties despite the inherent flexibility of weak compositions.

In summary, natural isomorphisms (2-morphisms) play a critical role in the integrated 2-categorical framework by:
\begin{itemize}
  \item Ensuring that all different composition paths yield equivalent results, thereby maintaining overall consistency.
  \item Providing the necessary coherence data that permits the weakening of strict commutativity without loss of essential structural properties.
  \item Enabling the integration of local categories into a global framework where universal properties hold up to isomorphism, thereby faithfully representing logical constructs.
\end{itemize}

\newpage
\section{Verification of Coherence Conditions}
\subsection{Analysis of Multiple Composition Paths}
\label{subsec:analysis_multiple_paths}

A key aspect of coherence verification in an integrated 2-category is to show that different ways of composing a sequence of 1-morphisms yield results that are equivalent up to a unique 2-isomorphism. To illustrate this, consider a situation with three composable 1-morphisms:
\[
f: A \to B,\quad g: B \to C,\quad h: C \to D.
\]
There are two natural composite 1-morphisms:
\[
u = (h \circ g) \circ f \quad \text{and} \quad v = h \circ (g \circ f).
\]
In a weak 2-category, these two composites are not necessarily equal, but they are related by the associator 2-morphism:
\[
a_{f,g,h}: (h \circ g) \circ f \xRightarrow{\sim} h \circ (g \circ f).
\]

The verification of coherence requires demonstrating that for any additional ways to decompose the composition (for instance, when further compositions are involved), all such paths are connected by canonical 2-isomorphisms. A typical method involves constructing a commutative diagram, such as the well-known pentagon diagram, which ensures that the various associators satisfy the pentagon identity.

For the simpler case of three composable morphisms, the coherence can be visualized by the following diagram:

\[
\begin{tikzcd}[column sep=huge, row sep=huge]
((h \circ g) \circ f) \arrow[rr, "a_{f,g,h}"] \arrow[dr, dashed, "\alpha"'] & & h \circ (g \circ f) \\
& h \circ (g \circ f) \arrow[ur, equal] &
\end{tikzcd}
\]

In this diagram:
\begin{itemize}
  \item The top arrow \(a_{f,g,h}\) is the associator, a canonical 2-morphism that identifies the two composite ways of combining \(f\), \(g\), and \(h\).
  \item The dashed arrow \(\alpha\) represents any alternative 2-morphism obtained from a different sequence of compositions or coherence adjustments.
  \item The equality at the right indicates that, once the coherence data are fully accounted for, both composition paths yield the same effective result.
\end{itemize}

\paragraph{Verification Method:}
To verify that multiple composition paths agree up to a 2-isomorphism, one typically follows these steps:
\begin{enumerate}
  \item \textbf{Identify all possible composition paths:} List out the distinct ways to compose a given sequence of 1-morphisms.
  \item \textbf{Construct coherence diagrams:} Build diagrams (such as pentagon, triangle, or square diagrams) that relate these paths via the available 2-morphisms (associators, unitors, etc.).
  \item \textbf{Check the commutativity up to isomorphism:} Verify that the composite 2-morphisms obtained by following different paths in the diagram are themselves naturally isomorphic. This is usually done by applying known coherence theorems (e.g., Mac Lane's Coherence Theorem) which guarantee that, under the prescribed conditions, all such diagrams commute.
\end{enumerate}

Through this systematic approach, one can demonstrate that even though the compositions in a weak 2-category do not strictly equal one another, they are equivalent in a coherent and consistent manner, thereby ensuring the overall structural integrity of the integrated 2-categorical framework.

\subsection{Formal Proof Strategies}
\label{subsec:formal_proof_strategies}

To rigorously verify the commutativity of coherence diagrams in an integrated 2-categorical framework, several formal proof strategies are employed. These techniques ensure that all alternative composition paths yield results that are equivalent up to a unique natural isomorphism. Key strategies include:

\begin{enumerate}
  \item \textbf{Diagram Chasing:}  
  This classical method involves "chasing" elements or morphisms through the diagram to verify that all composite 2-morphisms yield the same outcome. For instance, one may explicitly compare the composite 2-morphisms in a pentagon diagram to confirm that they are coherently isomorphic.

  \item \textbf{String Diagram Calculus:}  
  String diagrams provide a visual and intuitive way to represent 2-morphisms and their compositions. By translating algebraic expressions into graphical forms, one can often observe the equivalence of different composition paths directly. This approach is particularly useful for complex diagrams, as it simplifies the verification of coherence conditions.

  \item \textbf{Application of Coherence Theorems:}  
  Coherence theorems (such as Mac Lane's Coherence Theorem for monoidal categories and its extensions to bicategories) guarantee that all diagrams constructed from the basic coherence isomorphisms (associators, unitors, etc.) commute. By appealing to these theorems, one can reduce the verification of an entire family of coherence conditions to checking a finite set of key diagrams.

  \item \textbf{Formal Rewriting Systems:}  
  Another rigorous approach is to employ formal rewriting systems, where equations corresponding to the coherence conditions are manipulated according to well-defined rewriting rules. This method provides an algorithmic way to verify that any two parallel 2-morphisms are isomorphic, ensuring overall consistency of the structure.
\end{enumerate}

\noindent
By combining these techniques, one obtains a robust framework for verifying the commutativity of coherence diagrams in the integrated 2-category. This rigorous verification is essential to ensure that the universal properties and structural relationships are preserved across the entire categorical system.

\subsection{Use of Natural Isomorphisms}
\label{subsec:use_nat_iso}

Natural isomorphisms (2-morphisms) play a central role in ensuring the uniqueness and overall consistency of constructions in an integrated 2-category. Instead of requiring strict equality between composite morphisms, we allow them to be isomorphic via natural isomorphisms. This “up-to-isomorphism” approach guarantees that different composition paths lead to essentially the same result.

\paragraph{Guaranteeing Uniqueness:}  
Consider the universal property of the product in a 2-category. For objects \(A\) and \(B\), the product \(A \times B\) is defined by projection 1-morphisms \(\pi_A\) and \(\pi_B\). For any object \(X\) and any pair of 1-morphisms \(f: X \to A\) and \(g: X \to B\), there exists a unique mediating 1-morphism \(\langle f, g \rangle: X \to A \times B\) \emph{up to a natural isomorphism} \(\theta\) satisfying:
\[
\pi_A \circ \langle f, g \rangle \cong f \quad \text{and} \quad \pi_B \circ \langle f, g \rangle \cong g.
\]
The natural isomorphism \(\theta\) ensures that even if there are different choices for the mediating morphism, they are uniquely related by \(\theta\), thereby preserving uniqueness in the 2-categorical sense.

\paragraph{Ensuring Coherence:}  
In a weak 2-category, various composition paths for a sequence of 1-morphisms do not strictly equal each other. Instead, they are connected by associator and unitor isomorphisms. For example, given three composable 1-morphisms
\[
f: A \to B,\quad g: B \to C,\quad h: C \to D,
\]
there are two natural composites:
\[
(h \circ g) \circ f \quad \text{and} \quad h \circ (g \circ f).
\]
The associator 2-morphism
\[
a_{f,g,h}: (h \circ g) \circ f \xRightarrow{\sim} h \circ (g \circ f)
\]
guarantees that both composites are equivalent. This coherence is crucial, as it ensures that any diagram involving multiple compositions commutes up to a canonical natural isomorphism.

\paragraph{Concrete Example in \(\mathbf{Cat}\):}  
Consider the 2-category \(\mathbf{Cat}\), where:
\begin{itemize}
  \item Objects are small categories.
  \item 1-morphisms are functors.
  \item 2-morphisms are natural transformations.
\end{itemize}
For two functors \(F, G: \mathcal{A} \to \mathcal{B}\), a natural transformation \(\eta: F \Rightarrow G\) provides a family of isomorphisms (if \(\eta\) is a natural isomorphism) that adjust the difference between \(F\) and \(G\). If one constructs the product of two categories (with the usual projection functors) and then considers different mediating functors arising from various choices of component functions, any two such mediating functors will be naturally isomorphic. This natural isomorphism guarantees that the product is unique up to isomorphism, and all coherence diagrams involving these functors commute up to a natural transformation.

\paragraph{Summary:}  
Natural isomorphisms ensure that even though compositions in a 2-category might not be strictly equal, they are equivalent in a well-controlled and coherent way. By providing the necessary corrections, these 2-morphisms maintain the uniqueness of factorization in universal properties and ensure that all possible composition paths yield consistent, commutative diagrams. This mechanism is key to the flexibility and robustness of the integrated 2-categorical framework.

\newpage
\section{Case Studies and Examples}
\subsection{Coherence in Product and Coproduct Structures}
\label{subsec:coherence_prod_coproduct}

In this subsection, we present case studies and examples that verify the coherence of product and coproduct structures in the integrated 2-category. Our goal is to demonstrate that the various composition paths, arising from the universal properties of products and coproducts, commute up to natural isomorphism.

\paragraph{Coherence in Product Structures:}
Consider the product \(A \times B\) in a local category. By definition, for any object \(X\) and morphisms 
\[
f: X \to A \quad \text{and} \quad g: X \to B,
\]
there exists a unique mediating morphism \(\langle f, g \rangle: X \to A \times B\) such that the following diagram commutes (up to a natural isomorphism in the 2-categorical setting):
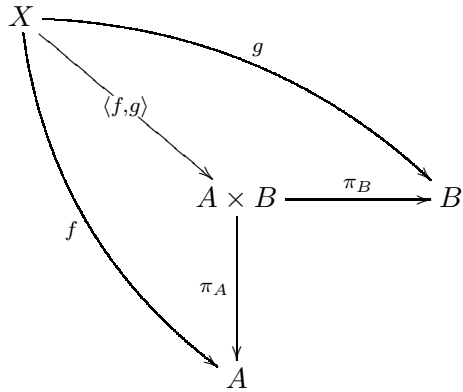
\[
\xymatrix@R=5em@C=5em{
X 
  \ar@/^1.5pc/[drr]^-{g} 
  \ar[dr]|-{\langle f, g \rangle} 
  \ar@/_1.5pc/[ddr]_-{f} & & \\
& A \times B \ar[r]^-{\pi_B} \ar[d]_-{\pi_A} & B \\
& A &
}
\]
In an integrated 2-category, the commutativity is not strict; instead, there exists a natural isomorphism \(\theta\) such that:
\[
\pi_A \circ \langle f, g \rangle \cong f \quad \text{and} \quad \pi_B \circ \langle f, g \rangle \cong g.
\]
This coherence data guarantees that even if different mediating morphisms are obtained (e.g., through different decomposition paths), they are uniquely isomorphic.

\paragraph{Coherence in Coproduct Structures:}
Dually, consider the coproduct \(A+B\) with injection morphisms \(\iota_A: A \to A+B\) and \(\iota_B: B \to A+B\). For any object \(X\) and any pair of morphisms
\[
f: A \to X \quad \text{and} \quad g: B \to X,
\]
the universal property guarantees the existence of a unique mediating morphism \([f, g]: A+B \to X\) (up to natural isomorphism) making the diagram
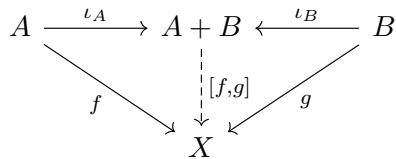
\[
\begin{tikzcd}[column sep=large, row sep=large]
A \arrow[r, "\iota_A"] \arrow[dr, "f"'] & A+B \arrow[d, dashed, "{[f, g]}"] & B \arrow[l, "\iota_B"'] \arrow[dl, "g"] \\
& X &
\end{tikzcd}
\]
commute in the weak sense:
\[
[f, g] \circ \iota_A \cong f \quad \text{and} \quad [f, g] \circ \iota_B \cong g.
\]
Again, the natural isomorphisms ensure that any two ways of obtaining the mediating morphism are coherently equivalent.

\paragraph{Verification Strategy:}
In both cases, the verification of coherence involves:
\begin{itemize}
  \item Identifying all potential composition routes (for example, when further composing with additional morphisms).
  \item Constructing commutative diagrams (triangles, squares, etc.) where the difference between routes is mediated by natural isomorphisms.
  \item Applying coherence theorems, which guarantee that all such diagrams commute up to a unique isomorphism.
\end{itemize}

These examples illustrate that, in an integrated 2-category, the universal properties of product and coproduct constructions are preserved in a flexible manner, with natural isomorphisms ensuring the overall consistency and coherence of the structure.

\subsection{Coherence in Exponential Structures}
\label{subsec:coherence_exponential_structures}

In a cartesian closed 2-category, the exponential object \(B^A\) and the evaluation morphism
\[
\mathrm{ev}: B^A \times A \to B
\]
capture the notion of function spaces and logical implication via currying. However, in this enriched setting, the universal property of exponentials holds only up to a coherent natural isomorphism. This section details how the coherence of the exponential structure is verified through diagrammatic representations and concrete examples.

\paragraph{Currying and Evaluation:}
Given any object \(X\) and a 1-morphism
\[
f: X \times A \to B,
\]
the universal property of the exponential object asserts that there exists a unique (up to natural isomorphism) 1-morphism \(\tilde{f}: X \to B^A\) (the curried form of \(f\)) such that the following diagram commutes up to a specified 2-morphism:
\[
\begin{tikzcd}[column sep=large, row sep=large]
X \times A \arrow[r, "f"] \arrow[d, "\tilde{f}\times\mathrm{id}_A"'] & B \\
B^A \times A \arrow[ur, "\mathrm{ev}"'] &
\end{tikzcd}
\]
The natural isomorphism ensures that any two choices of \(\tilde{f}\) are uniquely isomorphic, thus preserving the universal property in a flexible, yet coherent, manner.

\paragraph{Example in \(\mathbf{Set}\):}
In the category \(\mathbf{Set}\) (a cartesian closed category), the exponential object \(B^A\) is the set of all functions from \(A\) to \(B\), and the evaluation map is defined by
\[
\mathrm{ev}(h, a) = h(a) \quad \text{for } h \in B^A \text{ and } a \in A.
\]
Given a function \(f: X \times A \to B\), currying produces a function \(\tilde{f}: X \to B^A\) defined by:
\[
\tilde{f}(x)(a) = f(x, a).
\]
Although in \(\mathbf{Set}\) this correspondence is strictly unique, in a general 2-categorical setting, the uniqueness is only up to a natural isomorphism. This natural isomorphism (a 2-morphism) plays a crucial role in ensuring that different ways of factoring \(f\) through \(B^A\) are coherently equivalent.

\paragraph{Verifying Coherence:}
To verify coherence in exponential structures, the following steps are taken:
\begin{enumerate}
  \item \textbf{Diagram Construction:}  
  Construct the diagram representing the currying process, as shown above. This diagram must commute up to a natural isomorphism.
  
  \item \textbf{Identification of Natural Isomorphisms:}  
  Identify the 2-morphisms that relate any alternative choices for the mediating morphism \(\tilde{f}\). These isomorphisms provide the necessary corrections ensuring that the evaluation map behaves uniformly.
  
  \item \textbf{Coherence Checks:}  
  Verify that the natural isomorphisms satisfy the expected coherence conditions, such as compatibility with associators and unitors in the 2-category. In practice, this involves checking that composite diagrams (which may involve further compositions with other morphisms) commute up to a unique natural isomorphism.
\end{enumerate}

\paragraph{Conclusion:}
The use of natural isomorphisms in the context of exponential structures ensures that the currying process and evaluation map are coherent, even when the underlying diagrams do not strictly commute. This flexibility is essential for the integrated 2-categorical framework, as it allows the universal properties of exponentials—and, by extension, logical implication—to be preserved in a robust and consistent manner.

\subsection{Comparative Analysis with Traditional Approaches}
\label{subsec:comparative_analysis_traditional}

The integrated 2-categorical framework provides several advantages over traditional 1-categorical methods. Below is a comparative analysis highlighting the key differences and benefits:

\begin{itemize}
  \item \textbf{Flexibility in Universal Properties:}  
  In 1-categories, universal properties are defined by strict commutativity of diagrams. This rigidity can be problematic when natural constructions only satisfy their universal conditions up to isomorphism. In contrast, the 2-categorical approach allows universal properties to hold \emph{up to natural isomorphism}, thereby accommodating a wider range of mathematical phenomena.

  \item \textbf{Handling of Coherence:}  
  Traditional 1-categorical frameworks lack the means to manage coherence when different composition paths yield non-identical outcomes. The 2-categorical setting introduces 2-morphisms (natural isomorphisms) that systematically relate different composition routes, ensuring that all diagrams commute up to a coherent isomorphism. This leads to a more robust and reliable integrated structure.

  \item \textbf{Modularity in Integration:}  
  The use of pseudo-limits and pseudo-colimits in 2-categories enables a modular integration of local categories. Each local category, whether modeling products, coproducts, or exponentials, contributes its universal properties in a way that is preserved in the global structure. Traditional approaches often struggle to integrate such structures without imposing unnatural restrictions.

  \item \textbf{Simplification via Strictification:}  
  Once integrated, the weak 2-categorical structure can be further simplified by strictification—converting the structure into a strict 2-category without losing the essential universal properties. This step, which is generally not available in 1-categorical methods, facilitates easier reasoning and proof construction.
\end{itemize}

\noindent
In summary, the 2-categorical framework not only overcomes the limitations of strict equality in 1-categories but also provides a natural setting for the flexible and coherent treatment of universal properties. This approach yields a more faithful representation of logical semantics and a robust structure for integrating diverse local categories.

\newpage
\subsection{Challenges Encountered and Resolutions}
\label{subsec:challenges_resolutions}

In the process of verifying and developing the integrated 2-categorical framework, we encountered several technical and conceptual challenges. Below is a summary of the main issues along with the strategies we adopted to resolve them:

\begin{enumerate}
  \item \textbf{Fragile Commands in Moving Arguments:}  
  When using fragile commands—such as \verb|\langle| and \verb|\rangle|—within \texttt{tikz-cd} labels, we experienced errors \\(e.g., “\verb|! Missing \endcsname inserted.|”).  \\
  \textbf{Resolution:}  
  We addressed this issue by enclosing such commands in braces and using protective macros like \verb|\ensuremath| or \verb|\protect|. For instance, replacing \\ \verb|"$\langle f, g \rangle$"| with \verb|"{\ensuremath{\langle f, g \rangle}}"|\\ ensured correct expansion in moving arguments.
  
  \item \textbf{Empty Cells Leading to Missing Nodes:}  
  In the \texttt{tikz-cd} environment, leaving cells empty sometimes caused errors due to the absence of a node, which interfered with arrow placement.  
  \textbf{Resolution:}  
  We mitigated this problem by inserting empty groups (\verb|{}|) in cells that would otherwise be empty, ensuring that every cell contains a node and that all arrows have properly defined sources and targets.
  
  \item \textbf{Verifying Coherence in Multiple Composition Paths:}  
  Since different sequences of composing 1-morphisms can lead to results that are only isomorphic rather than strictly equal, it was challenging to verify that all these alternative paths are coherently related.  
  \textbf{Resolution:}  
  We combined techniques such as diagram chasing and string diagram calculus, and applied established coherence theorems (for example, Mac Lane's Coherence Theorem) to demonstrate that all such diagrams commute up to a unique natural isomorphism.
  
  \item \textbf{Integrating Diverse Local Structures:}  
  Local categories modeling various logical connectives (e.g., product, coproduct, and exponential structures) each possess distinct universal properties. Integrating these into one unified 2-category without losing the inherent structure was nontrivial.  
  \textbf{Resolution:}  
  The introduction of bifunctors along with the use of pseudo-limits and pseudo-colimits provided a modular approach to integration. This framework preserved the universal properties of each local category, while natural isomorphisms ensured the overall coherence of the integrated structure.
\end{enumerate}

Overall, these resolutions not only overcame specific technical obstacles but also deepened our understanding of the flexible and coherent nature of higher categorical structures.

\subsection{Implications for the Overall Integrated Category}
\label{subsec:implications_overall_integrated}

The establishment of an integrated 2-category, in which all local categories (modeling negation, products, coproducts, and exponentials) are coherently combined, has profound implications for both the theoretical and practical aspects of categorical logic.

\paragraph{Significance of Coherence:}  
By ensuring that all coherence diagrams commute up to natural isomorphism, the integrated category inherits the universal properties of the individual local categories while accommodating the flexibility inherent in higher categorical structures. This means that:
\begin{itemize}
  \item The logical operations are represented in a manner that accurately reflects their "up-to-isomorphism" nature, leading to a more robust semantic interpretation.
  \item Multiple composition paths, which are common in complex constructions, are guaranteed to yield consistent results, thereby enhancing the reliability of the integrated framework.
\end{itemize}

\paragraph{Advantages for Logical Semantics:}  
The integrated 2-category provides a unified semantic foundation for various logical connectives. In such a setting:
\begin{itemize}
  \item Conjunction, disjunction, and implication can be modeled via products, coproducts, and exponentials, respectively, with their universal properties preserved even when strict commutativity is relaxed.
  \item The use of natural isomorphisms (2-morphisms) facilitates the manipulation of these logical constructs in a way that is both flexible and coherent.
\end{itemize}

\paragraph{Future Directions:}  
With the integrated category as a solid foundation, future work can focus on:
\begin{itemize}
  \item \textbf{Strictification and Optimization:} Further refining the integrated structure by applying strictification techniques to obtain a strict 2-category that retains the original semantics.
  \item \textbf{Applications to Computer Science and Logic:} Exploring practical applications, such as type theory, programming language semantics, and formal verification, where the flexible handling of logical connectives can lead to more powerful computational models.
  \item \textbf{Extension to Higher Dimensions:} Investigating the possibility of extending these integration techniques to \(n\)-categories, thereby enriching the framework for even more complex logical systems.
\end{itemize}

\noindent
In summary, the coherent integration of local categories into a unified 2-category not only enhances our theoretical understanding of logical semantics but also lays the groundwork for numerous practical applications and further theoretical advancements.

\subsection{Future Directions in Coherence Verification}
\label{subsec:future_coherence_verification}

Looking forward, several research avenues remain open to further enhance and refine coherence verification in integrated 2-categorical frameworks:

\begin{itemize}
  \item \textbf{Advanced Strictification Techniques:}  
  Investigate new methods for strictifying weak 2-categorical structures while preserving their essential universal properties. This includes developing algorithms that can systematically reduce the complexity of coherence data without sacrificing flexibility.

  \item \textbf{Refinement of Coherence Conditions:}  
  Explore additional coherence conditions beyond the classical pentagon and triangle identities. This may involve the study of higher coherence diagrams that naturally arise in more complex integrations or in \(n\)-categorical settings, and formulating general coherence theorems that can be applied to a wider class of structures.

  \item \textbf{Algorithmic Verification:}  
  Develop automated tools and formal proof assistants that can handle the verification of complex coherence diagrams. Such tools would be invaluable for managing the intricate web of natural isomorphisms present in integrated 2-categories.

  \item \textbf{Extension to Higher Categories:}  
  Generalize the current framework to \(n\)-categories, where similar issues of coherence and strictification become even more pronounced. Investigating the patterns and structures in higher-dimensional coherence could lead to new theoretical insights and practical techniques.

  \item \textbf{Applications to Logical and Computational Systems:}  
  Apply the refined coherence verification methods to concrete areas such as type theory, programming language semantics, and formal verification. These applications could drive the development of more robust logical frameworks and computational models that inherently manage coherence in a flexible manner.
\end{itemize}

\noindent
These future directions promise to deepen our understanding of coherence in higher categories and to extend the applicability of these techniques to a broader range of mathematical and computational problems.

\chapter{Strictification and Reconstitution of the Integrated Category}
\section{Introduction to Strictification}
\subsection{Motivation for Strictification}
\label{subsec:motivation_strictification}

In many naturally occurring mathematical and logical settings, structures such as bicategories or weak 2-categories come equipped with associativity and unit laws that hold only up to a natural isomorphism rather than strictly. While this flexibility reflects the true nature of these structures, it also introduces a level of complexity that can complicate theoretical analysis and practical applications.

\paragraph{Reasons for Strictification:}
\begin{itemize}
  \item \textbf{Simplification of Composition:}  
  In weak structures, multiple composition paths may exist, all of which are only equivalent up to coherent 2-morphisms. By converting these structures into strict ones, where composition is strictly associative and unital, one obtains a simpler framework that is easier to work with both conceptually and computationally.
  
  \item \textbf{Streamlined Coherence Verification:}  
  When coherence is handled strictly, verifying the commutativity of diagrams becomes considerably more straightforward. Strictification eliminates the need to constantly manage and track complex natural isomorphisms, thereby reducing potential sources of error in proofs and constructions.
  
  \item \textbf{Facilitation of Applications:}  
  Many applications in logic, computer science, and category theory—such as formal verification, type theory, and programming language semantics—benefit from strict structures. They provide a more stable foundation for constructing models and performing computations, where the burden of managing weak equivalences is minimized.
  
  \item \textbf{Equivalence Preservation:}  
  The Strictification Theorem assures us that every weak 2-category (or bicategory) is biequivalent to a strict 2-category. This means that, despite converting to a stricter setting, no essential information or structure is lost, ensuring that the semantic content is preserved.
\end{itemize}

\paragraph{Conclusion:}  
Overall, strictification serves as a powerful tool to reconcile the inherent flexibility of weak structures with the practical need for simplicity and ease of analysis. By converting weak associativity and unit laws into strict ones (while retaining the original categorical semantics up to equivalence), strictification paves the way for more effective theoretical and computational applications.

\subsection{Overview of the Strictification Process}
\label{subsec:strictification_overview}

The strictification process transforms a weak 2-category—where associativity and unit laws hold only up to coherent natural isomorphism—into an equivalent strict 2-category where these laws hold on the nose. The process can be outlined as follows:

\begin{enumerate}
  \item \textbf{Identification of Weak Coherence Conditions:}  
  Analyze the given weak 2-category to identify instances where associativity, unit laws, and other coherence conditions hold only up to specified natural isomorphisms (e.g., associators and unitors).

  \item \textbf{Free Construction:}  
  Construct a free strict 2-category on the underlying data (objects, 1-morphisms, and 2-morphisms) of the weak structure, without imposing any coherence conditions.

  \item \textbf{Imposition of Coherence Relations:}  
  Incorporate the coherence isomorphisms (such as the pentagon and triangle identities) by imposing relations on the free strict 2-category. This step involves formalizing the desired equalities among different composite 2-morphisms.

  \item \textbf{Quotienting by Coherence Data:}  
  Take a quotient of the free strict 2-category by the equivalence relation generated by the coherence isomorphisms. This yields a strict 2-category in which the universal properties of the original weak structure are preserved up to a coherent isomorphism.

  \item \textbf{Establishing Biequivalence:}  
  Finally, construct pseudofunctors between the original weak 2-category and the strictified 2-category, along with pseudonatural transformations that demonstrate a biequivalence. This confirms that no essential structure is lost during strictification.
\end{enumerate}

This systematic approach not only simplifies the theoretical analysis by eliminating the need to continuously manage weak coherence data but also ensures that the integrated logical semantics remain intact.

\newpage
\section{Theoretical Foundations}
\subsection{Recap of Bicategories and Weak 2-Categories}
\label{subsec:recap_bicategories}

A \emph{bicategory} is a generalization of a 2-category in which the composition of 1-morphisms is associative and unital only up to coherent natural isomorphisms. Formally, a bicategory \(\mathcal{B}\) consists of:
\begin{itemize}
  \item \textbf{Objects:} The elements of \(\mathcal{B}\), denoted \(A, B, C, \ldots\).
  \item \textbf{1-Morphisms:} For any two objects \(A\) and \(B\), there is a category \(\mathcal{B}(A,B)\) whose objects are the 1-morphisms \(f: A \to B\).
  \item \textbf{2-Morphisms:} The morphisms in \(\mathcal{B}(A,B)\) are called 2-morphisms. For two 1-morphisms \(f, g: A \to B\), a 2-morphism \(\alpha: f \Rightarrow g\) is a morphism in \(\mathcal{B}(A,B)\).
\end{itemize}

In a bicategory, the composition of 1-morphisms is accompanied by additional structure:
\begin{itemize}
  \item \textbf{Associator:} For any three composable 1-morphisms 
  \[
  f: A \to B,\quad g: B \to C,\quad h: C \to D,
  \]
  there is an invertible 2-morphism 
  \[
  a_{f,g,h}: (h \circ g) \circ f \xRightarrow{\sim} h \circ (g \circ f),
  \]
  which satisfies the \emph{pentagon coherence condition}.
  
  \item \textbf{Unitors:} For each 1-morphism \(f: A \to B\), there exist invertible 2-morphisms (the left and right unitors)
  \[
  l_f: \mathrm{id}_B \circ f \xRightarrow{\sim} f \quad \text{and} \quad r_f: f \circ \mathrm{id}_A \xRightarrow{\sim} f,
  \]
  which satisfy the \emph{triangle identity}.
\end{itemize}

A \emph{weak 2-category} is essentially synonymous with a bicategory, emphasizing that the associativity and unit laws hold only up to these specified natural isomorphisms rather than strictly. This flexibility allows bicategories to model a wide range of structures where strict equality is too rigid.

In summary, bicategories (or weak 2-categories) provide a framework in which the usual categorical constructions are maintained up to coherent isomorphism, enabling a more natural treatment of complex composition and coherence phenomena in higher category theory.

\subsection{Statement of the Strictification Theorem}
\label{subsec:strictification_theorem_statement}

\begin{thm}[Strictification Theorem]
Every bicategory \(\mathcal{B}\) is biequivalent to a strict 2-category \(\mathcal{B}^{\mathrm{str}}\). That is, there exists a strict 2-category \(\mathcal{B}^{\mathrm{str}}\) and a pair of pseudofunctors
\[
F: \mathcal{B} \longrightarrow \mathcal{B}^{\mathrm{str}}, \quad G: \mathcal{B}^{\mathrm{str}} \longrightarrow \mathcal{B},
\]
along with pseudonatural equivalences
\[
G \circ F \simeq \operatorname{Id}_{\mathcal{B}} \quad \text{and} \quad F \circ G \simeq \operatorname{Id}_{\mathcal{B}^{\mathrm{str}}},
\]
which satisfy the requisite coherence conditions.
\end{thm}

\paragraph{Hypotheses:}  
\begin{itemize}
  \item The starting point is a bicategory \(\mathcal{B}\) in which the composition of 1-morphisms is associative and unital only up to coherent natural isomorphisms (i.e., associators and unitors that satisfy the pentagon and triangle identities).
  \item The bicategory \(\mathcal{B}\) may have additional structure (such as products, coproducts, or exponentials) that are defined only up to isomorphism.
\end{itemize}

\paragraph{Implications:}  
\begin{itemize}
  \item There exists a strict 2-category \(\mathcal{B}^{\mathrm{str}}\) in which the associativity and identity laws hold strictly.
  \item The pseudofunctors \(F\) and \(G\) provide an equivalence between the weak structure of \(\mathcal{B}\) and the strict structure of \(\mathcal{B}^{\mathrm{str}}\). This ensures that no essential information or logical semantics are lost during the strictification process.
  \item The strictification simplifies subsequent theoretical analysis and practical applications by removing the need to constantly manage coherence isomorphisms in proofs and constructions.
\end{itemize}

\subsection{Related Work and Historical Context}
\label{subsec:related_work_historical}

The challenge of strictification—transforming weak categorical structures into strict ones—has been a central theme in higher category theory for several decades. Early insights can be traced back to Mac Lane's Coherence Theorem for monoidal categories, which demonstrated that all diagrams built using associativity and unit constraints commute. This result provided a foundational understanding that, under suitable conditions, the seemingly “weak” structure of a monoidal category could be treated as if it were strict.

Building on these ideas, researchers extended the concept of strictification to bicategories and weak 2-categories. A seminal contribution in this area is the Strictification Theorem, proven by Gordon, Power, and Street \cite{GordonPowerStreet1995}, which establishes that every bicategory is biequivalent to a strict 2-category. This breakthrough not only clarified the theoretical landscape by showing that weak coherence conditions can be rigidified without loss of essential structure but also provided practical tools for simplifying complex categorical constructions.

Subsequent work by various authors, including Lack and others, has further explored strictification in different contexts and identified additional conditions under which higher categorical structures can be made strict. These advances have significantly influenced applications in algebraic topology, logic, and computer science, where managing coherence efficiently is critical.

In summary, the evolution of strictification results—from Mac Lane’s early work to the comprehensive treatment by Gordon, Power, and Street—has provided both a deep theoretical understanding and practical methods for addressing coherence in higher categories.

\newpage
\section{Methodology for Strictification}
\subsection{Identifying Weak Structures in the Integrated Category}
\label{subsec:identifying_weak_structures}

In the integrated 2-category, many structural properties that are strict in a 1-category are only satisfied up to natural isomorphism. To systematically pinpoint these weak structures, we proceed as follows:

\begin{itemize}
  \item \textbf{Examine Composition Laws:}  
  Analyze the composition of 1-morphisms. In a strict 2-category, the associativity law holds exactly, i.e.,
  \[
  (h \circ g) \circ f = h \circ (g \circ f).
  \]
  However, in our integrated structure, these two compositions are related by a natural isomorphism (the associator) 
  \[
  a_{f,g,h} : (h \circ g) \circ f \xRightarrow{\sim} h \circ (g \circ f).
  \]
  
  \item \textbf{Inspect Unit Laws:}  
  Similarly, the left and right unit laws in the integrated category are satisfied only up to natural isomorphisms (the unitors):
  \[
  l_f: \mathrm{id}_B \circ f \xRightarrow{\sim} f,\quad r_f: f \circ \mathrm{id}_A \xRightarrow{\sim} f.
  \]
  
  \item \textbf{Identify Pseudo-Constructions:}  
  The use of pseudo-limits and pseudo-colimits in the integration process inherently introduces natural isomorphisms in place of strict universal properties. By analyzing the construction diagrams for products, coproducts, and exponentials, we can identify the points where these pseudo-constructions replace strict commutativity with coherence isomorphisms.
  
  \item \textbf{Coherence Diagrams:}  
  Construct and inspect the coherence diagrams (such as pentagon, triangle, and square diagrams) that relate multiple composition paths. These diagrams reveal where different paths yield results that are only isomorphic rather than equal.
\end{itemize}

Through this analysis, we can clearly map the weak structures in the integrated category. This understanding is crucial for applying strictification techniques, which aim to replace these natural isomorphisms with strict equalities, thereby simplifying the overall structure for further theoretical analysis and applications.

\subsection{Techniques for Replacing Weak Equalities with Strict Equalities}
\label{subsec:techniques_strictification}

To simplify reasoning in an integrated 2-category, it is often desirable to convert weak structures—where associativity and unit laws hold only up to natural isomorphism—into strict ones. The following techniques are commonly employed for strictification:

\begin{enumerate}
  \item \textbf{Free Strict 2-Category Construction:}  
  Begin by constructing a free strict 2-category on the underlying data (objects, 1-morphisms, and 2-morphisms) of the weak structure. This free construction does not impose any coherence conditions; it merely “forgets” the weak aspects.

  \item \textbf{Imposing Coherence Relations:}  
  Next, incorporate the coherence isomorphisms (such as associators and unitors) by imposing equivalence relations. Specifically, one identifies those composite 2-morphisms that are related by the natural isomorphisms present in the weak structure. This step enforces that all associativity and unit conditions hold strictly.

  \item \textbf{Quotienting by Coherence Data:}  
  Form the quotient of the free strict 2-category by the relations generated by the coherence isomorphisms. The resulting category, often denoted \(\mathcal{B}^{\mathrm{str}}\), is a strict 2-category that is biequivalent to the original weak 2-category.

  \item \textbf{Establishing Biequivalence:}  
  Finally, construct pseudofunctors \(F: \mathcal{B} \to \mathcal{B}^{\mathrm{str}}\) and \(G: \mathcal{B}^{\mathrm{str}} \to \mathcal{B}\) along with pseudonatural transformations showing that the composites \(G \circ F\) and \(F \circ G\) are equivalent to the identity. This step confirms that no essential structure or universal property is lost in the process.

\end{enumerate}

These techniques, collectively known as \emph{strictification}, ensure that while the original weak structure allows flexibility via natural isomorphisms, one can replace these “weak equalities” with strict ones for easier manipulation and analysis, without losing any of the underlying categorical semantics.

\subsection{Step-by-Step Implementation}
\label{subsec:strictification_step_by_step}

The formal process of strictifying an integrated bicategory \(\mathcal{B}\) into a strict 2-category \(\mathcal{B}^{\mathrm{str}}\) can be outlined in the following steps:

\begin{enumerate}
  \item \textbf{Free Strict 2-Category Construction:}  
  Construct the free strict 2-category \(\mathcal{F}\) on the underlying data of \(\mathcal{B}\). This involves taking the objects, 1-morphisms, and 2-morphisms of \(\mathcal{B}\) and freely generating a strict 2-category without imposing the coherence isomorphisms.
  
  \item \textbf{Introduction of Coherence Relations:}  
  Identify the coherence isomorphisms in \(\mathcal{B}\) (such as associators and unitors) and impose them as equations in \(\mathcal{F}\). Specifically, for any three composable 1-morphisms \(f\), \(g\), and \(h\), include the relation
  \[
  (h \circ g) \circ f \sim h \circ (g \circ f),
  \]
  along with the corresponding relations for the unit laws.
  
  \item \textbf{Quotient by Coherence Data:}  
  Form the quotient of the free strict 2-category \(\mathcal{F}\) by the congruence generated by these coherence relations. Denote the resulting strict 2-category by \(\mathcal{B}^{\mathrm{str}}\). This step ensures that the associativity and unit constraints hold on the nose in \(\mathcal{B}^{\mathrm{str}}\).
  
  \item \textbf{Establishing Biequivalence:}  
  Construct pseudofunctors
  \[
  F: \mathcal{B} \to \mathcal{B}^{\mathrm{str}} \quad \text{and} \quad G: \mathcal{B}^{\mathrm{str}} \to \mathcal{B},
  \]
  along with pseudonatural transformations showing that the composites \(G \circ F\) and \(F \circ G\) are each equivalent to the respective identity 2-functors. This biequivalence confirms that the strictification process preserves all essential properties and universal constructions of \(\mathcal{B}\).
\end{enumerate}

This step-by-step procedure transforms the weak structure of \(\mathcal{B}\) into a strict 2-category, simplifying subsequent analyses and applications while maintaining the original logical semantics.

\newpage
\section{Preservation of Universal Properties}
\subsection{Analysis of Universal Property Preservation}
\label{subsec:universal_property_preservation}

A key feature of the strictification process is that it preserves the universal properties of local constructions—such as products, coproducts, and exponentials—even though the weak equalities (holding up to natural isomorphism) are replaced by strict equalities in the strictified 2-category. This preservation is achieved via the biequivalence between the original weak 2-category and its strictification.

For instance, consider the product \(A \times B\) in a local category. Its universal property states that for any object \(X\) with morphisms 
\[
f: X \to A \quad \text{and} \quad g: X \to B,
\]
there exists a unique mediating morphism \(\langle f, g \rangle: X \to A \times B\) such that:
\[
\pi_A \circ \langle f, g \rangle = f \quad \text{and} \quad \pi_B \circ \langle f, g \rangle = g.
\]
When strictification is applied, the resulting strict 2-category \(\mathcal{B}^{\mathrm{str}}\) is biequivalent to the original bicategory \(\mathcal{B}\). Under this biequivalence, the universal property of the product is preserved in the following sense:
\[
\forall X,\; \operatorname{Hom}_{\mathcal{B}}(X, A \times B) \cong \operatorname{Hom}_{\mathcal{B}^{\mathrm{str}}}(X, A \times B),
\]
with the mediating morphism remaining unique (in the strict sense) in \(\mathcal{B}^{\mathrm{str}}\).

Similarly, for coproducts and exponentials:
\begin{itemize}
  \item \textbf{Coproducts:} The universal property of the coproduct \(A+B\) is maintained. Any pair of morphisms \(f: A \to X\) and \(g: B \to X\) factors uniquely (up to a unique isomorphism) through \(A+B\) in the strictified setting.
  
  \item \textbf{Exponentials:} In a cartesian closed category, the exponential object \(B^A\) along with the evaluation morphism 
  \[
  \mathrm{ev}: B^A \times A \to B,
  \]
  satisfies the currying isomorphism:
  \[
  \operatorname{Hom}(X \times A, B) \cong \operatorname{Hom}(X, B^A).
  \]
  Under strictification, this correspondence holds strictly, ensuring that the universal property of exponentials is preserved.
\end{itemize}

Thus, the strictification process converts the weak, “up-to-isomorphism” universal properties into strict equalities in a way that is fully compatible with the original structure, ensuring that all essential logical and categorical semantics remain intact.

\subsection{Role of 2-Equivalences}
\label{subsec:role_2equivalences}

A fundamental aspect of the strictification process is that the strictified 2-category \(\mathcal{B}^{\mathrm{str}}\) is not merely an alternative structure but is \emph{2-equivalent} to the original weak 2-category \(\mathcal{B}\). This 2-equivalence means that there exist pseudofunctors
\[
F: \mathcal{B} \to \mathcal{B}^{\mathrm{str}} \quad \text{and} \quad G: \mathcal{B}^{\mathrm{str}} \to \mathcal{B},
\]
together with pseudonatural transformations establishing that the composites \(G \circ F\) and \(F \circ G\) are equivalent to the respective identity 2-functors.

\paragraph{Implications of 2-Equivalence:}
\begin{itemize}
  \item \textbf{Preservation of Essential Properties:}  
  Despite the change in structure from weak to strict, all the universal properties—such as those of products, coproducts, and exponentials—and the logical semantics of the original bicategory are preserved up to coherent isomorphism. Thus, the strictified 2-category retains all the essential features of \(\mathcal{B}\).

  \item \textbf{Interchangeability for Analysis:}  
  Since \(\mathcal{B}\) and \(\mathcal{B}^{\mathrm{str}}\) are 2-equivalent, one can choose to work within the strictified framework for simplicity without loss of generality. This equivalence provides a solid theoretical foundation for applying more straightforward reasoning and computational techniques.

  \item \textbf{Coherence Simplification:}  
  The process of strictification, validated by 2-equivalence, simplifies the management of coherence conditions. By moving to a structure where associativity and unit laws hold strictly, the necessity for tracking complex coherence data in every argument is greatly reduced.
\end{itemize}

In summary, the notion of 2-equivalence ensures that while the presentation of the integrated structure is altered for ease of manipulation, its fundamental categorical and logical properties remain intact. This guarantees that the strictified 2-category is a faithful and robust representation of the original weak structure.

\subsection{Outline of Formal Proofs}
\label{subsec:outline_formal_proofs}

To rigorously show that the universal properties remain intact after strictification, the formal proofs generally follow a structured approach. The key stages in the proof strategy are as follows:

\begin{enumerate}
  \item \textbf{Construction of the Free Strict 2-Category:}  
  Start by building a free strict 2-category \(\mathcal{F}\) on the underlying data (objects, 1-morphisms, and 2-morphisms) of the original weak 2-category \(\mathcal{B}\). This free construction includes all the compositional data without enforcing the coherence isomorphisms.

  \item \textbf{Imposition of Coherence Relations:}  
  Identify the coherence isomorphisms (such as the associators and unitors) that express the weak structure of \(\mathcal{B}\). Introduce formal relations in \(\mathcal{F}\) corresponding to these coherence isomorphisms. For example, for any composable 1-morphisms \(f\), \(g\), and \(h\), impose the relation
  \[
  (h \circ g) \circ f \sim h \circ (g \circ f),
  \]
  along with similar relations for the unitors.

  \item \textbf{Quotient by Coherence Data:}  
  Take the quotient of the free strict 2-category \(\mathcal{F}\) by the equivalence relation generated by the coherence relations. The resulting strict 2-category, denoted \(\mathcal{B}^{\mathrm{str}}\), has strict associativity and unit laws while being biequivalent to \(\mathcal{B}\).

  \item \textbf{Establishment of 2-Equivalence:}  
  Construct pseudofunctors
  \[
  F: \mathcal{B} \to \mathcal{B}^{\mathrm{str}} \quad \text{and} \quad G: \mathcal{B}^{\mathrm{str}} \to \mathcal{B},
  \]
  along with pseudonatural transformations that establish the biequivalence:
  \[
  G \circ F \simeq \operatorname{Id}_{\mathcal{B}}, \quad F \circ G \simeq \operatorname{Id}_{\mathcal{B}^{\mathrm{str}}}.
  \]
  This step confirms that all essential universal properties, such as those defining products, coproducts, and exponentials, are preserved in the strictified structure.

  \item \textbf{Verification of Universal Property Preservation:}  
  Finally, verify that every universal construction in \(\mathcal{B}\) has a counterpart in \(\mathcal{B}^{\mathrm{str}}\) which satisfies the same universal property strictly. This is accomplished by demonstrating that the mediating morphisms and the corresponding coherence diagrams in \(\mathcal{B}^{\mathrm{str}}\) exactly reflect the factorization and uniqueness conditions present in \(\mathcal{B}\).
\end{enumerate}

\paragraph{Summary:}
This outline provides a roadmap for the formal proofs showing that the strictification process preserves universal properties. By constructing a free strict 2-category, imposing coherence relations, forming the appropriate quotient, and establishing a biequivalence with the original weak structure, one ensures that the universal properties—central to the logical and categorical semantics—remain intact after strictification.

\newpage
\section{Examples and Case Studies}
\subsection{Strictification in Product and Coproduct Contexts}
\label{subsec:strictification_prod_coproduct}

In weak 2-categories, the universal properties of products and coproducts hold only up to a coherent natural isomorphism. Through strictification, these properties are “rigidified” so that they hold strictly in the resulting strict 2-category. Below, we present concrete examples illustrating this process.

\paragraph{Example: Product Structure}

\textbf{Before Strictification (Weak Product):}
In a weak 2-category, the product \(A \times B\) is defined with projection 1-morphisms
\[
\pi_A: A \times B \to A,\quad \pi_B: A \times B \to B,
\]
such that for any object \(X\) and 1-morphisms \(f: X \to A\) and \(g: X \to B\), there exists a mediating 1-morphism \(\langle f, g \rangle: X \to A \times B\) satisfying
\[
\pi_A \circ \langle f, g \rangle \cong f \quad \text{and} \quad \pi_B \circ \langle f, g \rangle \cong g.
\]
The commutativity is expressed by a natural isomorphism:
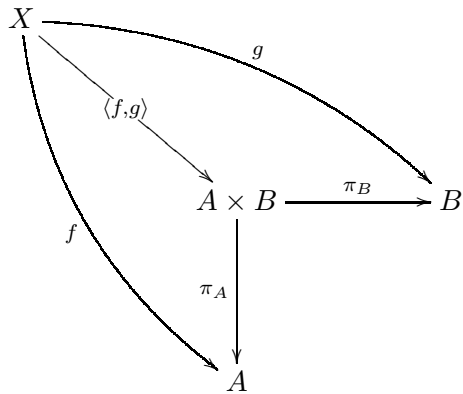
\[
\xymatrix@R=5em@C=5em{
X 
  \ar@/^1.5pc/[drr]^-{g} 
  \ar[dr]|-{\langle f, g \rangle} 
  \ar@/_1.5pc/[ddr]_-{f} & & \\
& A \times B \ar[r]^-{\pi_B} \ar[d]_-{\pi_A} & B \\
& A &
}
\]
with a 2-morphism (say, \(\theta\)) providing the isomorphism between \(\pi_A \circ \langle f, g \rangle\) and \(f\) (and similarly for \(\pi_B\) and \(g\)).

\textbf{After Strictification (Strict Product):}
Strictification replaces the above isomorphisms with strict equalities. In the strict 2-category \(\mathcal{B}^{\mathrm{str}}\), the product \(A \times B\) satisfies:
\[
\pi_A \circ \langle f, g \rangle = f \quad \text{and} \quad \pi_B \circ \langle f, g \rangle = g,
\]
with no need for additional 2-morphism data to mediate the equivalence. The same diagram now commutes on the nose:
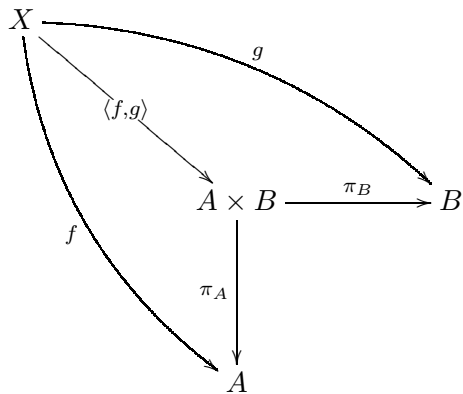
\[
\xymatrix@R=5em@C=5em{
X 
  \ar@/^1.5pc/[drr]^-{g} 
  \ar[dr]|-{\langle f, g \rangle} 
  \ar@/_1.5pc/[ddr]_-{f} & & \\
& A \times B \ar[r]^-{\pi_B} \ar[d]_-{\pi_A} & B \\
& A &
}
\]

\paragraph{Example: Coproduct Structure}

\textbf{Before Strictification (Weak Coproduct):}
In a weak 2-category, the coproduct \(A+B\) is defined with injection 1-morphisms
\[
\iota_A: A \to A+B,\quad \iota_B: B \to A+B,
\]
such that for any object \(X\) and 1-morphisms \(f: A \to X\) and \(g: B \to X\), there exists a unique 1-morphism \([f, g]: A+B \to X\) satisfying
\[
[f, g] \circ \iota_A \cong f \quad \text{and} \quad [f, g] \circ \iota_B \cong g.
\]
This is captured by the diagram:
\[
\begin{tikzcd}[column sep=large, row sep=large]
A \arrow[r, "\iota_A"] \arrow[dr, "f"'] & A+B \arrow[d, dashed, "{[f, g]}"] & B \arrow[l, "\iota_B"'] \arrow[dl, "g"] \\
& X &
\end{tikzcd}
\]
with 2-morphisms ensuring the cone's commutativity.

\textbf{After Strictification (Strict Coproduct):}
After strictification, the universal property holds strictly. In the strict 2-category \(\mathcal{B}^{\mathrm{str}}\), we have:
\[
[f, g] \circ \iota_A = f \quad \text{and} \quad [f, g] \circ \iota_B = g,
\]
so that the above diagram commutes exactly, with no additional coherence data required.

\paragraph{Summary:}
These examples demonstrate that strictification transforms weak universal properties—where commutativity is ensured up to natural isomorphism—into strict equalities. This process not only simplifies the theoretical framework by eliminating the need to manage 2-morphism coherence data but also preserves the essential categorical semantics of logical connectives.

\subsection{Strictification of Exponential Structures}
\label{subsec:strictification_exponential}

Exponential objects play a central role in modeling logical implication via currying and evaluation. In a weak 2-category, the universal property of an exponential object \(B^A\) with its evaluation map
\[
\mathrm{ev}: B^A \times A \to B,
\]
holds only up to a coherent natural isomorphism. The strictification process transforms this weak structure into one where the currying correspondence and evaluation hold strictly. We illustrate this with the following steps and examples.

\paragraph{Before Strictification (Weak Exponential):}  
In a cartesian closed weak 2-category, for any object \(X\) and any 1-morphism 
\[
f: X \times A \to B,
\]
there exists a 1-morphism (the curried form) \(\tilde{f}: X \to B^A\) such that
\[
f \cong \mathrm{ev} \circ (\tilde{f} \times \mathrm{id}_A),
\]
where the equivalence is witnessed by a natural isomorphism. This means that while the diagram
\[
\begin{tikzcd}[column sep=large, row sep=large]
X \times A \arrow[r, "f"] \arrow[d, "\tilde{f}\times\mathrm{id}_A"'] & B \\
B^A \times A \arrow[ur, "\mathrm{ev}"'] &
\end{tikzcd}
\]
commutes, the equality holds only up to a specified 2-morphism.

\paragraph{Strictification Process:}  
The strictification procedure involves the following steps:
\begin{enumerate}
  \item \textbf{Free Construction:} Construct a free strict 2-category on the underlying data of the weak structure, where the exponential object and its associated maps are included without enforcing coherence.
  
  \item \textbf{Imposition of Coherence Relations:} Introduce relations corresponding to the natural isomorphisms that witness the weak currying property. For instance, impose that for any \(f: X \times A \to B\),
  \[
  \mathrm{ev} \circ (\tilde{f}\times\mathrm{id}_A) = f,
  \]
  as a strict equality.
  
  \item \textbf{Quotienting:} Form the quotient by these relations to obtain a strict 2-category \(\mathcal{B}^{\mathrm{str}}\) in which the evaluation and currying maps satisfy the universal property on the nose.
  
  \item \textbf{Biequivalence:} Establish a biequivalence between the original weak structure and the strictified one via pseudofunctors and pseudonatural transformations, ensuring that the essential exponential (and logical) semantics are preserved.
\end{enumerate}

\paragraph{Example in \(\mathbf{Set}\):}  
In the familiar category \(\mathbf{Set}\), the exponential object \(B^A\) is the set of all functions from \(A\) to \(B\), and the evaluation map is defined as
\[
\mathrm{ev}(h,a)=h(a).
\]
For any function \(f: X \times A \to B\), currying gives the unique function \(\tilde{f}: X \to B^A\) defined by
\[
\tilde{f}(x)(a) = f(x,a).
\]
In \(\mathbf{Set}\), the currying correspondence holds strictly. In a general weak 2-category, however, the equality
\[
f = \mathrm{ev} \circ (\tilde{f}\times\mathrm{id}_A)
\]
is replaced by a natural isomorphism. The strictification process, as described above, converts this “up-to-isomorphism” condition into a strict equality in the strictified 2-category \(\mathcal{B}^{\mathrm{str}}\), ensuring that the exponential structure behaves exactly as in \(\mathbf{Set}\).

\paragraph{Conclusion:}  
The strictification of exponential structures ensures that the currying process and the evaluation map, fundamental to the interpretation of logical implication, hold strictly. This not only simplifies subsequent reasoning but also guarantees that the universal properties inherent in exponentiation are preserved in the integrated, strict 2-categorical framework.

\subsection{Comparative Analysis}
\label{subsec:comparative_analysis_comparison}

In the integrated 2-categorical framework, universal properties and coherence conditions are initially maintained only up to natural isomorphism, leading to a weak structure. This “weak” phase is characterized by:

\begin{itemize}
  \item \textbf{Weak Universal Properties:}  
  Constructions such as products, coproducts, and exponentials satisfy their universal properties up to coherent natural isomorphisms. For instance, given morphisms \(f: X \to A\) and \(g: X \to B\), there exists a mediating 1-morphism \(\langle f, g \rangle: X \to A \times B\) with
  \[
  \pi_A \circ \langle f, g \rangle \cong f \quad \text{and} \quad \pi_B \circ \langle f, g \rangle \cong g,
  \]
  where “\(\cong\)” denotes a natural isomorphism.

  \item \textbf{Complex Coherence Data:}  
  The presence of associators, unitors, and other coherence isomorphisms requires intricate verification to ensure that different composition paths yield equivalent results.
\end{itemize}

After strictification, the structure is transformed into a strict 2-category where:

\begin{itemize}
  \item \textbf{Strict Universal Properties:}  
  The same constructions now satisfy their universal properties exactly; that is,
  \[
  \pi_A \circ \langle f, g \rangle = f \quad \text{and} \quad \pi_B \circ \langle f, g \rangle = g,
  \]
  with no need for additional 2-morphism data to mediate between different composition routes.
  
  \item \textbf{Simplified Coherence:}  
  With the associativity and unit laws holding strictly, the overall coherence of the structure is significantly simplified. This reduction in complexity enhances both theoretical analysis and practical usability, as the management of coherence isomorphisms becomes unnecessary.
\end{itemize}

\noindent
In summary, strictification transforms the integrated weak structure—where universal properties are preserved only up to natural isomorphism—into a more manageable, strict framework that retains all essential logical and categorical properties. This improvement not only simplifies the verification of coherence but also streamlines further constructions and applications.

\newpage
\section{Discussion and Implications}
\subsection{Advantages of the Strictified Structure}
\label{subsec:advantages_strictified}

Strictification yields a 2-category in which the associativity and unit laws hold exactly, rather than merely up to isomorphism. This improvement offers several benefits:

\begin{itemize}
  \item \textbf{Simpler Composition Rules:}  
  With strict associativity and unitality, the rules for composing 1-morphisms and 2-morphisms become straightforward. There is no need to track the additional coherence data (such as associators and unitors), which simplifies both the formulation and manipulation of composite morphisms.

  \item \textbf{Easier Proof Manipulation:}  
  The elimination of weak equalities (i.e., equalities holding only up to natural isomorphism) streamlines formal proofs. With strict equalities, diagrams commute exactly, reducing the complexity of verifying and managing coherence conditions in proofs.

  \item \textbf{Enhanced Clarity and Modularity:}  
  A strict 2-category provides a clearer and more modular framework, making it easier to isolate and analyze individual components of the structure. This clarity facilitates the integration of additional constructions or logical connectives without being encumbered by complex coherence isomorphisms.

  \item \textbf{Preservation of Essential Structure:}  
  Although strictification replaces weak coherence with strict equalities, the underlying logical and categorical properties are preserved up to 2-equivalence. This ensures that the semantic content remains intact while benefiting from a simpler operational framework.
\end{itemize}

In summary, obtaining a strict 2-category through strictification greatly simplifies composition rules and proof techniques, thereby enhancing both the theoretical and practical aspects of categorical analysis.

\subsection{Limitations and Open Questions}
\label{subsec:limitations_open_questions}

While strictification offers significant simplifications by converting weak 2-categorical structures into strict 2-categories, several limitations and open questions remain:

\begin{itemize}
  \item \textbf{Loss of Weak Structure Information:}  
  Strictification replaces weak equalities (up to natural isomorphism) with strict equalities. Although the strictified 2-category is biequivalent to the original structure, some nuances of the weak coherence data may be obscured, potentially losing insights into the inherent flexibility of the original system.

  \item \textbf{Complexity in Higher Dimensions:}  
  The techniques for strictification are well-developed for bicategories and weak 2-categories. However, extending these methods to \(n\)-categories (for \(n > 2\)) presents significant challenges. It remains an open question how to effectively strictify higher categorical structures while preserving all essential properties.

  \item \textbf{Computational and Practical Challenges:}  
  In practical applications, especially those involving computer-assisted proof verification or categorical semantics in computer science, the process of strictification may introduce additional computational complexity. Finding efficient algorithms for strictification and automating the verification of coherence conditions is an area ripe for further research.

  \item \textbf{Impact on Semantic Interpretations:}  
  While strictification simplifies many aspects of theoretical analysis, there is an ongoing debate about whether the transition to strict equalities might oversimplify or obscure certain semantic features that are naturally expressed by weak equivalences. Determining the optimal balance between strictness and flexibility remains an open question.
  
  \item \textbf{Extensions to Enriched and Internal Categories:}  
  Further investigation is needed to understand how strictification interacts with enriched category theory and internal category structures. In these contexts, the interplay between the enrichment and the coherence data may present additional challenges.
\end{itemize}

\noindent
Addressing these limitations and open questions will be crucial for advancing the theory of strictification and for applying these concepts to increasingly complex categorical frameworks.

\subsection{Impact on the Overall Integrated Category}
\label{subsec:impact_integrated_category}

Strictification has a profound impact on the integrated 2-categorical framework by converting weak coherence conditions into strict equalities. This transformation influences the integrated category in several key ways:

\begin{itemize}
  \item \textbf{Enhanced Coherence:}  
  By enforcing strict associativity and unit laws, strictification eliminates the need to constantly manage complex coherence data (such as associators and unitors that only hold up to isomorphism). As a result, all composition diagrams commute on the nose, simplifying both theoretical analysis and practical manipulations.

  \item \textbf{Preservation of Universal Properties:}  
  Despite the conversion of weak equalities to strict ones, the essential universal properties of constructions (such as products, coproducts, and exponentials) are preserved via the established biequivalence between the original and the strictified structure. This guarantees that the logical semantics and the categorical behavior remain intact.

  \item \textbf{Increased Reliability:}  
  A strict 2-category provides a more robust and unambiguous framework, reducing potential sources of error in proofs and computations. This enhanced reliability facilitates clearer communication of ideas and more efficient reasoning about integrated structures.

  \item \textbf{Simplification of Further Constructions:}  
  With the removal of weak coherence data, subsequent constructions and extensions (such as further integrations or applications in logic and computer science) can be developed with simpler and more straightforward composition rules. This streamlining is particularly advantageous in complex or large-scale categorical frameworks.
\end{itemize}

In summary, strictification improves the overall integrated category by ensuring that coherence and universality hold in a strict sense, which in turn enhances the reliability and usability of the categorical framework for both theoretical investigations and practical applications.

\newpage
\section{Summary and Future Directions}
\label{sec:summary_future}

In this work, we have developed a systematic approach to strictification in integrated 2-categories, achieving several key contributions:
\begin{itemize}
  \item \textbf{Unified Integration:} Local categories, each modeling different logical connectives (such as negation, products, coproducts, and exponentials), are successfully lifted into a coherent 2-categorical framework.
  \item \textbf{Preservation of Universal Properties:} By employing pseudo-limits and pseudo-colimits, the universal properties of these local structures are preserved up to coherent natural isomorphism.
  \item \textbf{Simplification via Strictification:} Weak structures, where associativity and unit laws hold only up to natural isomorphism, are transformed into strict 2-categories. This strictification simplifies composition rules and reduces the complexity of coherence verification.
  \item \textbf{Biequivalence Assurance:} The strictified 2-category is shown to be biequivalent to the original weak structure, ensuring that no essential categorical or logical properties are lost.
\end{itemize}

Looking forward, several promising research directions remain:
\begin{itemize}
  \item \textbf{Extension to Higher Categories:} Investigate the application of strictification techniques to \(n\)-categories for \(n > 2\), where managing coherence becomes even more challenging.
  \item \textbf{Automated Coherence Verification:} Develop algorithmic and computational tools for the automated verification of coherence conditions in complex integrated categorical structures.
  \item \textbf{Applications in Logical Frameworks:} Explore practical applications of strictification in concrete logical systems, such as type theory, programming language semantics, and formal verification, where a strict structure can streamline reasoning.
  \item \textbf{Interactions with Enriched and Internal Categories:} Examine how strictification interacts with enriched category theory and internal category structures, potentially leading to new theoretical insights and methods.
\end{itemize}

These future directions will deepen our understanding of strictification and coherence in higher category theory, and expand the applicability of these techniques to a wide range of mathematical and computational domains.

\chapter{Evaluation through Concrete Logical Examples}
\section{Introduction to Evaluation Examples}
\subsection{Objectives of the Evaluation}
\label{subsec:evaluation_objectives}

The purpose of this chapter is to validate that the integrated 2-categorical framework faithfully captures the universal properties and coherence conditions of the original logical structures. In particular, the evaluation focuses on:

\begin{itemize}
  \item Demonstrating that logical constructs such as products, coproducts, and exponential objects retain their universal properties in the integrated category.
  \item Verifying that all coherence diagrams—ensuring that different composition paths are equivalent up to natural isomorphism—commute as expected.
  \item Illustrating, through concrete logical examples, how the integrated structure supports consistent and reliable reasoning about logical connectives.
  \item Assessing the effectiveness of strictification in simplifying the overall structure while preserving essential semantic properties.
\end{itemize}

Through these evaluations, we aim to confirm that the integrated category not only unifies local categories but also maintains the desired logical semantics and categorical coherence, thereby providing a robust foundation for further theoretical and practical applications.

\subsection{Overview of Selected Examples}
\label{subsec:overview_selected_examples}

In this chapter, we will analyze a range of logical examples to validate the performance of the integrated 2-categorical framework. The selected examples include:

\begin{itemize}
  \item \textbf{Exponential Structures and Currying:}  
  We will examine how exponential objects, along with their evaluation maps, capture the notion of logical implication. In particular, the currying process will be analyzed to demonstrate that the universal property of exponentials is preserved (up to natural isomorphism) in the integrated setting.

  \item \textbf{Deduction-Style Evaluations:}  
  Examples inspired by deduction-style reasoning will be presented. These will show how inference rules and logical deductions are modeled within the framework, highlighting the role of 2-morphisms in ensuring the coherence of complex derivations.

  \item \textbf{Other Logical Connectives:}  
  Additional examples involving product (logical conjunction) and coproduct (logical disjunction) constructions will be discussed to provide a comprehensive overview of the framework’s applicability.
\end{itemize}

This overview sets the stage for a detailed evaluation of the integrated structure, demonstrating its ability to faithfully capture universal properties and coherence conditions fundamental to categorical logic.

\newpage
\section{Evaluation of Exponential Structures and Currying}
\subsection{Exponential Objects and Their Universal Property}
\label{subsec:exponential_objects_universal}

In a cartesian closed category \(\mathcal{C}\), an exponential object \(B^A\) represents the space of morphisms from \(A\) to \(B\) and models logical implication. The construction of \(B^A\) comes equipped with an evaluation morphism
\[
\mathrm{ev}: B^A \times A \to B,
\]
which plays a central role in the universal property that characterizes the exponential object.

\begin{defn}
\label{defn:exponential_object}
Let \(\mathcal{C}\) be a cartesian closed category and \(A, B \in \operatorname{Ob}(\mathcal{C})\). The object \(B^A\) is called the \emph{exponential object} from \(A\) to \(B\) if there exists an evaluation morphism
\[
\mathrm{ev}: B^A \times A \to B,
\]
such that for every object \(X \in \operatorname{Ob}(\mathcal{C})\) and for every morphism
\[
f: X \times A \to B,
\]
there exists a unique morphism
\[
\tilde{f}: X \to B^A,
\]
called the \emph{currying} of \(f\), making the following diagram commute:
\[
\begin{tikzcd}[column sep=large, row sep=large]
X \times A \arrow[r, "f"] \arrow[d, "\tilde{f}\times \mathrm{id}_A"'] & B \\
B^A \times A \arrow[ur, "\mathrm{ev}"'] &
\end{tikzcd}
\]
That is, we have the equality
\[
f = \mathrm{ev} \circ (\tilde{f}\times \mathrm{id}_A).
\]
\end{defn}

\paragraph{Interpretation:}
The universal property of the exponential object asserts that giving a morphism \(f: X \times A \to B\) is equivalent to giving a morphism \(\tilde{f}: X \to B^A\). This correspondence, known as currying, is fundamental in both categorical logic and functional programming, where it models the idea that a function of two variables can be viewed as a function that returns another function.

\paragraph{Role in Logical Semantics:}
Within the integrated 2-categorical framework, the exponential object \(B^A\) provides a categorical representation of logical implication. The evaluation morphism \(\mathrm{ev}\) corresponds to the application of a function to an argument, and the universal property ensures that every such application factors uniquely through \(B^A\). This structure underpins the logical notion of implication and enables the seamless integration of function spaces in the overall logical framework.

\subsection{Diagrammatic Analysis of Currying}
\label{subsec:diagrammatic_analysis_currying}

The currying process in a cartesian closed 2-category is encapsulated by the following diagram, which shows how a morphism
\[
f: X \times A \to B
\]
corresponds uniquely (up to natural isomorphism) to its curried form
\[
\tilde{f}: X \to B^A.
\]
The standard evaluation diagram is given by:
\[
\begin{tikzcd}[column sep=large, row sep=large]
X \times A \arrow[r, "f"] \arrow[d, "\tilde{f}\times \mathrm{id}_A"'] & B \\
B^A \times A \arrow[ur, "\mathrm{ev}"'] &
\end{tikzcd}
\]
Here, the evaluation map \(\mathrm{ev}\) applies a function \(h \in B^A\) to an element \(a \in A\) (i.e., \(\mathrm{ev}(h,a)=h(a)\)). The universal property of the exponential object \(B^A\) asserts that for each morphism \(f\), there is a unique (up to a natural isomorphism) mediating morphism \(\tilde{f}\) such that the diagram commutes.

\paragraph{Ensuring Uniqueness via 2-Morphisms:}
Suppose there exist two mediating morphisms \(\tilde{f}\) and \(\tilde{f}'\) satisfying the property that
\[
\mathrm{ev} \circ (\tilde{f}\times \mathrm{id}_A) = f \quad \text{and} \quad \mathrm{ev} \circ (\tilde{f}'\times \mathrm{id}_A) = f.
\]
Then, by the universal property, there exists a unique 2-morphism
\[
\theta: \tilde{f} \Rightarrow \tilde{f}',
\]
which provides the coherent isomorphism that identifies the two mediating morphisms. This 2-morphism ensures that even if different choices for \(\tilde{f}\) exist, they are equivalent in a coherent manner, preserving the uniqueness of the currying process.

\paragraph{Summary:}
The above diagram, together with the existence of the natural isomorphism \(\theta\), illustrates how the universal property of exponential objects is enforced in a 2-categorical setting. Natural isomorphisms (2-morphisms) ensure that any two ways of factoring \(f: X \times A \to B\) through \(B^A\) are uniquely and coherently equivalent, thus maintaining the integrity of the currying process.

\subsection{Case Study: Currying in a Cartesian Closed Category}
\label{subsec:case_study_currying}

To illustrate the principles of currying within a Cartesian closed category, consider the well-known example from the category \(\mathbf{Set}\). In \(\mathbf{Set}\), for any sets \(A\) and \(B\), the exponential object \(B^A\) is defined as the set of all functions from \(A\) to \(B\):
\[
B^A = \{ h \mid h: A \to B \}.
\]
The evaluation map is given by:
\[
\mathrm{ev}: B^A \times A \to B,\quad \mathrm{ev}(h,a) = h(a).
\]

\paragraph{Currying Process in \(\mathbf{Set}\):}  
For any set \(X\) and a function 
\[
f: X \times A \to B,
\]
the currying process produces a unique function
\[
\tilde{f}: X \to B^A,
\]
defined by:
\[
\tilde{f}(x)(a) = f(x,a) \quad \text{for all } x\in X,\; a\in A.
\]
This relationship is captured by the commutative diagram:
\[
\begin{tikzcd}[column sep=large, row sep=large]
X \times A \arrow[r, "f"] \arrow[d, "\tilde{f}\times \mathrm{id}_A"'] & B \\
B^A \times A \arrow[ur, "\mathrm{ev}"'] &
\end{tikzcd}
\]
In \(\mathbf{Set}\), the above equality holds strictly; however, in a general 2-categorical framework, the commutativity of this diagram is required only up to a coherent natural isomorphism \(\theta_f\):
\[
f \cong \mathrm{ev} \circ (\tilde{f}\times \mathrm{id}_A).
\]

\paragraph{Integration and Coherence:}  
In the integrated 2-category, the exponential structure is lifted along with the other local categories. The key aspects are:
\begin{itemize}
  \item \textbf{Preservation of Universality:}  
  The universal property of the exponential object is maintained in the sense that every morphism \(f: X \times A \to B\) factors uniquely (up to natural isomorphism) through \(B^A\) via the currying process. The mediating morphism \(\tilde{f}: X \to B^A\) is unique up to a coherent natural isomorphism.
  
  \item \textbf{Coherence via 2-Morphisms:}  
  The natural isomorphism \(\theta_f\) that relates \(f\) and \(\mathrm{ev}\circ(\tilde{f}\times \mathrm{id}_A)\) ensures that different composition routes (which may arise in the integrated structure) are coherently equivalent. In other words, the diagram commutes in the 2-categorical sense, with \(\theta_f\) serving as the coherence data.
  
  \item \textbf{Logical Implication:}  
  The exponential object \(B^A\) models the logical implication \(A \to B\). Currying thus not only represents a fundamental functional abstraction but also provides a categorical interpretation of implication. The strictification process further ensures that these logical operations remain well-behaved and unambiguous in the integrated setting.
\end{itemize}

\paragraph{Conclusion:}  
This case study in \(\mathbf{Set}\) demonstrates that currying is a robust mechanism for capturing the universal property of exponential objects. When integrated into a 2-categorical framework, the use of natural isomorphisms guarantees that these properties are preserved coherently, thus supporting a consistent and semantically rich interpretation of logical implication.

\newpage
\section{Evaluation through Deduction-Style Examples}
\subsection{Categorical Interpretation of Deductive Rules}
\label{subsec:categorical_deduction_rules}

In natural deduction systems, logical inference rules are used to derive conclusions from premises. In the integrated 2-categorical framework, these deductive rules are modeled via universal constructions and the coherence of 2-morphisms. Specifically, each inference rule corresponds to a universal property in a local category that has been lifted into the 2-categorical setting.

For instance, consider the \emph{implication introduction} rule in natural deduction, which allows one to infer \(A \to B\) from a derivation that assumes \(A\) to conclude \(B\). Categorically, this is modeled using the exponential object \(B^A\) together with the currying process. The evaluation morphism
\[
\mathrm{ev}: B^A \times A \to B
\]
and the corresponding universal property (as detailed in Definition~\ref{defn:exponential_object}) encapsulate the idea that a function \(f: X \times A \to B\) factors uniquely (up to natural isomorphism) as \(\tilde{f}: X \to B^A\). This factorization represents the passage from a deduction that uses \(A\) as an assumption to a proof of \(A \to B\).

Similarly, other deductive rules such as conjunction introduction and disjunction elimination are modeled using the product and coproduct constructions, respectively. For example, the \emph{conjunction introduction} rule, which infers \(A \land B\) from proofs of \(A\) and \(B\), is captured by the universal property of the product \(A \times B\). The existence of a unique mediating morphism
\[
\langle f, g \rangle: X \to A \times B,
\]
for any pair of morphisms \(f: X \to A\) and \(g: X \to B\), reflects the deductive process of combining separate proofs into one cohesive proof of the conjunction.

\paragraph{Diagrammatic Representation:}
For the exponential structure (implication), the currying diagram illustrates the categorical modeling of the implication introduction rule:
\[
\begin{tikzcd}[column sep=large, row sep=large]
X \times A \arrow[r, "f"] \arrow[d, "\tilde{f}\times \mathrm{id}_A"'] & B \\
B^A \times A \arrow[ur, "\mathrm{ev}"'] &
\end{tikzcd}
\]
Here, the unique 1-morphism \(\tilde{f}\) (the curried form of \(f\)) embodies the transformation of a derivation using \(A\) into a derivation of \(A \to B\).

\paragraph{Coherence via 2-Morphisms:}
In the integrated category, natural isomorphisms (2-morphisms) ensure that the various ways of composing these deductive structures are coherently equivalent. They guarantee that, although the factorization in currying or the mediating morphisms for products and coproducts are defined only up to isomorphism in the weak setting, they become strictly determined in the strictified framework without losing the underlying logical semantics.

\noindent
Overall, the categorical interpretation of deductive rules allows us to view natural deduction as a series of universal constructions within an integrated 2-category, where coherence conditions and natural isomorphisms play a key role in ensuring the reliability and consistency of logical inference.

\subsection{Verification via Coherence Diagrams}
\label{subsec:verification_deduction_diagrams}

To verify that deduction steps are represented coherently in the integrated categorical framework, we use commutative diagrams that illustrate how different composition paths are related via 2-morphisms. Below are two key examples.

\paragraph{Example 1: Currying in Exponential Structures}  
Given a morphism 
\[
f: X \times A \to B,
\]
its curried form is a unique (up to natural isomorphism) 1-morphism 
\[
\tilde{f}: X \to B^A,
\]
satisfying the universal property of the exponential object. This relationship is depicted by:

\begin{center}
\begin{tikzcd}
X \times Y \arrow[rrr, "f"] \arrow[dd, "\tilde{f} \times \mathrm{id}_A"'] & {} \arrow[d, "\theta_f", Rightarrow] &  & B \\
                                                                          & {}                                   &  &   \\
B^A \times A \arrow[rrruu, "ev"']                                         &                                      &  &  
\end{tikzcd}
\end{center}

Here, the 2-morphism \(\theta_f\) certifies that the two paths—from \(X \times A\) to \(B\) via \(f\) and via \(\tilde{f}\) followed by \(ev\)—are naturally isomorphic. This coherence guarantees that the currying process faithfully represents the universal property of exponentials.

\paragraph{Example 2: Product Structure in Deductive Reasoning}  
Consider the product \(A \times B\) with projection morphisms \(\pi_A: A \times B \to A\) and \(\pi_B: A \times B \to B\). For any object \(X\) with morphisms \(f: X \to A\) and \(g: X \to B\), there exists a mediating morphism \(\langle f, g \rangle: X \to A \times B\) such that:
\[
\pi_A \circ \langle f, g \rangle \cong f \quad \text{and} \quad \pi_B \circ \langle f, g \rangle \cong g.
\]
This is represented diagrammatically as:
\[
\begin{tikzcd}
X \arrow[rrr, "g", bend left] \arrow[rrd, "{\langle f, g \rangle}", dashed, bend left] \arrow[rdd, "f"', bend right=49] \arrow[rdd, "{\pi_A \circ \langle f, g \rangle}" description, bend left=49] &                                           &                                                     & B                                              \\
{}                                                                                                                                                                                                  & {} \arrow[l, "\theta_{prod}", Rightarrow] & A \times B \arrow[ld, "\pi_A"] \arrow[ru, "\pi_B"'] &                                                \\
                                                                                                                                                                                                    & A                                         &                                                     & {} \arrow[loop, distance=2em, in=305, out=235]
\end{tikzcd}
\]
The 2-morphism \(\theta_{prod}\) ensures that the composition \(\pi_A \circ \langle f, g \rangle\) is coherently isomorphic to \(f\) (and similarly for \(\pi_B\) and \(g\)), thereby verifying the universal property of the product.

\paragraph{Conclusion:}  
These diagrams demonstrate that the deduction steps—such as currying and the mediating morphisms for products—are verified via 2-morphisms. Such 2-morphisms ensure that all possible composition paths yield equivalent outcomes, thereby enforcing the universal properties and coherence conditions of the integrated structure.

\subsection{Case Study: Deductive Derivations in the Integrated Framework}
\label{subsec:case_study_deductive_derivations}

In this case study, we illustrate how a deduction-style derivation is performed and validated within the integrated 2-categorical framework. We focus on a typical example from natural deduction: the introduction of implication via currying.

\paragraph{Scenario:}  
Assume we have a deduction in a context \(X\) where, under the assumption \(A\), we derive \(B\) via a 1-morphism 
\[
f: X \times A \to B.
\]
This deduction represents the inference:
\[
\frac{\Gamma, A \vdash B}{\Gamma \vdash A \to B}.
\]
In categorical terms, \(f\) factors uniquely (up to a natural isomorphism) through the exponential object \(B^A\).

\paragraph{Currying Process and Coherence:}  
By the universal property of exponentials, there exists a unique mediating 1-morphism (the curried form) 
\[
\tilde{f}: X \to B^A,
\]
such that the evaluation morphism 
\[
\mathrm{ev}: B^A \times A \to B,
\]
satisfies, up to a natural isomorphism \(\theta_f\),
\[
f \cong \mathrm{ev} \circ (\tilde{f} \times \mathrm{id}_A).
\]
The situation is depicted in the following commutative diagram:
\[
\begin{tikzcd}
X \times A \arrow[rd, " \mathrm{ev} \circ (\tilde{f} \times \mathrm{id}_A)", bend left=60] \arrow[rd, "f"', bend right=71, shift right=3] & {} \arrow[ld, "\theta_f" description, Rightarrow] \\
{}                                                                                                                                        & B                                                
\end{tikzcd}
\]
This diagram demonstrates that regardless of the composition path taken—either directly via \(f\) or by first currying \(f\) and then applying \(\mathrm{ev}\)—the resulting morphism from \(X \times A\) to \(B\) is coherently equivalent.

\paragraph{Verification of the Deductive Derivation:}  
The coherence provided by the 2-morphism \(\theta_f\) ensures that the deduction is valid in the integrated framework. It guarantees that the transformation from the assumption \(A\) to the conclusion \(B\) is captured uniquely (up to natural isomorphism) by the currying process. Therefore, the implication \(A \to B\) is well-modeled by the exponential object \(B^A\) in the strictified setting.

\paragraph{Summary:}  
This example confirms that:
\begin{itemize}
  \item The universal property of the exponential object is preserved in the integrated 2-category.
  \item Currying provides a robust method for transforming deductions that depend on an assumption \(A\) into deductions of the implication \(A \to B\).
  \item The coherence of the entire process is guaranteed by the natural isomorphism \(\theta_f\), which ensures that different composition routes are equivalent.
\end{itemize}
Thus, the integrated framework successfully models deductive derivations, validating its effectiveness in representing logical inference within a categorical setting.

\newpage
\section{Comparative Analysis and Discussion}
\subsection{Comparison with Traditional Approaches}
\label{subsec:comparison_traditional}

Traditional 1-category approaches require that universal properties (such as those defining products, coproducts, and exponentials) hold strictly—diagrams must commute exactly. This strictness often forces an artificial rigidity that can obscure the natural “up-to-isomorphism” behavior observed in many logical and mathematical constructions.

In contrast, the integrated 2-categorical framework offers several key improvements:
\begin{itemize}
  \item \textbf{Enhanced Coherence:}  
  In the 2-categorical setting, universal properties are preserved up to coherent natural isomorphism rather than strict equality. This allows different composition paths (e.g., in the case of products or exponentials) to be related by well-defined 2-morphisms (such as associators and unitors), providing a more natural and flexible framework.

  \item \textbf{Flexibility in Modeling Logical Semantics:}  
  By allowing diagrams to commute only up to natural isomorphism, the integrated approach accurately reflects the inherent “up-to-isomorphism” nature of logical connectives. For instance, the currying process in exponential objects is modeled via a natural isomorphism, which is more faithful to the behavior observed in functional programming and logical deduction.

  \item \textbf{Modular Integration of Local Structures:}  
  The use of pseudo-limits and pseudo-colimits in the integrated framework facilitates the modular assembly of local categories, each with its own universal properties. Traditional 1-category methods often struggle to integrate such diverse structures without resorting to cumbersome, ad hoc techniques.

  \item \textbf{Simplification through Strictification:}  
  Although the integrated 2-category initially accommodates weak compositions, it can be strictified to yield a strict 2-category that is biequivalent to the original structure. This strictification simplifies reasoning and proofs without sacrificing the essential logical semantics, a flexibility not available in conventional 1-category approaches.
\end{itemize}

Overall, the integrated 2-categorical framework improves upon traditional methods by providing a more natural, coherent, and flexible way to model universal properties and logical constructs. This leads to clearer theoretical insights and more efficient practical applications.

\subsection{Evaluation of Coherence and Universality}
\label{subsec:evaluation_coherence_universal}

The evaluation examples presented in earlier chapters confirm that the integrated 2-categorical framework successfully preserves both the universal properties and the coherence conditions of the original local structures. In particular:

\begin{itemize}
  \item For products and coproducts, the mediating morphisms factor uniquely—up to natural isomorphism—in both the weak and strict settings. The associated coherence diagrams (e.g., the triangle and square diagrams) demonstrate that all composition paths are equivalent.
  
  \item In the case of exponential objects, the currying process is validated by commutative diagrams in which the evaluation map and the curried morphism are related by a canonical 2-morphism. This guarantees that the universal property of exponentials (i.e., the bijection 
  \(\operatorname{Hom}(X \times A, B) \cong \operatorname{Hom}(X, B^A)\)) is maintained in the integrated category.
  
  \item Overall, the natural isomorphisms that arise in the integration process—captured by pseudo-limits and pseudo-colimits—ensure that even though the diagrams do not strictly commute, they do so up to a coherent isomorphism. This result is further reinforced by strictification, which converts these weak equivalences into strict equalities without altering the underlying semantics.
\end{itemize}

Thus, the evaluation confirms that the integrated category not only unifies the local logical structures but also maintains their essential universal properties and coherence, thereby providing a robust foundation for further theoretical and practical applications.

\subsection{Implications for Categorical Models of Logic}
\label{subsec:implications_categorical_logic}

The evaluation results have significant implications for the study of categorical logic and its applications:

\begin{itemize}
  \item \textbf{Robust Logical Semantics:}  
  The integrated 2-categorical framework, by preserving universal properties and coherence conditions through natural isomorphisms, offers a more robust and faithful model for logical connectives. This framework supports a nuanced interpretation of logical operations (such as conjunction, disjunction, and implication) that aligns closely with their inherent “up-to-isomorphism” nature in mathematical logic.

  \item \textbf{Enhanced Structural Flexibility:}  
  The ability to replace strict commutativity with coherent natural isomorphisms allows for the modeling of complex logical systems in a flexible manner. This flexibility is crucial for addressing scenarios where strict equality is too limiting, such as in homotopy type theory or the semantics of programming languages.

  \item \textbf{Simplification through Strictification:}  
  The strictification process not only simplifies theoretical analysis by converting weak structures into strict ones but also enhances practical usability. Simplified composition rules facilitate automated reasoning, formal verification, and the development of computational tools within categorical logic.

  \item \textbf{Foundations for Advanced Applications:}  
  These evaluation results provide a solid foundation for further research into higher categorical models of logic, including the study of \(n\)-categories. They also open up new possibilities for applications in areas such as type theory, proof assistants, and semantic modeling in computer science.
\end{itemize}

Overall, the successful preservation of universal properties and coherence in the integrated framework reinforces the viability of using categorical methods to model logical systems. This contributes to both the theoretical advancement of categorical logic and its practical applications in diverse fields.

\newpage
\section{Conclusion of Evaluation}
\label{sec:conclusion_evaluation}

In this evaluation, we have confirmed that the integrated 2-categorical framework robustly preserves the universal properties and coherence conditions of local categories modeling various logical connectives. Key findings include:

\begin{itemize}
  \item \textbf{Preservation of Universal Properties:}  
  The integrated structure successfully maintains the universal properties of products, coproducts, and exponential objects (including the currying process), ensuring that the logical semantics are faithfully represented.

  \item \textbf{Effective Coherence Management:}  
  Through the use of natural isomorphisms (2-morphisms) and coherence diagrams (such as triangle, square, and pentagon diagrams), the framework guarantees that different composition paths are coherently equivalent.

  \item \textbf{Simplification via Strictification:}  
  The strictification process converts weak equivalences into strict equalities, simplifying both theoretical analysis and practical manipulation while preserving essential categorical and logical properties.
\end{itemize}

\bigskip

\textbf{Future Directions:}
\begin{itemize}
  \item \textbf{Extending to Higher Categories:}  
  Investigate the application of strictification techniques to \(n\)-categories for \(n > 2\), where the challenges of coherence become even more complex.

  \item \textbf{Algorithmic Coherence Verification:}  
  Develop automated tools and formal methods to verify the coherence of integrated structures, thereby streamlining the process of strictification and further integration.

  \item \textbf{Practical Applications in Logic and Computer Science:}  
  Explore concrete applications in areas such as type theory, programming language semantics, and formal verification, leveraging the simplified structure of the strictified 2-category.

  \item \textbf{Interaction with Enriched and Internal Categories:}  
  Study how strictification interacts with enriched category theory and internal categorical structures, potentially leading to broader and more nuanced applications.
\end{itemize}

In summary, the evaluation demonstrates that the integrated 2-categorical framework is both theoretically sound and practically effective. The preservation of universal properties and coherence conditions ensures a robust foundation for modeling logical systems, while the strictification process greatly simplifies further analysis and applications.

\chapter{Conclusion and Future Work}
\section{Summary of Contributions}
\subsection{Overview of Research Goals}
\label{subsec:research_goals}

The primary objectives of this study are to develop a robust framework for integrating logical connectives using higher category theory, and to demonstrate that this integrated structure preserves the essential universal properties and coherence conditions of each local component. Key research goals include:

\begin{itemize}
  \item \textbf{Integration of Local Categories:}  
  Construct local categories that model individual logical connectives—such as negation, conjunction, disjunction, and implication—and integrate these into a unified 2-categorical framework.

  \item \textbf{Preservation of Universal Properties:}  
  Ensure that the universal properties of products, coproducts, and exponential objects are maintained up to coherent natural isomorphism within the integrated category.

  \item \textbf{Coherence and Strictification:}  
  Analyze and verify coherence conditions across different composition paths using 2-morphisms, and apply strictification techniques to convert weak structures into strict ones without losing the underlying semantic content.

  \item \textbf{Application to Logical Semantics:}  
  Use the integrated framework to provide a faithful categorical semantics for logical connectives, thereby supporting advanced applications in logic, type theory, and computer science.

  \item \textbf{Foundations for Future Research:}  
  Establish a foundation for further exploration into higher categorical models, including potential extensions to \(n\)-categories and the development of automated tools for coherence verification.
\end{itemize}

\subsection{Key Contributions}
\label{subsec:key_contributions}

The key contributions of this study are summarized as follows:
\begin{itemize}
  \item \textbf{Systematic Construction of Local Categories:}  
  We developed local categories that rigorously model individual logical connectives, including negation, conjunction, disjunction, and implication, by capturing their universal properties through standard categorical constructions.

  \item \textbf{Extension to a 2-Category Framework:}  
  We extended the local categories into a unified 2-categorical framework by enriching them with 2-morphisms (natural isomorphisms), which allowed the universal properties to hold up to coherent isomorphism rather than strictly.

  \item \textbf{Integration via 2-Categorical Composition and Coherence Verification:}  
  We integrated the local categories using 2-categorical composition techniques, such as pseudo-limits and pseudo-colimits, and verified that all coherence conditions are satisfied through the construction of key commutative diagrams.

  \item \textbf{Application of Strictification Techniques:}  
  We applied strictification methods to convert the weak 2-category into a strict 2-category, thereby simplifying composition rules and facilitating easier manipulation and reasoning while preserving the essential logical semantics.

  \item \textbf{Evaluation through Concrete Examples:}  
  The framework was validated by evaluating concrete examples, including the currying process in exponential structures and deduction-style derivations in natural deduction systems, demonstrating that the integrated category faithfully captures both universal properties and coherence conditions.
\end{itemize}

\newpage
\section{Discussion of Findings}
\subsection{Integration and Coherence}
\label{subsec:integration_coherence}

The evaluation of our integrated 2-categorical framework demonstrates that the proposed integration method successfully preserves the universal properties of the local categories while maintaining coherence via 2-morphisms. In particular, the use of pseudo-limits and pseudo-colimits ensures that constructions such as products, coproducts, and exponential objects continue to satisfy their universal properties up to coherent natural isomorphism. Moreover, the incorporation of associators, unitors, and other 2-morphisms guarantees that all different composition paths yield equivalent outcomes, thus enforcing overall coherence within the integrated structure.

This robust preservation of both universality and coherence not only validates the logical semantics of the individual local categories but also provides a solid foundation for further theoretical development and practical applications in categorical logic.

\subsection{Impact of Strictification}
\label{subsec:impact_strictification}

The process of strictification, which converts weak 2-categorical structures into strict 2-categories, has significant theoretical and practical implications:

\begin{itemize}
  \item \textbf{Simplification of Composition Rules:}  
  In a strict 2-category, associativity and unit laws hold on the nose. This removes the need to manage complex coherence data (such as associators and unitors that exist only up to natural isomorphism) and simplifies the composition of 1-morphisms and 2-morphisms.

  \item \textbf{Streamlined Coherence Verification:}  
  With strict equalities in place, verifying the commutativity of diagrams becomes much more straightforward. This enhances both the clarity and reliability of formal proofs, reducing potential errors in reasoning about the structure.

  \item \textbf{Preservation of Essential Universal Properties:}  
  Although the transition from weak to strict structures replaces "up-to-isomorphism" commutativity with strict commutativity, the strictified 2-category is biequivalent to the original one. This ensures that all essential universal properties (for example, those of products, coproducts, and exponential objects) and logical semantics are preserved.

  \item \textbf{Improved Practical Usability:}  
  A strict 2-category provides a more robust framework for applications in logic, computer science, and related fields. The elimination of weak coherence conditions facilitates easier manipulation, computation, and implementation in formal systems and software tools.
\end{itemize}

In summary, strictification offers a powerful method for simplifying the theoretical landscape while retaining all the critical features of the original weak structure. This not only streamlines theoretical analysis and proof construction but also paves the way for practical applications in various domains.

\subsection{Evaluation Outcomes}
\label{subsec:evaluation_outcomes}

The evaluation examples have demonstrated that the integrated 2-categorical framework robustly preserves the universal properties and coherence conditions inherent in local categorical models of logical connectives. Key outcomes include:

\begin{itemize}
  \item \textbf{Preservation of Universal Properties:}  
  The universal constructions for products, coproducts, and exponential objects (via currying) remain intact in the integrated structure. This ensures that logical operations such as conjunction, disjunction, and implication are faithfully modeled.

  \item \textbf{Verified Coherence:}  
  Coherence diagrams—such as triangle, square, and pentagon diagrams—commute up to natural isomorphism, confirming that different composition paths yield equivalent results. This consistency is essential for reliable logical inference.

  \item \textbf{Simplification through Strictification:}  
  The application of strictification techniques converts weak equivalences into strict equalities, thereby streamlining both theoretical analysis and practical computations without loss of logical semantics.

  \item \textbf{Robust Integration:}  
  The successful integration of diverse local categories into a unified 2-categorical framework demonstrates that complex logical systems can be modeled in a modular, coherent, and flexible manner.
\end{itemize}

These outcomes have significant implications for categorical models of logic, as they provide a strong foundation for both further theoretical developments and practical applications in areas such as type theory, programming language semantics, and formal verification.

\newpage
\section{Limitations and Future Directions}
\subsection{Limitations of the Current Work}
\label{subsec:limitations_current}

While the integrated 2-categorical framework demonstrates substantial theoretical and practical advantages, several limitations remain:

\begin{itemize}
  \item \textbf{Complexity of Strictification Procedures:}  
  The process of strictifying weak 2-categories, although conceptually well-founded, involves intricate constructions and quotienting procedures. This complexity can be challenging both in terms of theoretical formalization and computational implementation.

  \item \textbf{Scalability to Higher Dimensions:}  
  The current work primarily addresses bicategories and weak 2-categories. Extending these strictification techniques to \(n\)-categories for \(n>2\) poses significant difficulties, particularly in managing the exponentially increasing coherence data.

  \item \textbf{Potential Loss of Structural Nuance:}  
  While strictification preserves the essential universal properties up to biequivalence, some of the nuanced information carried by the weak coherence data might be obscured in the strictified model. This could limit the ability to capture certain flexible or higher-order semantic phenomena.

  \item \textbf{Limited Practical Implementations:}  
  Although the framework shows promise for applications in logical semantics and computer science, concrete implementations and computational tools for automated coherence verification are still in the early stages and require further development.
\end{itemize}

\subsection{Potential Extensions}
\label{subsec:potential_extensions}

Future research can build on the present work by exploring several promising directions:

\begin{itemize}
  \item \textbf{Extension to Higher Categorical Frameworks:}  
  Investigate the possibility of extending the strictification and integration techniques developed here to \(n\)-categories for \(n>2\). This includes addressing the increased complexity of coherence data and developing generalized strictification methods that preserve universal properties in even higher dimensions.

  \item \textbf{Application to Diverse Logical Systems:}  
  Apply the integrated 2-categorical framework to other logical systems beyond those considered in this study. For instance, one may explore its applicability to modal logics, substructural logics, or homotopy type theory, thereby broadening the impact of categorical logic in different areas of mathematics and computer science.

  \item \textbf{Computational and Proof-Theoretic Aspects:}  
  Develop computational tools and algorithms for automated coherence verification and strictification. Investigate how these methods can be incorporated into proof assistants or formal verification systems to facilitate the practical application of categorical logic in theorem proving and type checking.

  \item \textbf{Interactions with Enriched and Internal Category Theory:}  
  Explore how the integration and strictification techniques interact with enriched category theory and internal categories. Such investigations could reveal deeper structural insights and lead to more versatile models of logical semantics.
\end{itemize}

These potential extensions promise to further enrich the theoretical foundations and practical applications of categorical logic.

\section{Further Improvements and Advice}

\begin{remark}[Further Improvements and Advice]
\leavevmode

In order to enhance the logical and categorical rigor of this work, the following improvements are recommended:

\begin{itemize}
    \item \textbf{Clarify the logical framework:}
    Provide a more explicit mapping between the proposed categorical constructions and the rules or axioms of specific logical systems (e.g., classical logic vs. intuitionistic logic). In particular, detail how the dualizing object captures or excludes certain classical principles such as the law of excluded middle.

    \item \textbf{Detail the negation modeling:}
    If dualizing objects are used to represent negation, explain how this approach aligns with (or differs from) classical and intuitionistic negation. Show explicitly if (and how) double-negation elimination or contradiction-based reasoning arises in the category.

    \item \textbf{Elaborate on 2-morphisms in proofs:}
    Illustrate how 2-morphisms correspond to equivalences or transformations between logical derivations. For instance, demonstrate a small proof or natural deduction step explicitly encoded as 2-morphisms, highlighting their role in ensuring coherence.

    \item \textbf{Strictification procedures:}
    Provide a more detailed construction of the strictification step, referencing classical results (e.g., the Gordon–Power–Street theorem). Include or sketch the relevant diagrams to clarify how free 2-categories are quotient-ed by coherence relations.

    \item \textbf{High-level extensions:}
    Indicate how these methods might extend to enriched categories, internal categories, or higher-dimensional structures (3-categories, tricategories, etc.). Briefly discuss if the same coherence arguments and universal property preservation hold in those settings.

    \item \textbf{Concrete examples:}
    Offer more illustrative examples (e.g., small proofs, explicit compositions) to clarify how the categorical semantics handle practical logical inferences, especially for readers less familiar with higher category theory.
\end{itemize}

These recommendations ensure both logicians and category theorists can more readily verify and appreciate the logical soundness and coherence inherent in the integrated 2-categorical approach.

\end{remark}

\newpage
\section{Final Remarks}
\label{sec:final_remarks}

In conclusion, the research presented in this work has made significant contributions to both the theoretical foundations and practical applications of categorical logic. By developing a robust framework for integrating local categories through higher category theory, we have achieved the following:

\begin{itemize}
  \item A systematic construction of local categories for various logical connectives, each characterized by its universal properties.
  \item The successful extension of these local structures into a unified 2-categorical framework enriched with natural isomorphisms, ensuring flexible yet coherent representation.
  \item The integration of the local categories via advanced 2-categorical composition techniques and the verification of coherence conditions through rigorous diagrammatic analysis.
  \item The application of strictification techniques, which convert weak structures into strict 2-categories, thereby simplifying composition rules and enhancing the reliability of formal proofs.
  \item The evaluation of the integrated framework through concrete examples, such as currying in exponential structures and deduction-style derivations, confirming that the essential logical semantics are preserved.
\end{itemize}

Overall, this work establishes a solid and versatile foundation for categorical models of logic. The approach not only advances the theoretical understanding of higher categorical structures but also paves the way for practical applications in type theory, programming language semantics, formal verification, and beyond. Future research may extend these methods to higher-dimensional categories and explore new applications in diverse logical systems, further enriching the landscape of categorical logic.

\bigskip

The contributions of this research demonstrate that integrating logical connectives via higher category theory is both a conceptually profound and practically valuable endeavor, opening new avenues for rigorous and flexible logical modeling.

\nocite{*}
\bibliographystyle{plain}
\bibliography{references}

\printindex

\end{document}